\documentclass[preprint,11pt]{imsart}

\RequirePackage[OT1]{fontenc}
\RequirePackage{amsthm,amsmath,natbib}
\usepackage{hyperref}
\usepackage[ngerman,english]{babel}
\usepackage[ansinew]{inputenc}
\usepackage{a4wide}
\usepackage{calc}
\usepackage{times}
\usepackage{fancyhdr}
\usepackage{amssymb,amsmath,amstext,amsthm, ams fonts, amsfonts}
\usepackage{colortbl,color,graphics,graphicx,epsfig}
\usepackage{framed}
\definecolor{gray}{rgb}{0.8,0.8,0.8}
\usepackage{dsfont}
\usepackage{bbm}
\usepackage{nicefrac}
\usepackage{comment}
\newcommand{\KLEINO}{{\scriptstyle{\mathcal{O}}}}
\startlocaldefs
\newcommand{\KK}{L}

\renewcommand{\P}{\mathbb{P}}
\newcommand{\F}{\mathcal{F}}
\newcommand{\td}{\theta}
\newcommand{\ld}{\mathfrak{l}}

\newcommand{\E}{\mathbb{E}}
\newcommand{\N}{\mathds{N}}
\newcommand{\R}{\mathds{R}}
\newcommand{\aalpha}{\mathfrak{a}}
\newcommand{\1}{\mathbbm{1}}
\DeclareMathAccent{\verywidehat}{\mathord}{largesymbols}{'144}
\newcommand{\var}{\mathbb{V}\hspace*{-0.05cm}\textnormal{a\hspace*{0.02cm}r}}

\newcommand{\Q}{\mathbb{Q}}

\newcommand{\dd}{\Delta^n_i}

\newcommand{\pn}{\stackrel{\P}{\longrightarrow}}
\newcommand{\weak}{\stackrel{w}{\longrightarrow}}

\newcommand{\WU}{\widetilde{U}}
\newcommand{\WV}{\widetilde{V}}

\newcommand{\WY}{\widetilde{Y}}

\DeclareMathOperator{\Var}{Var}

\newcommand{\xiv}{\mbox{\boldmath$\xi$}}

\newcommand{\chiv}{\mbox{\boldmath$\chi$}}

\newtheorem{prop}{Proposition}[section]
\newtheorem{ass}[prop]{Assumption}
\newtheorem{testingproblem}[prop]{Testing problem}
\newtheorem{exa}[prop]{Example}

\newtheorem{lem}[prop]{Lemma}
\newtheorem{theo}[prop]{Theorem}
\newtheorem{rem}[prop]{Remark}
\newtheorem{alg}[prop]{Algorithm}

\allowdisplaybreaks[4]

\begin{document}
\begin{frontmatter}
\title{Nonparametric change-point analysis of volatility}
\runtitle{Nonparametric change-point analysis of volatility}

\begin{aug}
\author{\fnms{Markus} \snm{Bibinger}\thanksref{t2,t3}\ead[label=e1]{bibinger@uni-mannheim.de}},
\author{\fnms{Moritz} \snm{Jirak}\thanksref{t2,t3}\ead[label=e2]{jirak@math.hu-berlin.de}},
\author{\fnms{Mathias} \snm{Vetter}\thanksref{t3}\ead[label=e3]{vetter@math.uni-kiel.de}}

\thankstext{t2}{Financial support from the Deutsche Forschungsgemeinschaft via SFB 649 {\it \"Okonomisches Risiko} and FOR 1735 {\it Structural Inference in Statistics: Adaptation and Efficiency} is gratefully acknowledged.}
\thankstext{t3}{We thank Marc Hoffmann for helpful remarks on testing hypotheses of smoothness classes. We also thank two anonymous referees for valuable
comments on an earlier version.}
\address{Markus Bibinger,\\ Department of Economics,\\ Mannheim University,\\ L7,3-5, 68131 Mannheim, Germany,\\ \printead{e1}}
\address{Moritz Jirak,\\
Institut f\"ur Mathematik,\\ Humboldt-Universit\"at zu Berlin,\\
Unter den Linden 6,\\ 10099 Berlin,
Germany,\\
\printead{e2}}
\address{Mathias Vetter,\\ Mathematisches Seminar,\\ Christian-Albrechts-Universität zu Kiel,\\ Ludewig-Meyn-Straße 4,\\ 24118 Kiel, Germany,\\  
\printead{e3}}

\runauthor{Bibinger, Jirak, Vetter}

\affiliation{Mannheim University, Humboldt University of Berlin and Christian-Albrechts University of Kiel}

\end{aug}

\begin{abstract}
This work develops change-point methods for statistics of high-frequency data. The main interest is in the volatility of an It\^{o} semi-martingale, the latter being discretely observed over a fixed time horizon. We construct a minimax-optimal test to discriminate continuous paths from paths comprising volatility jumps. This is embedded into a more general theory to infer the smoothness of volatilities. In a high-frequency framework we prove weak convergence of the test statistic under the hypothesis to an extreme value distribution. Moreover, we develop methods to infer changes in the Hurst parameter of fractional volatility processes. A simulation study demonstrates the practical value in finite-sample applications.
\end{abstract}


\end{frontmatter}

\section{Introduction\label{sec:1}}
Change-point theory classically focuses on detecting one or several structural breaks in the trend of time series. Statistical methods to infer change-points have a long and rich history, dating back to the pioneering work of \cite{page}. Prominent approaches as e.g.\, by \cite{hinkley}, \cite{pettitt}, \cite{andrews} or \cite{baiperron}, among many others, provide statistical tests for the hypothesis of no change-point against the alternative that changes occur. Moreover, they allow for localization of change-points (estimation) and confidence intervals. Change-point methods usually rely on maximum statistics and exploit limit theorems from extreme value theory; see \cite{horvath} for an overview. Less focus has been laid on discriminating jumps from continuous motion in a nonparametric framework. Important exceptions are \cite{mueller}, \cite{stadtmueller}, \cite{spokoiny} and \cite{wuzhao2007} in the framework of nonparametric regression analysis. The latter serves as an important point of orientation for this work. 

Statistics of high-frequency data is concerned with discretizations of continuous-time stochastic processes, most generally It\^o semi-martingales. The continuous part of an It\^{o} semi-martingale is of the form
\begin{align}\label{sm}
X_t = X_0 + \int_0^t a_s \,ds + \int_0^t \sigma_s \,dW_s\,,
\end{align}
defined on a filtered probability space $(\Omega,\mathcal{F},(\mathcal{F}_t),\P)$ with a standard $(\mathcal{F}_t)$-Brownian motion $W$ and 
adapted drift and volatility processes $a$ and $\sigma$. One key topic is statistical inference on the volatility under high-frequency asymptotics when the mesh of a discretization on a fixed time horizon tends to zero. There is a vast body of works related to this problem and its economic implications; see e.g.\, \cite{andersenbollerslev98}, \cite{inference} and \cite{jacodrosenbaum}, among many others. Statistics for a discretized continuous-time martingale is closely related to Gaussian calculus as highlighted by \cite{gaussian} what is also at the heart of our analysis. Many contributions evolve around the question if jumps are present in the It\^o semi-martingale modeling the log-price of a financial asset; see \cite{sahaliajacod} for a statistical test.

A more involved problem which is of key interest for economics and finance is to infer the smoothness of the underlying stochastic volatility process and to check whether volatility jumps occur. In particular, inference on volatility jumps allows to investigate the impact of certain news arrivals on financial risk. A first empirical study by \cite{voljumps} indicates that volatility jumps can occur but, due to the lack of statistical methods, has been based on direct observations of the VIX, the most prominent available volatility index. Further contributions consider joint price-volatility jumps. \cite{jacodtodorov} have designed a test to decide from high-frequency observations if contemporaneous jumps of an It\^o semi-martingale and its volatility process have taken place at least once over some fixed time interval. These methods do not generalize to test directly for volatility jumps without restricting to a finite set of large price adjustments first. One main profit from our change-point analysis of high-frequency data is a general test for volatility jumps. Moreover, results on estimation of the time of a volatility jump are provided.

\begin{figure}[t]\begin{framed}
\includegraphics[width=7.0cm]{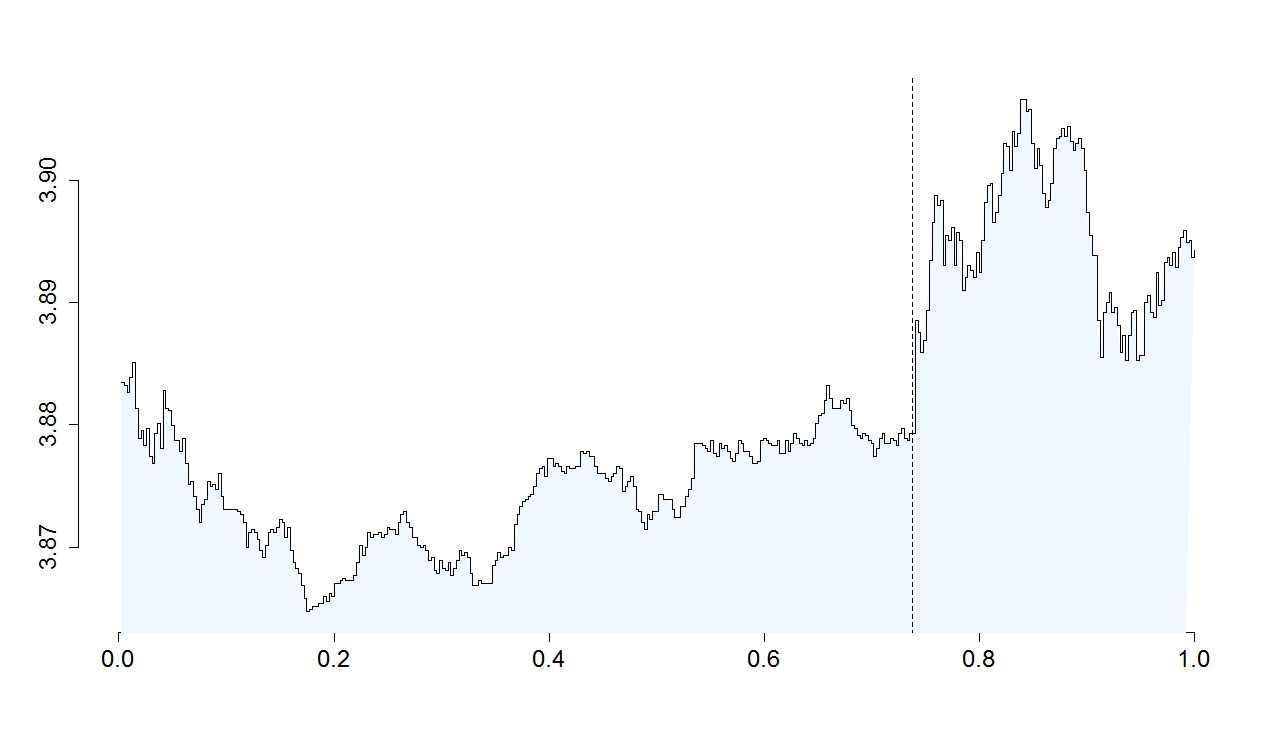}\includegraphics[width=7.0cm]{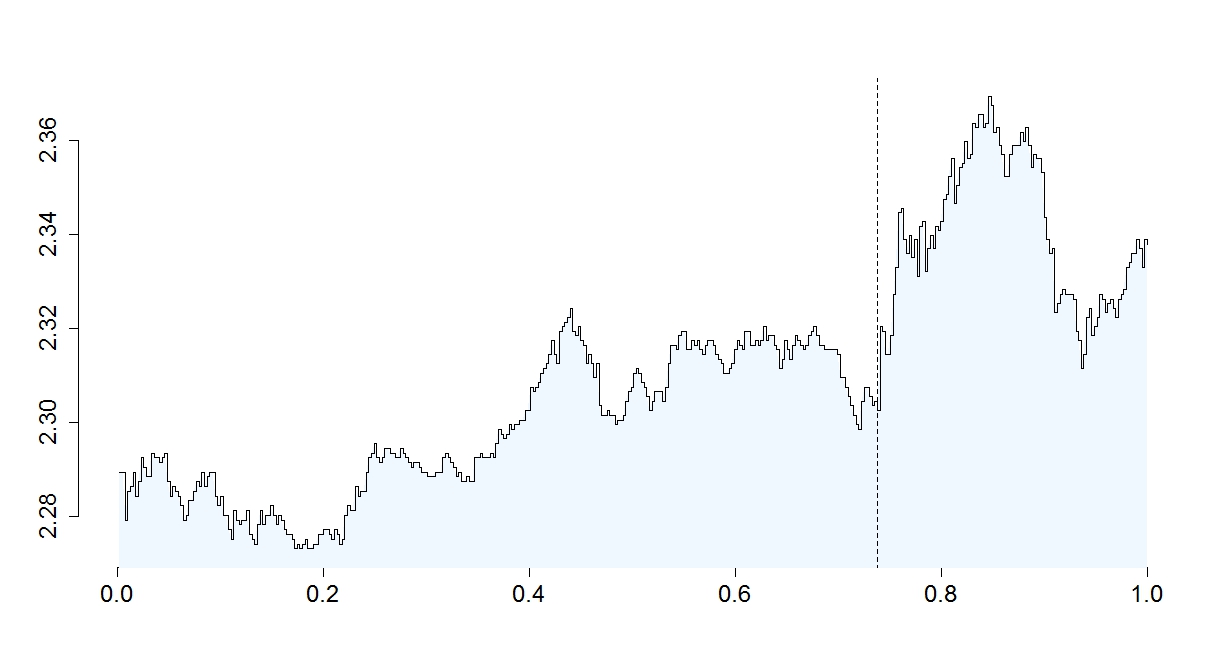}
\includegraphics[width=7.0cm]{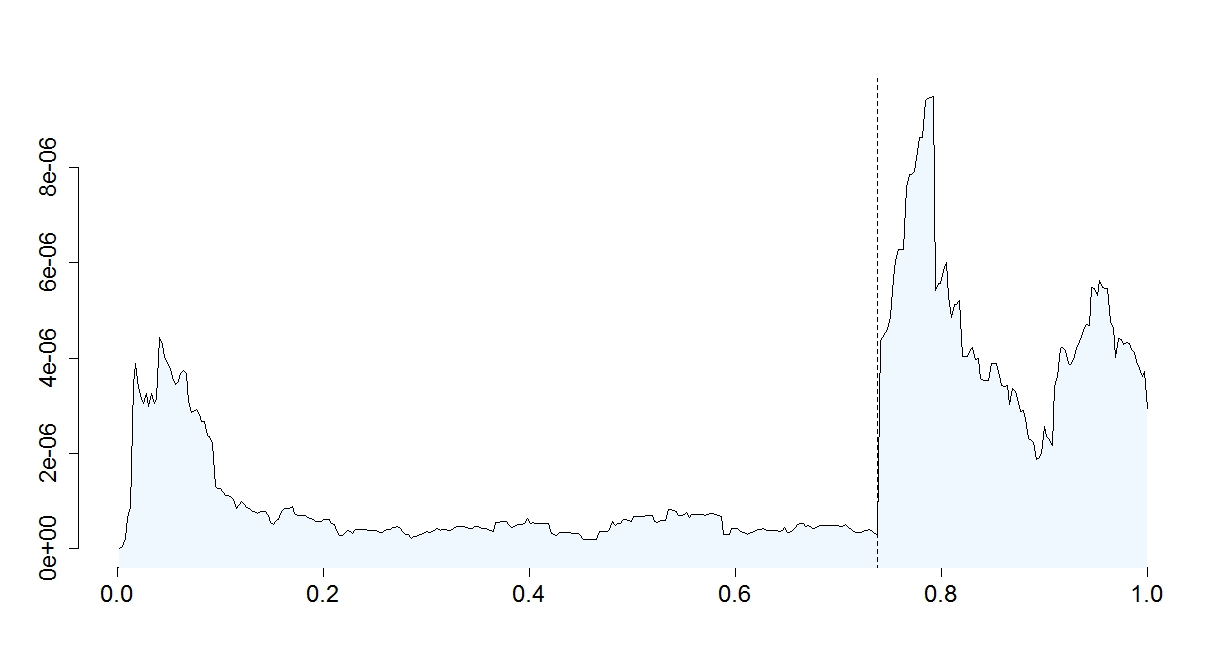}\includegraphics[width=7.0cm]{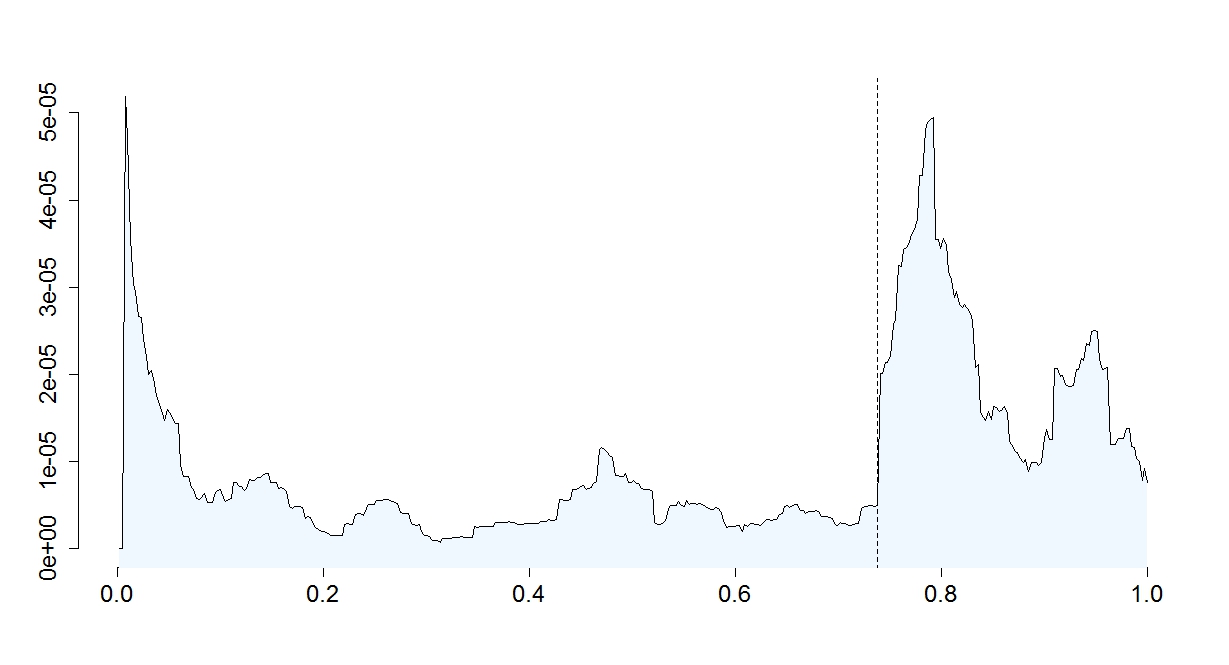}
\end{framed}
\caption{\label{Fig:Intro}Log-price intra-day evolutions (top) and estimated spot squared volatilities (bottom) for MMM (left) and GE (right) on March 18th, 2009.}
\end{figure}

As an example, we illustrate in Figure \ref{Fig:Intro} the evolution of log-prices of two blue-chip stocks, 3M and GE, over the NASDAQ intra-day trading period (6.5\,h rescaled to the unit interval) on March 18th, 2009. We consider one minute returns from executed trades\footnote{reconstructed from the order book using LOBSTER, \url{https://lobster.wiwi.hu-berlin.de/}} to ensure the semi-martingale model is adequate and limit a manipulation by market microstructure frictions. Available tests and criteria do not identify price adjustments so large to be ascribed to jumps such that the test by \cite{jacodtodorov} is not applicable. It seems as if a common source of news drives price dynamics at the end of that day concertedly. The picture becomes much clearer when focusing on the estimated spot squared volatilities in Figure \ref{Fig:Intro}, for which we average at each time point the previous 20 rescaled squared returns. This example suggests that volatility dynamics vary over time. Here, the volatilities of both assets sky-rocket at exactly the same time. This common volatility jump coincides in time with a press release at 02:15 p.m. EST subsequent to a meeting of the Federal Open Market Committee. The time is marked in Figure \ref{Fig:Intro} by the dashed lines. In light of increasing economic slack, the FOMC announced ``to employ all available tools to promote economic recovery and to preserve price stability''\footnote{source: \url{www.federalreserve.gov/monetarypolicy/fomcminutes20090318.htm}}, including a guarantee for an exceptionally low level of the federal funds rate for an extended period and a considerable increase of the size of the Federal Reserve's balance sheet. The mathematical concepts developed in this work provide a novel device for assessing volatility dynamics and jumps.

Change-point methods for volatility in a time-series environment, which is quite different to our high-frequency semi-martingale setting, have been discussed by \cite{spok}.
Quasi-likelihood estimation of a change-point in a diffusion parameter in a high-frequency setting has been considered by \cite{iacusyoshida}, pointing out already one very useful bridge between change-point theory and high-frequency statistics. Our main focus is on testing for the presence of changes in a general setup exploiting localization techniques. Beyond the analysis of possible jumps of volatility there is great interest in the smoothness regularity of volatilities; see e.g.\, \cite{gatheral} for a recent work, not least because of its crucial role for setting up volatility models.

We focus on volatilities which are almost surely locally bounded and strictly positive adapted processes.
For our testing problem we consider classes of squared volatilities
\begin{align}\label{hyposet}\Sigma(\aalpha,\KK_n)=\Big\{(\sigma_t^2)_{t\in[0,1]}\big|\sup_{s,t\in[0,1],|s-t|<\delta}\big|\sigma_t^2-\sigma_s^2\big|\le \KK_n\delta^{\aalpha}~\Big\}\,,\end{align}
for an appropriate sequence $\KK_n$ converging to infinity; cf.\,Assumption \ref{assVola} for a precise statement about the conditions under the null hypothesis. The regularity exponent $\aalpha>0$ is the key parameter to describe the null hypothesis $H_0$. We may now more formally ask the following questions:
\begin{description}
\item[(i)] Is there a jump in the volatility, i.e.\, $\Delta \sigma_{\td}^2=\big(\sigma^2_{\td}-\lim_{s\uparrow\theta}\sigma^2_{s}\big) > 0$ for some $\td \in (0,1)$?
\item[(ii)] Does volatility get rougher in the sense of a regularity exponent $\aalpha' < \aalpha$ on $(\theta,1]$?
\end{description}

Question (i) poses a \textit{local} problem, whereas question (ii) entails a \textit{local} or \textit{global} problem. In particular, one single jump is a discontinuity point which is not informative about the volatility's smoothness elsewhere. More general than jumps, our theory to address the local problem includes abrupt yet continuous adjustments of the volatility over a short time period. This is an example of a local change of regularity, where $\aalpha$ drops to $0<\aalpha' < \aalpha$ for a short period of time before attaining its original value again. In this framework, the case of a jump corresponds to $\aalpha' = 0$. 

On the other hand, at some time price fluctuations may considerably increase or decrease and this rougher behaviour persists permanently over the remaining time interval. Then, we witness a \textit{global} change. As a key example, our methods devoted to the global problem allow to infer changes in the Hurst parameter in a fractional volatility model. Naturally, as presumably in Figure \ref{Fig:Intro}, both local and global events may occur simultaneously. This work presents methods to test local and global alternatives, both relying on different foundations. A desirable property for a global approach is robustness with respect to local changes which is crucial to distinguish the two problems. This means in particular that a test statistic to decide the question of global changes should not be affected by a fixed number of jumps, which may be interpreted as a part of the hypothesis of no global change. The other way round, an approach to test for local changes may very well be affected by a global change as well, for instance at the very time of change or by an exceptionally large fluctuation. 

We consider minimax-optimal testing and estimation in both situations covering broad classes of volatility processes. For a conceptual introduction to minimax-optimal tests, let us focus first on discriminating smooth volatilities in the sense of \eqref{hyposet} from volatilities with at least one jump. From a statistical perspective, the key question is which sizes of volatility jumps can be detected. For example, it is clear that we cannot detect jumps of arbitrarily small size. Loosely speaking, if we say that `no jump' is our null hypothesis $H_0$ and `there is a jump' is our alternative $H_1$, then we face the problem of \textit{distinguishability between $H_0$ and $H_1$}. The minimum size $b_n$ of a jump $\Delta \sigma_{\td}^2$, such that we are still able to uniformly control the type I and type II errors, is called \textit{detection boundary}. If we are interested to test for the presence of jumps, we are thus led to consider for $\td\in(0,1)$ alternatives of the form
\begin{align}\label{jumpalternatives}
\mathcal{S}^{J}_{\td}(\aalpha,b_n,\KK_n)=\bigg\{(\sigma^2_t)_{t\in[0,1]}\big| (\sigma^2_t - \Delta \sigma^2_t)_{t \in [0,1]} \in \Sigma(\aalpha,\KK_n) \, ; \,  |\Delta \sigma^2_{\td}|\ge b_n\bigg\}\,
\end{align}
with a decreasing sequence $b_n$. We then address the testing problem
\begin{align}\label{hypo}
H_0:(\sigma_t^2(\omega))_{t\in[0,1]}\in\Sigma(\aalpha,\KK_n)\quad vs. \quad H_1: \exists \,\td \in(0,1)\,\text{with}\,(\sigma_t^2(\omega))_{t\in[0,1]}\in\mathcal{S}_{\td}^{J}(\aalpha,b_n,\KK_n)\,.
\end{align}
In this context $\td$ is commonly referred to as a \textit{change-point}. The test alternative means that we demand at least one jump but do not exclude multiple jumps. The dependence on $\omega$ in \eqref{hypo} is natural in the definition of the hypotheses, as different realizations might lead to different paths on $[0,1]$.

For the testing problem \eqref{hypo}, we establish the minimax-optimal rate of convergence under high-frequency asymptotics. We follow the notion of minimax-optimality of statistical tests from the seminal contributions of \cite{Ingster1993}. For tests $\psi$ that map a sample ${\bf X}_n$ to zero or one, where $\psi$ accepts the null hypothesis $H_0$ if $\psi = 0$ and rejects if $\psi = 1$, we consider the maximal type I error $\alpha_{\psi}\bigl(\aalpha\bigr) = \sup_{\sigma^2\in\Sigma(\aalpha,\KK_n)} \P_{\sigma}\bigl(\psi = 1\bigr)$ and the maximal type II error\\ $\beta_{\psi}\bigl(\aalpha,b_n\bigr) = \sup_{\theta\in(0,1)}\sup_{\sigma^2 \in \mathcal{S}^J_{\td}(\aalpha,b_n,\KK_n)}\P_{\sigma}\bigl(\psi = 0\bigr)$ and define the global testing error as
\begin{align}\label{testerror}
\gamma_{\psi}\bigl(\aalpha,b_n \bigr) = \alpha_{\psi}\bigl(\aalpha\bigr) + \beta_{\psi}\bigl(\aalpha,b_n\bigr).
\end{align}
The primary interest now lies on tests that minimize $\gamma_{\psi}\bigl(\aalpha,b_n\bigr)$, given the boundary $b_n$. We aim to find sequences of tests $\psi_n$ and boundaries $b_n$ with the property that
\begin{align*}
\gamma_{\psi_n}\bigl(\aalpha, b_n \bigr) \to 0 \quad \text{as ~ $n \to \infty$.}
\end{align*}
The smaller $b_n>0$, the harder it is for a test to control the global testing error, i.e.\;to distinguish between $H_0$ and $H_1$. It is thus natural to pose the question given $\aalpha$, what is the minimal size of $ b_n>0$ such that
\begin{align}\label{defn_minimax_optimal_test}
\lim_{n\rightarrow\infty} \inf_{\psi} \gamma_{\psi}\bigl(\aalpha,b_n \bigr) = 0
\end{align}
holds? The optimal $b_n^{opt}$ is called \textit{minimax distinguishable boundary}, and a sequence of tests $\psi_n$ that satisfies \eqref{defn_minimax_optimal_test} for all  $b_n \geq b_n^{\text{opt}}$ \textit{minimax-optimal}.
If $\KK$ in \eqref{hyposet} is constant, we prove that $b_n\propto (n/\log (n))^{\frac{-\aalpha}{2\aalpha+1}}$ constitutes the minimax distinguishable boundary for testing \eqref{hypo} and our constructed test is eligible to attain minimax-optimality. If $\KK_n$ is indeed a sequence, the rate only slightly changes; see Section \ref{sec:4.1} for precise results.

For the lower bound proof we simplify the problem by information-theoretic reductions passing to more informative sub-classes of the parameter space. The lower bound established for the sub-class then serves a fortiori as a lower bound in the more general and less informative model. After gradually transforming the problem by showing strong Le Cam equivalences of the considered sub-experiment to more common situations with i.i.d.\,chi-square and Gaussian variables, the lower bound is proved by classical arguments based on the theory in \cite{ingster}.

The paper is organized as follows: Section \ref{sec:2} serves as an illustration for the benefit of cusum-based statistics in the simple, yet important model of a continuous It\^o semimartingale with constant volatility. This illuminates the connection of classical change-point methods and high-frequency statistics. More involved, but also more important in practice is the case where the volatility is both time-varying and random. 
Section \ref{sec:3} is devoted to this nonparametric local problem. 
As the volatility process is latent, which requires estimation based on smoothed squared increments of the semi-martingale, this poses an intricate statistical problem which to the best of the authors' knowledge had not been addressed so far. We establish a consistent test and derive a limit theorem under the hypothesis. The asymptotic analysis utilizes nonparametric change-point theory, stochastic calculus and bounds on the approximation error in the invariance principle. Our test allows to distinguish paths with jumps from continuous paths under remarkably general smoothness assumptions on the hypothesis. In Section \ref{sec:3.2} we discuss the situation in which the underlying It\^o semimartingale might have jumps as well. Section \ref{sec:4.1} provides the theory on minimax-optimality with the lower bound, while Section \ref{sec:4.2} deals with the estimation of the location of the change in volatility under the alternative. Finally, a minimax-optimal nonparametric test for the global problem is established in Section \ref{sec:5}. In practice, it is attractive to interconnect both methods which complement each other. A simulation study that investigates the finite-sample performance of the proposed methods and discusses some practical issues can be found in Section \ref{sec:6}. All proofs are postponed to the Appendix.

\section{Change-points in a parametric volatility model\label{sec:2}}
Arguably, the simplest model of a continuous-time It\^{o} diffusion $X$ is the case of no drift and a constant volatility, such that $X$ is given by
\begin{align}X_t=X_0+\int_0^t\sigma \,dW_s\,,\end{align}
where $W$ denotes a standard Brownian motion. Throughout this work, the underlying process $X$ is recorded at discrete regular times ${i\Delta_n}$ with a mesh $\Delta_n \to 0$. To keep the notation uncluttered, we assume to be on the fixed time interval $[0,1]$ and set $n = \Delta_n^{-1}\in\N$, so that we have observations $X_{i\Delta_n}, i=0,\ldots,n$.

Inference on the squared volatility $\sigma^2$ is usually based on increments $\dd X=X_{i\Delta_n}-X_{(i-1)\Delta_n}$. In case one is interested in changes in the volatility, a natural quantity to discuss is the cusum statistic which reads
\begin{align}\label{cusum}S_{n,m}=\frac{1}{\sqrt{n}}\sum_{i=1}^m\Big(n\big(\dd X\big)^2-\sum_{j=1}^n\big(\Delta_j^n X\big)^2\Big)\,,m\in\{1,\ldots,n\}\,.\end{align}

In order to derive the asymptotics of the cusum statistic, recall the functional (stable) central limit theorem for the realized volatility from observations of a continuous  It\^{o} semi-martingale \eqref{sm} by \cite{jacod1}. Under mild assumptions, we have
\begin{align}\label{cltjacod}\sqrt{n}\Bigg(\sum_{i=1}^{\lfloor nt\rfloor }\big(\dd X\big)^2-\int_0^t\sigma_s^2\,ds\Bigg)\rightarrow \int_0^t\sqrt{2}\sigma_s^2\, dB_s~,t\in[0,1]\,,\end{align}
as $n\rightarrow \infty$ weakly in the Skorokhod space with a standard \begin{figure}[t]
\fbox{
\includegraphics[width=7.5cm]{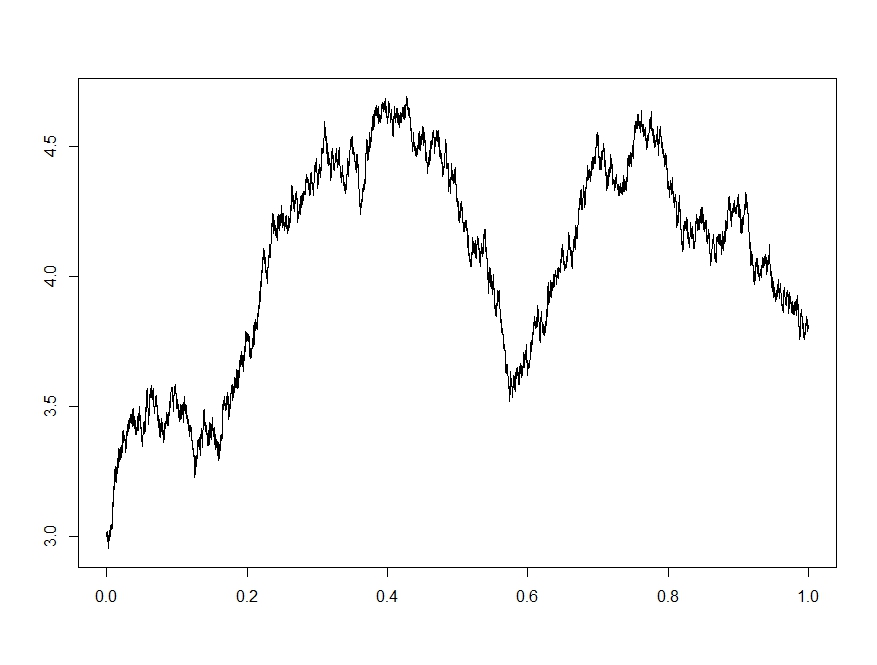}~\includegraphics[width=7.5cm]{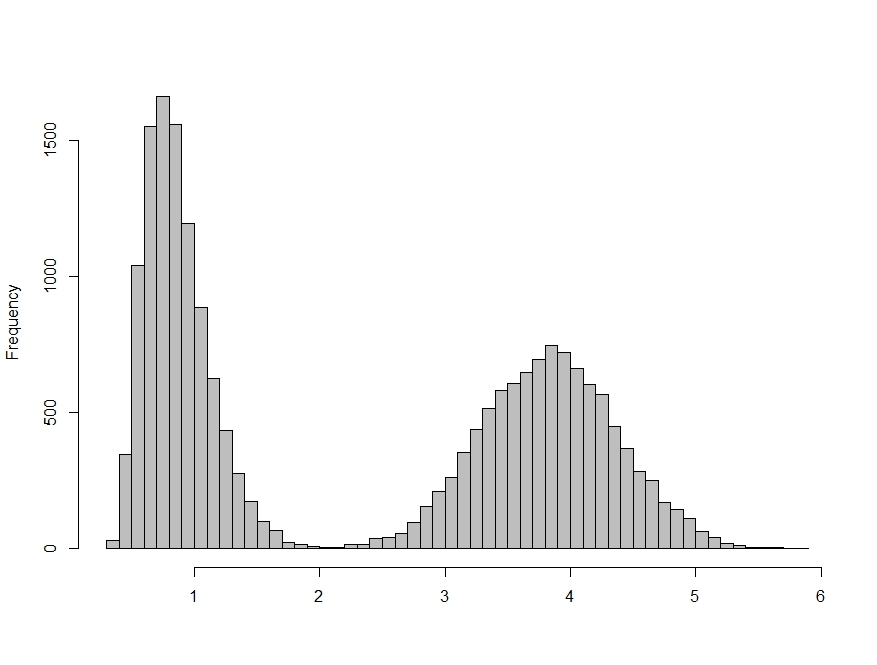}}
\caption{\label{Fig:1}Left: One realized path of $X$ with structural break in volatility at $t=1/2$. Right: Empirical results for the test statistic from 10000 iterations under both alternative and hypothesis.}\end{figure} Brownian motion $B$ independent of $W$. In particular, if $\sigma_s=\sigma$ is constant, this result directly implies
\begin{align}S_{n,{\lfloor nt\rfloor}}\rightarrow \gamma\big(B_t-tB_1\big)\,,\end{align}
with $\gamma^2=\lim_{n\rightarrow\infty} n\,\var\big(\sum_{i=1}^n\big(\dd X\big)^2\big)=2\sigma^4$, which coincides with a standard cusum limit theorem in the vein of \cite{phillips}. The quarticity estimator by \cite{bn}, $\hat\gamma^2=(2n/3)\sum_{i=1}^n\big(\dd X\big)^4$, may be used to obtain a self-normalizing version:
\begin{align}\Bigg(\frac{2n}{3}\sum_{i=1}^n\big(\dd X\big)^4\Bigg)^{-1/2}\,S_{n,\lfloor nt\rfloor}\rightarrow B_t-tB_1\,,\end{align}
where the limit process is a standard Brownian bridge. Testing for jumps (resp.\,structural breaks, change-points) of the volatility is then pursued based on
\begin{align}\label{testks}T_n=\sup_{m=1,\ldots,n}\Big|\big(\hat\gamma^2\big)^{-1/2}\,S_{n,m}\Big|\,,\end{align}
as the test statistic which (under the null that the volatility is constant) tends as $n\rightarrow\infty$ to a Kolmogorov-Smirnov law; see \cite{eval}. Under the alternative $T_n$ diverges almost surely.

Figure \ref{Fig:1} shows an example in which we observe $n=10000$ values of a standard Brownian motion under the hypothesis, while under the alternative the volatility jumps in $t=1/2$ from 1 to 1.1. Out of 10000 Monte Carlo iterations for hypothesis and alternative, only 21 realizations of \eqref{testks} under the hypothesis are larger than the minimum under the alternative. The other way round, in 11 iterations the values under the alternative fall below the maximum of the generated values from the hypothesis. The cusum approach hence clearly allows to separate hypothesis and alternative here, even for the relatively small volatility jump which is not readily identifiable from the path of $X$ in Figure \ref{Fig:1}.

This is illustrated within the histogram in Figure \ref{Fig:1}. The left part stems from realizations under the hypothesis which closely track the asymptotic Kolmogorov-Smirnov law. The right part instead is due to realizations under the alternative. For larger volatility jumps the right part moves further to the right such that the two distributions separate even more clearly.
This test based on \eqref{testks} of Kolmogorov-Smirnov type permits to test the hypothesis of a constant volatility against structural breaks in an efficient way.

Beyond this bridging of classical change-point analysis and structural breaks in a parametric volatility model, our main focus in the sequel is nonparametric: to distinguish volatility jumps from a continuous motion of volatility or to identify changes in the regularity exponent.

\section{Nonparametric change-point test for the volatility against jump alternatives\label{sec:3}}
\subsection{Construction and limit behavior under the hypothesis\label{sec:3.1}}
Suppose we observe a continuous It\^{o} semi-martingale \eqref{sm} at the regular times $i\Delta_n,i=0,\ldots,n \in\N$. In this setting, we want to construct a test for \eqref{hypo}. 
With the volatility process being time-varying, it becomes apparent from \eqref{cltjacod} that the test statistic \eqref{testks} is not suitable to test $H_0$ against $H_1$. Our core idea is to utilize local two-sample $t$-tests over asymptotically small time blocks instead. As a first test statistic, we consider
\begin{align} \label{statV}
V_n =  \max_{i=0,\ldots,\lfloor n/k_n \rfloor -2} \vert RV_{n,i}/RV_{n,i+1} - 1 \vert,
\end{align}
where $k_n \to \infty$ is an auxiliary sequence of integers depending on $n$ and
\begin{align} \label{statX}
RV_{n,i} = \frac{n}{k_n} \sum_{j=1}^{k_n} (\Delta_{ik_n + j}^n X)^2\,,i=0,\ldots,\lfloor n/k_n\rfloor -1\,,
\end{align}
is a rescaled local version of realized volatility over blocks of the partition \([ik_n\Delta_n,(i+1)k_n\Delta_n]\). The $RV_{n,i}$ estimate a block-wise constant proxy of the spot volatility $\sigma_{ik_n\Delta_n}$ on the respective blocks.
Asymptotic properties of $RV_{n,i}$ were e.g.\, derived in \cite{alvaetal2012}. A large distance between $RV_{n,i}$ and $RV_{n,i+1}$ suggests the presence of a jump or unsmooth breaks in the volatility close to time $ik_n \Delta_n$. In order to obtain normalized statistics we work with ratios instead of differences. Identifying breaks, in particular jumps, from large deviations between the ratios of two successive local volatility estimators and one, $V_n$ appears to be a reasonable test statistic for our problem.

Our second test statistic is of the same nature as \eqref{statV}, but instead of non-overlap\-ping blocks it takes into account all overlapping blocks of $k_n$ increments:
\begin{align} \label{statV2}
V_n^* = \max_{i=k_n,\ldots,n-k_n} \Bigg\vert  \frac{ \frac{n}{k_n}\sum_{j=i-k_n+1}^{i}(\Delta_{j}^n X)^2}{\frac{n}{k_n}\sum_{j=i+1}^{i+k_n} (\Delta_{j}^n X)^2}-1\Bigg\vert\,.
\end{align}
In comparison to nonparametric change-point approaches like the one by \cite{wuzhao2007}, both statistics \eqref{statV} and \eqref{statV2} are based on ratios rather than differences. This makes sense intuitively, since we are not dealing with the typical additive error structure of time series models. In our setting, we have e.g.\, $n(\Delta_i^n X)^2\approx \sigma_{i\Delta_n}^2\chi_i^2, i=1,\ldots,n$, with i.i.d. $\chi_1^2$-distributed random variables $\chi_i^2$, so that the volatility $\sigma$ plays the role of a multiplicative error. Therefore, by computing ratios first, we basically deal with a maximum of identically distributed variables in the asymptotics. This is of key importance to obtain a distribution free limit under the hypothesis.

In order to discuss the asymptotics of $V_n$ and $V_n^*$ under the null hypothesis we need a couple of additional assumptions, all of which are rather mild and are covered by a variety of stochastic volatility models.

\begin{ass} \label{assVola}
The following assumptions on the processes $a$ and $\sigma$ are in order:
\begin{enumerate}
	\item[(1)] $a$ and $\sigma$ are locally bounded processes.
	\item[(2)] $\sigma$ is almost surely strictly positive, i.e.\, $\inf_{t\in[0,1]}\sigma_t^2\ge \sigma_-^2>0$.
	\item[(3)] On the hypothesis, $\Omega^c\subset\Omega$, 
	the modulus of continuity
\[
w_\delta(\sigma)_t = \sup_{s,r \leq t} \{ |\sigma_s - \sigma_r| : |s-r| < \delta \}
\]
is locally bounded in the sense that there exists $\aalpha > 0$ and a sequence of stopping times $T_n \to \infty$ such that $w_\delta(\sigma)_{(T_n\wedge 1)} \leq \KK_n \delta^\aalpha$, for some $\aalpha > 0$ and some (a.s.\ finite) random variables $\KK_n$. 	
\end{enumerate}
\end{ass}
In particular this implies $\sigma_t^2\in\Sigma(\aalpha,L_n)$ for some $n$ almost surely on $\Omega^c$. To consider sequences $L_n$ becomes important when developing lower bounds, see the first paragraph of Section \ref{sec:4.1} for a detailed explanation.
We choose the sequence $k_n\rightarrow\infty$, as $n\rightarrow\infty$, such that the following growth condition holds:
\begin{align}\label{assk}k_n^{-1} \Delta_n^{-\epsilon}+\sqrt{k_n} (k_n \Delta_n)^\aalpha \sqrt{\log(n)}&\to 0\,,\end{align}
for some $\epsilon>0$ and with $\aalpha >0$ from Assumption \ref{assVola} (3).

There are two conditions contained in \eqref{assk}. First, $k_n\rightarrow\infty $ faster than some power of $n$ which is a mild lower bound on the growth of $k_n$ as $n\rightarrow\infty$. This ensures consistency of the estimates \eqref{statX}. The second condition gives an upper bound related to the continuity of $\sigma$. Naturally, the smaller $\aalpha$ (and the less smooth $\sigma$), the smaller we have to choose the size of the blocks over which we estimate $\sigma$.
\begin{theo} \label{thm1}
Set $m_n = \lfloor n/k_n\rfloor $ and $\gamma_{m_n} = [4 \log(m_n) - 2 \log(\log (m_n))]^{1/2}$. If Assumption \ref{assVola} holds and $k_n$ satisfies condition \eqref{assk}, then we have on $\Omega^c$ (under $H_0$)
\begin{align}\label{thm1lt}
\sqrt{\log(m_n)} \big(\big(k_n^{1/2}/\sqrt{2}\big) V_{n} - \gamma_{m_n}\big) \weak V,\\
\label{thm1lt2}\sqrt{\log(m_n)}\big(k_n^{1/2}/\sqrt{2}\big) V^*_{n} - 2\log{(m_n)}-\frac12 \log{\log{(m_n)}}-\log{(3)} \weak V,
\end{align}
where $V$ follows an extreme value distribution with distribution function
\begin{align}\P(V \leq x) = \exp(-\pi^{-1/2} \exp(-x))\,.\label{V}\end{align}

\end{theo}

\begin{rem} \rm
It is remarkable that Theorem \ref{thm1} in combination with condition \eqref{assk} allows asymptotically to distinguish between volatility paths with jumps and volatility paths without jumps, where we only require some granted smoothness $\aalpha>0$ in Assumption \ref{assVola} (3). Note that less smooth paths require smaller block lengths $k_n$ by \eqref{assk} which reduces the rate in Theorem \ref{thm1} and the power of the test.
Most importantly, we can cope with standard models for $\sigma$. For a continuous semi-martingale volatility, we have $\aalpha\approx 1/2$. In this case, we take $k_n\propto n^{1/2-\epsilon}$ for $\epsilon>0$ and $\epsilon$ small to preserve the highest possible power. Similarly, for a Lipschitz volatility, i.e.\;$\aalpha=1$, one might choose $k_n\propto n^{2/3-\epsilon}$. Thus, the choice of the block length is close to optimal window sizes for spot volatility estimation.
\qed
\end{rem}
As we show in Theorem \ref{thm_upper1} that $V_n^*$ and $V_n$ diverge under the alternative almost surely, Theorem \ref{thm1} provides a consistent test with asymptotic power 1 by critical values from the limit law under the hypothesis.

\subsection{A test in presence of jumps in the observed process\label{sec:3.2}}
In order to provide a valid approach for various economic applications an important aim is to account for possible jumps in the process $X$ as well. Thereto, consider a general It\^{o} semi-martingale
\begin{align}\label{sm2}
X_t = X_0 + \int_0^t a_s \,ds + \int_0^t \sigma_s \,dW_s&+\int_0^t\int_{\mathds{R}}\kappa(\delta(s,x))(\mu-\nu)(ds,dx)\\
&\notag+\int_0^t\int_{\mathds{R}}\bar\kappa(\delta(s,x))\mu(ds,dx)\,,
\end{align}
where a truncation function $\kappa$, $\bar\kappa(x)=x-\kappa(x)$, separates large from compensated small jumps. The compensating intensity measure $\nu$ of the Poisson random measure $\mu$ admits the form $\nu(ds,dx)=ds\otimes \lambda(dx)$ for a $\sigma$-finite measure $\lambda$. Our notation follows \cite{jacodjumps}.
\begin{ass} \label{assjumps}
Grant Assumption \ref{assVola} for the continuous part of $X$. Suppose $\sup_{\omega,x}|\delta(s,x)|/\gamma(x)$ is locally bounded for some deterministic non-negative function $\gamma$ which satisfies for some $r<2$:
\begin{align}\label{bg}\int_{\mathds{R}}(1\wedge \gamma^r(x))\lambda(dx)<\infty\,.\end{align}
\end{ass}
In condition \eqref{bg}, $r$ is a jump activity index that bounds the path-wise generalized Blumenthal-Getoor index from above. Imposing $r<1$ restricts to jumps of finite variation and $r=0$ to finite jump activity. To develop test statistics which are robust against jumps we employ a truncation principle as introduced for integrated volatility estimation by \cite{mancini} and \cite{jacodjumps}. The analogue of \eqref{statV} with truncated squared increments reads
\begin{align} \label{statVtr}
V_{n,u_n} &=  \max_{i=0,\ldots,\lfloor n/k_n \rfloor -2} \vert TRV_{n,u_n,i}/TRV_{n,u_n,i+1} - 1 \vert,\\
\label{statXtr}
TRV_{n,u_n,i} &= \frac{n}{k_n} \sum_{j=1}^{k_n} (\Delta_{ik_n + j}^n X)^2\1_{\{|\Delta_{ik_n + j}^n X|\le u_n\}}\,,i=0,\ldots,\lfloor n/k_n\rfloor -1.
\end{align}
The truncation sequence $u_n\propto n^{-\tau},\tau\in(0,1/2)$, is used to exclude large squared increments which can be ascribed to jumps.
In the same way we can generalize statistic \eqref{statV2} with overlapping blocks:
\begin{align} \label{statV2tr}
V_{n,u_n}^* =  \max_{i=k_n,\ldots,n-k_n} \Bigg\vert  \frac{ \frac{n}{k_n} \sum_{j=i-k_n+1}^{i} (\Delta_{j}^n X)^2\1_{\{|\Delta_{j}^n X|\le u_n\}}}{\frac{n}{k_n} \sum_{j=i+1}^{i+k_n}(\Delta_{j}^n X)^2\1_{\{|\Delta_{j}^n X|\le u_n\}}}-1\Bigg\vert\,.
\end{align}
We prove below that truncation is an appropriate concept to asymptotically eliminate the influence by jumps, at least under certain restrictions on the jump activity, on $k_n$ and on $\tau$. In particular, under the hypothesis we obtain the same limit behaviour of the test statistics as in Theorem \ref{thm1}.
\begin{prop}\label{propjumps}
Suppose $k_n \propto n^{\beta}$ for $0 < \beta < 1$, such that condition \eqref{assk} is satisfied. Furthermore, grant Assumption \ref{assjumps} for some
\begin{align}\label{conjumps}r<\min{\Big(2\big(2-\tau^{-1}(1-\beta/2)\big), \tau^{-1}\min(1/2,1-\beta)\Big)}
\end{align}
as well. Then, with $m_n = \lfloor n/k_n\rfloor $ and $\gamma_{m_n} = [4 \log(m_n) - 2 \log(\log (m_n))]^{1/2}$ as before, and if either $r=0$ or the jump process is a time-inhomogeneous L\'evy process, we have on $\Omega^c$ (under $H_0$) for the statistics \eqref{statVtr} and \eqref{statV2tr} the weak convergences
\begin{align}
\label{thmj1}\sqrt{\log(m_n)} \big(\big(k_n^{1/2}/\sqrt{2}\big) V_{n,u_n} - \gamma_{m_n}\big) \weak V,\hspace*{-.1cm}\\
\label{thmj2}\sqrt{\log(m_n)}\big(k_n^{1/2}/\sqrt{2}\big) V^*_{n,u_n} \hspace*{-.1cm}-\hspace*{-.05cm} 2\log{(m_n)}\hspace*{-.025cm}-\hspace*{-.05cm}\frac12 \log{\log{(m_n)}}\hspace*{-.025cm}-\hspace*{-.025cm}\log{(3)} \weak V,\hspace*{-.1cm}
\end{align}
where $V$ is distributed according to \eqref{V}.
\end{prop}

\begin{rem}
\rm A simple computation shows that a necessary condition in order for \eqref{conjumps} to hold is $r<1$, but obviously it involves further conditions on the interplay between $r$, $\beta$ and $\tau$. This restriction on the jump activity is stronger than the usual $r<1$ for truncated realized volatility as in \cite{jacodjumps}, and it is due to the maximum in \eqref{statV2tr} compared to linear estimators.
\end{rem}

\begin{rem} \label{trunc}
\rm Arguably, it is most relevant from an applied perspective that the test based on \eqref{statV2tr} copes with finite activity jumps. In this case \eqref{conjumps} reads as $\tau>1/2-\beta/4$, and the only requirement is that $u_n$ is not chosen too large. Beyond the finite activity case, the choice of the tuning parameters becomes more complex and depends on the statistician's interest. Ideally, one would choose $\beta$ large to secure a high power, but it should not become too large as \eqref{conjumps} then is more restrictive and interesting models on the jumps might be ruled out. As an equilibrium choice $\beta \approx 1/2$ is recommended in which case $\tau \approx 1/2$ is optimal, leading to the condition $r < 1$. Choosing $\tau$ close to $1/2$ and $u_n$ small improves the precision of localized truncated realized volatilities \eqref{statXtr}. As an exact choice for $u_n$ one typically picks it sufficiently large to just not interfere with the continuous component of $X_t$. Based on extreme value theory for Gaussian sequences,
\begin{align}u_n=C \sqrt{2\log(n)} n^{-1/2}\end{align}
with constant $C$ is an accurate choice. Once it is guaranteed that $C>\sup_{t\in[0,1]}\sigma^2_t$, almost surely no increments of the continuous part of \eqref{sm2} are truncated. In practice, some suitable upper bound $C$ for the volatility can be obtained from historical data.\qed 
\end{rem}

\section{Asymptotic minimax-optimality results for the local change problem\label{sec:4}}
\subsection{Consistency and minimax-optimal rate of convergence\label{sec:4.1}}
In this section it becomes important that \textit{stochastic} squared volatility processes lie under $H_0$ in $\Sigma(\aalpha,\KK_n)$, defined in \eqref{hyposet}, where we take into account strictly positive increasing sequences $\KK_n$. This is crucial as we cannot describe the random processes as members of a fixed H\"{o}lder class. If $\sigma_t^2$ satisfies
\begin{align*}
\E\bigl[|\sigma_t^2 - \sigma_s^2|^{\mathfrak{b}} \bigr] \leq C \bigl|t - s \bigr|^{\gamma + \mathfrak{b} \aalpha} \quad \text{for some $\mathfrak{b},C > 0$ and $\gamma>1$},
\end{align*}
then the Kolmogorov-\u{C}entsov Theorem implies that \(\lim_{n \to \infty} \P\bigl((\sigma_t^2)_{0 \leq t \leq 1} \in\Sigma(\aalpha,\KK_n) \bigr) = 1\), provided $\KK_n \to \infty$ arbitrarily slowly. Hence, up to a negligible set, $\Sigma(\aalpha,\KK_n)$ contains the \textit{paths} generated by a huge number of popular volatility models when considering $\KK_n \to\infty$. On the other hand, if $\KK$ is fixed, we are in the familiar framework of Hölder classes.

At this stage, we integrate alternatives where the volatility is less smooth than under the hypothesis, but which require not necessarily jumps. The statistical devices developed above may be applied to discriminate $H_0$ from alternatives without jumps where until some change-point $\td \in [0,1)$, the process $(\sigma_{t\wedge \theta}^2)$ behaves as a process in $\Sigma(\aalpha,\KK_n)$.
After $\td$, the regularity exponent drops to some $0 < \aalpha' < \aalpha$. Since $\Sigma\bigl(\aalpha,\KK_n\bigr) \subset \Sigma\bigl(\aalpha',\KK_n\bigr)$, we require functions that `exploit their roughness' in a certain sense. Processes in the alternative set may be smoother on parts of the interval, but we require that they `exploit their roughness' somewhere on $[0,1)$ (close to $\td$ respectively), such that in particular $(\sigma_t^2)_{t\in[0,1]}\not \in \Sigma\bigl(\aalpha,\KK_n\bigr)$. To describe the alternative sets, define
\begin{align*}
\Delta_h^{\aalpha'} f_t = \frac{f_{t + h} - f_{t}}{|h|^{\aalpha'}}, \quad t \in[0,1],h\in [-t,1-t].
\end{align*}
We then express the set of possible alternatives 
\vspace*{.4cm}
\begin{align*}
&\mathcal{S}_{\td}^{\text{R}}\bigl(\aalpha,\aalpha',b_n,\KK_n\bigr) =\Bigl\{\bigl(\sigma^2_{t \wedge \td}\bigr)_{t \in [0,1]} \in \Sigma(\aalpha',\KK_n) \bigl|
\inf_{|h|\leq 2k_n\Delta_n}\Delta_h^{\aalpha'} \sigma^2_{\td} \ge b_n
~~\text{or} \sup_{|h| \leq 2k_n\Delta_n}\Delta_h^{\aalpha'} \sigma^2_{\td} \le -b_n \Bigr\},\end{align*}
and consider the testing problem
\begin{align}\label{hypo_change}
& H_0\hspace*{-.05cm}:(\sigma_t^2(\omega))_{t\in[0,1]}\in\Sigma(\aalpha,\KK_n) \; vs.\,
H_1^{\text{R}}\hspace*{-.05cm}:\exists \, \td \in [0,1) \,\vert (\sigma_t^2(\omega))_{t\in[0,1]}\in\mathcal{S}_{\td}^{\text{R}}\bigl(\aalpha,\aalpha',b_n,\KK_n,k_n\bigr).\hspace*{-.15cm}
\end{align}
Since $k_n$ is selected dependent on $L_n$ and $\aalpha$, the dependence of $\mathcal{S}_{\td}^{\text{R}}$ on $k_n$ becomes redundant. Demanding the exceedance period of the difference quotient in $\mathcal{S}_{\td}^{\text{R}}$ to be at least two block lengths ensures that our block-wise comparison in the test statistics \eqref{statV}, \eqref{statV2} is able to detect the roughness, also for $\td=0$.

Let us elaborate on the specific form of the alternative sets. In general, it is impossible to test $\Sigma(\aalpha,\KK_n)$ against $\Sigma(\aalpha',\KK_n)$ for $\aalpha > \aalpha'$, and it is necessary to consider special subsets of $\Sigma(\aalpha',\KK_n)$. Intuitively, it is clear that one needs at least to remove $\Sigma(\aalpha,\KK_n)$ from $\Sigma(\aalpha',\KK_n)$, but this is not sufficient. In fact, one needs to focus on the functions which exploit their roughness in a certain sense; cf.\,\cite{nicklhoffmann} for a detailed discussion in a related context. Geometrically, this means that the functions of interest are those with discontinuities or rough behavior as characterized in $\mathcal{S}_{\td}^{\text{R}}$ (or which fluctuate considerably more, like the ones considered in Section \ref{sec:5}). 
However, as the sample size $n$ grows, we only require the difference quotient to exceed a level $b_n$ that becomes smaller and smaller.

For the testing problems \eqref{hypo} and \eqref{hypo_change}, we first present a negative result, that also serves as a minimax lower bound for the problem depicted in \eqref{defn_minimax_optimal_test}.
\begin{theo}\label{thm_lowerbound}
Assume that $\aalpha> \aalpha' >0$ and $\inf_t \sigma_t^2 \geq \sigma_{-}^2 > 0$. Consider either set of hypotheses $\bigl\{H_0,H_1\bigr\}$ or
$\bigl\{H_0,H_1^{\text{R}}\bigr\}$. Then for
\begin{align}\label{eq_thm_lower_bound_defn}
b_n  \leq \bigl(n/\log(m_n) \bigr)^{-\frac{\aalpha-\aalpha'}{2 \aalpha + 1}} \bigl(\KK_n \bigr)^{-\frac{2\aalpha'+1}{2 \aalpha + 1}}\,\sigma_-^2\,,
\end{align}
with $\aalpha'=0$ for $H_1$, we have in both cases \(\lim_{n \to \infty} \inf_{\psi}\gamma_{\psi}\bigl(\aalpha,b_n \bigr) = 1.\)
\end{theo}
Theorem \ref{thm_lowerbound} reveals that it is impossible to construct a minimax-optimal test in the sense of \eqref{defn_minimax_optimal_test} if $b_n$ is bounded as in \eqref{eq_thm_lower_bound_defn}. 
Consequently, we deduce that
\begin{align}\label{eq_lower_bound_for_bn_opt}
b_n^{\text{opt}} \geq \bigl(n/\log(m_n) \bigr)^{-\frac{\aalpha-\aalpha'}{2 \aalpha + 1}} \bigl(\KK_n \bigr)^{-\frac{2\aalpha'+1}{2 \aalpha + 1}}\,\sigma_-^2\,.
\end{align}
In Theorem \ref{thm_upper1} we shall establish a corresponding upper bound up to a constant, and thus \eqref{eq_lower_bound_for_bn_opt} already gives the optimal rate for the minimax distinguishable boundary.
Observe that based on $V_n^{*}$ from \eqref{statV2}, we can obtain the following test $\psi^{\diamond}$.
\begin{align}\label{eq_defn_test}
\psi^{\diamond}\bigl((X_{i\Delta_n})_{0\le i\le n}\bigr) = 1,~\text{if}~ V_n^{*} \geq 2 C^{\diamond} \sqrt{2 \log(m_n^{\diamond})/k_n^{\diamond}},
\end{align}
\vspace*{-1.25cm}

\begin{align}\label{eq_kn_and_mn}
\mbox{where $C^{\diamond} > 2$ and}~~k_n^{\diamond} = \bigl(\sqrt{\log(m_n^{\diamond})}n^{\aalpha}/\KK_n \bigr)^{\frac{2}{2 \aalpha +1}}, \quad m_n^{\diamond} = \lfloor n/k_n^{\diamond} \rfloor.~~~~~~~~~~
\end{align}
Alternatively, one might base a test on $V_n$ from \eqref{statV}.

To simplify the discussion, we restrict to positive volatility jumps, i.e.\, $\inf_t \Delta \sigma_t \geq 0$, which appears natural from an economic point of view. We point out that an analogue result can be shown for negative, or positive and negative jumps, which however requires a further technical structural condition (that jumps do not cancel each other) in case of multiple jumps in a vicinity for the alternative set.
\begin{theo}\label{thm_upper1}Consider \eqref{hypo} with $\inf_t \Delta \sigma_t \geq 0$, or \eqref{hypo_change} with $0 < \aalpha' < \aalpha \leq 1$ and $\KK_n = \KLEINO\bigl((n/k_n^{\diamond})^{\aalpha - \aalpha'}\bigr)$.
\begin{align}\label{ass_b_n_diamond}\hspace*{-2cm}\mbox{If   }\hspace*{2cm}
b_n^{\diamond} >  \Big(4 C^{\diamond} \sqrt{2}\sup_{t \in [0,1]} \sigma_t^2+2\Big)\bigl(n/\log(m_n) \bigr)^{-\frac{\aalpha-\aalpha'}{2 \aalpha + 1}} \bigl(\KK_n \bigr)^{-\frac{2\aalpha'+1}{2 \aalpha + 1}}\,,
\end{align}
where $k_n^{\diamond}$, $m_n^{\diamond}$ and $C^{\diamond}$ are as in \eqref{eq_kn_and_mn}, then
$\lim_{n \to \infty} \gamma_{\psi^{\diamond}}\bigl(\aalpha, b_n^{\diamond} \bigr) = 0$.
This implies that
\begin{align*}
b_n^{\text{opt}}\propto  \bigl(n/\log(m_n) \bigr)^{-\frac{\aalpha-\aalpha'}{2 \aalpha + 1}} \bigl(\KK_n \bigr)^{-\frac{2\aalpha'+1}{2 \aalpha + 1}}\,.
\end{align*}
\end{theo}
\begin{rem} \rm
If $\KK$ defined in \eqref{hyposet} is a deterministic constant, we get the minimax distinguishable boundary $b_n\propto (n/\log (n))^{\frac{-(\aalpha-\aalpha')}{2\aalpha+1}}$. \qed
\end{rem}
\begin{rem} \rm
Volatility paths lying in $\mathcal{S}_{\td}^{\text{R}}$ having (locally) a rough -- but still continuous -- increase or decrease cannot be distinguished from volatility paths with jumps at or below the boundary $b_n$ stated in \eqref{eq_thm_lower_bound_defn}. Both alternatives $H_1$ and $H_1^{\text{R}}$ are in that sense intimately connected. Formally, we cannot include $\mathcal{S}_{\td}^{J}$ in $\mathcal{S}_{\td}^{\text{R}}$ setting $\aalpha'=0$, since for $\mathcal{S}_{\td}^{\text{R}}$ we demand the roughness to persist over an (asymptotically small) interval. 

Let us point out that in this testing problem the union of hypothesis and alternative can not cover the set of all possible volatility paths. One situation of interest in which $\sigma_t^2\notin \{\mathcal{S}_{\td}^{\text{R}}\cup \mathcal{S}_{\td}^{J}\cup \Sigma(\aalpha,L_n)\}$ is the case of fractional processes with Hurst parameter $\aalpha'<\aalpha$. This different situation is addressed in Section \ref{sec:5}. \qed
\end{rem}
\subsection{Estimating the change-point\label{sec:4.2}}
Once one has opted to reject the null hypothesis of no change, the actual locations of jumps become of interest for further inference. This location problem has been extensively discussed in the literature in different frameworks; see for instance \cite{horvath} and \cite{mueller}.
\subsubsection{One change-point alternative}
First, we restrict ourselves to the `one change-point alternative' involving a jump in the volatility, i.e.\;we specify the alternative hypothesis $H_1^*$ as
\begin{align*}
H_1^*: \quad \bigl|\sigma_{\td}^2 - \sigma_{\td-}^2\bigr| =: \delta_n >0 \quad \text{for a unique $\td \in (0,1)$.}
\end{align*}
The jump size $\delta_n$ may be fixed or we consider a decreasing sequence $(\delta_n)$.
To assess the possible time of change, we use slightly modified versions of the building blocks of the test statistic $V_n^*$ from \eqref{statV2}, defined as
\begin{align*}
V_{n,i}^{\diamond} = \frac{1}{\sqrt{k_n}} \biggl|\sum_{j = i- k_n + 1}^{i}n(\Delta_j^n X)^2 - \sum_{j = i + 1}^{i + k_n} n(\Delta_j^n X)^2 \biggr|\,,
\end{align*}
for $i=k_n,\ldots, n - k_n$, and $V_{n,i}^{\diamond} =0$ else. In contrast to the construction of $V_n^*$, we may employ a simpler unweighted version. One can also consider the rescaled versions as in $V_n^*$, and we conjecture that the following theoretical results of these estimators coincide. Here, we switch from ratios to differences which simplifies the analysis a bit and is enough to obtain the following properties. The possible time of change is then estimated via
\begin{align}\label{defn_hat_td}
n\widehat{\td}_n = \operatorname{argmax}_{i=k_n,\ldots, n-k_n} V_{n,i}^{\diamond}\,.
\end{align}
The following proposition establishes quantitative bounds for the quality of estimation.
\begin{prop}\label{propest}
Assume that the assumptions of Theorem \ref{thm1} hold and that $H_1^*$ is valid. Then, for $\delta_n\ge 2k_n^{-1/2}\sqrt{\log(n)}\sup_{t\in[0,1]}\sigma_t^2$, we have that
\begin{align}\label{estrate}
\bigl|\widehat{\td}_n - \td \bigr| = \mathcal{O}_{\P}\biggl(\frac{\sqrt{k_n \log (n)}}{n \delta_n}\biggr)\,.
\end{align}
\end{prop}

\begin{rem} \rm 
If $\delta_n$ does not tend to zero, the condition on $\delta_n$ in the proposition is always satisfied. The estimator extends to jumps of $X$ using truncation as in \eqref{statV2tr}, and Proposition \ref{propest} then applies to the generalized estimator under the assumptions of Proposition \ref{propjumps}. In the setup of continuous breaks under alternative $\mathcal{S}_{\td}^{\text{R}}$ the same estimator is consistent (only) when except on a small interval around $\td$ the volatility is $\aalpha$-regular. When the maximal length of this interval is $\sqrt{k_n \log (n)}/(n \delta_n)$, \eqref{estrate} applies when we replace $\delta_n$ by $\delta_n (k_n\Delta_n)^{\aalpha'}$. Clearly, when the volatility violates $\aalpha$-regularity over longer time horizons such a result is not available.\qed
\end{rem}

Obviously, the quality of the estimator $\widehat{\td}_n$ depends on the bandwidth $k_n$, and the smaller, the better. This is the complete opposite case compared to the test based on statistic $V_n^*$, where a larger choice of $k_n$ increases the power. This is no contradiction, since both problems have a different, essentially reciprocal nature. Also note that $k_n$ cannot be chosen arbitrarily small; see condition \eqref{assk}.

While classical estimators as the argmax of statistic \eqref{cusum} attain a standard $\sqrt{n}$-rate, corresponding to $k_n\approx n$, our nonparametric localization approach readily facilitates improved convergence rates as known for state-of-the-art change-point estimators, as e.g.\,in \cite{aue}. The following proposition sheds light on optimal convergence rates for the estimation problem.
\begin{prop}\label{prop_time_est_optimal}
On the assumptions of Proposition \ref{propest} for $k_n \propto \bigl(\sqrt{\log (n)} n^{\aalpha}\bigr)^{\frac{2}{2\aalpha + 1}}$, a consistent estimator for $\td$ does not exist in the case that $ \delta_n = \KLEINO\bigl(\sqrt{\log (n)}k_n^{-1/2}\bigr)$.
\end{prop}
\subsubsection{Multiple change-point alternatives}
In the sequel, we demonstrate how the previous theory can be extended to multiple change points. To keep this exposition at a reasonable length, we focus on the alternative where the volatility exhibits jumps. From a general perspective, multiple change-point detection is typically a challenging multiple testing problem compared to the one change-point detection problem. The main probabilistic difficulty usually lies in controlling the overall stochastic error. Fortunately, in the present context we have already successfully dealt with the overall stochastic error, see Theorems \ref{thm1} and \ref{thm_upper1}. Thus, treating the multiple change-point problem only requires small adjustments. For $N \in \N$, let
\begin{align}
0 < \td_1 < \ldots < \td_N < 1, \quad \Theta_N = \{\td_1,\ldots,\td_N\}.
\end{align}
We then consider the alternative 
\begin{align*}
H_1^*: \quad \bigl|\sigma_{\td_i}^2 - \sigma_{\td_i-}^2\bigr| =: \delta_{n,i} >0 \quad \text{for $1 \leq i \leq N$.}
\end{align*}
The number of changes $N$ in unknown to the experimenter. The goal is to provide uniformly consistent estimates for the multiple change points $\td \in \Theta_N$. To this end, given an index set $\mathcal{I} \subseteq \{k_n,\ldots,n-k_n\}$, we define in analogy to \eqref{defn_hat_td}
\begin{align}\label{defn_hat_td_multiple}
n\widehat{\td}_n(\mathcal{I}) = \operatorname{argmax}_{i \in {\mathcal{I}}} V_{n,i}^{\diamond}\,.
\end{align}

Based on the test $\psi^{\diamond}$ introduced in \eqref{eq_defn_test}, we propose the following algorithm for multiple change-point detection.
\begin{alg}\label{alg_multiple:detect}
\hfill
\begin{description}
\item[Initialize] Set $\hat{\mathcal{I}} = \{k_n,\ldots,n-k_n\}$, $\hat{\Theta} = \emptyset$ and select $r_n = \KLEINO(n)$ such that $k_n = \KLEINO(r_n)$, $k_n \to \infty$.
\item[(i)]If $\psi^{\diamond}\bigl((X_{i\Delta_n})_{i \in \hat{\mathcal{I}}}\bigr) = 0$, stop and return $\hat{\mathcal{I}}$ and $\hat{\Theta}$. Otherwise go to step {\bf (ii)}.
\item[(ii)] Estimate one time of change $\td$ using $\widehat{\td}_n(\hat{\mathcal{I}})$ from \eqref{defn_hat_td_multiple}.
\item[(iii)] Set $\hat{\mathcal{I}} = \hat{\mathcal{I}} \setminus \{\lfloor\hat{\td}_{n}(\hat{\mathcal{I}})n \rfloor - r_n, \ldots, \lceil\hat{\td}_{n}(\hat{\mathcal{I}})n\rceil + r_n\}$, $\hat{\Theta} = \hat{\Theta}\cup \{\widehat{\td}_n(\hat{\mathcal{I}})\}$, and go to step {\bf (i)}.
\end{description}
\end{alg}

Algorithm \ref{alg_multiple:detect} is a sequential top-down algorithm, similar in spirit to the well-known bisection methods. Observe that it not only returns an estimate for the set of change-points $\Theta_N$, but also the set $\hat{\mathcal{I}}$ of non-contaminated indices, which can be used for further inference. The following result provides consistency of the proposed set estimators.

\begin{prop}\label{thm_multiple:change-points}
On the assumptions of Theorem \ref{thm1} and under $H_1^*$, if
\begin{itemize}
  \item[(i)] for some $N' = \KLEINO(n/r_n)$, it holds that $\,\inf_{1 \leq i \leq N-1} |\theta_{i+1} - \theta_i|\geq (N')^{-1}$,
  \item[(ii)] $\inf_{1 \leq i \leq N}\delta_{n,i} \geq 2k_n^{-1/2}\sqrt{\log(n)}\sup_{t\in[0,1]}\sigma_t^2$,
\end{itemize}
then we have consistency of $\hat{\Theta}$, i.e.;
$\begin{aligned}\P(|\hat{\Theta}| = N) \to 1,~\mbox{and}~\sup_{n=1,\ldots,N}|\widehat{\td}_n-\td_n|\pn 0\,.
\end{aligned}$ 
\end{prop}

Note that we can allow for increasing $N = \KLEINO(N')$ as the sample size $n$ increases. The bound in condition $(ii)$ is optimal if $k_n$ is selected in the optimal way $k_n \propto \bigl(\sqrt{\log (n)} n^{\aalpha}\bigr)^{\frac{2}{2\aalpha + 1}}$.

\section{A change-point test and asymptotic results for the global change problem\label{sec:5}}
In Section \ref{sec:3} we present methods to test the hypothesis of $\aalpha$-regular volatilities against local alternatives of less regular volatilities. An important example have been volatility jumps which violate the hypotheses for any $\aalpha>0$. If the hypothesis is rejected for a pre-specified $\aalpha$, however, the test does not reveal if this is due to a volatility jump or due to a change of the regularity where the volatility is $\aalpha$-regular on $[0,\td)$ and $\aalpha'$-regular on $(\td,1]$ with $\aalpha'<\aalpha$. In case of global changes the test (under certain conditions) rejects as well. Moreover, alternatives in Section \ref{sec:4} do not cover a change in the Hurst parameter of a fractional volatility process.
The latter constitutes a different testing problem of great interest which is addressed in this section. We develop a new test which discriminates between volatilities which are $\aalpha$-regular on $[0,1]$, except for a finite number of discontinuities, and volatilities where at $\td\in(0,1)$ the regularity exponent drops to $\aalpha'<\aalpha$ such that $\aalpha$-regularity is permanently violated on $[\td,1]$. We consider processes which satisfy the following regularity assumptions.
\begin{ass}
\begin{enumerate}
\item[(i)] Drift and volatility process in \eqref{sm} are c\`adl\`ag, $\inf_{t\in[0,1]}\sigma_t^2> 0$.
\item[(ii)] For a finite set $\mathcal{T}=\{\tau_1,\ldots,\tau_N\}$ of $N<\infty$ stopping times\\ $\sup_{j=1,\ldots,N}|\sigma_{\tau_j}-\lim_{t<\tau_j,t\rightarrow\tau_j}\sigma_{t}|\le K$ with a constant $K<\infty$.
\item[(iii)] $\sigma_t^2=\nu_t+\varrho_t$ with the $\sigma$-algebra $\sigma\big(\varrho_s,0\le s\le 1\big)$ being independent of $\sigma\big(W_s,\nu_s,0\le s\le 1\big)$.
For any $0\leq s,\tau\leq 1$ with $[s,\tau]\cap \mathcal{T}=\emptyset$, we have with a constant $K$ and some $\epsilon>0$:
\begin{align}
\label{nusmooth}\bigl(\E\bigl[|\nu_{\tau}-\nu_{s}|^8\bigr]\bigr)^{1/8} & \le K\,|\tau-s|^{(1/2+\epsilon)}\,.
\end{align}
\item[(iv)]For $\begin{aligned}\Delta_i^n \varrho=n\,\int_{(i-1)\Delta_n}^{i\Delta_n}(\varrho_{s}-\varrho_{s-\Delta_n})\,ds\end{aligned}$, it holds that 
\begin{align}\max_{2\le m\le n}\Big|\sum_{i=2}^m \Big((\Delta_i^n\varrho)^2-\E\big[(\Delta_i^n\varrho)^2\big]\Big)\Big|
    =\KLEINO_{\P}(\sqrt{n})\,.
    \end{align}
\end{enumerate}
\label{assvolaglobal}
\end{ass}
\begin{testingproblem}\label{testingproblem}
\textbf{Hypothesis:} The process $(\varrho)_{t\in[0,1]}$ is $\aalpha$-regular in the following sense:
\begin{align}
\label{rhosmooth}\E\bigl[(\Delta_i^n\varrho)^2\bigr] &= \vartheta_n^2 + \KLEINO(n^{-1/2}), ~\mbox{for all $i$}, \vartheta_n^2 \le K n^{-2\aalpha}\,,\\
		\label{rhosmooth2}\big(\E\bigl[(\Delta_i^n\varrho)^8\bigr]\big)^{1/8}&\le K\,n^{-\aalpha},~\mbox{for all $i$}\,,
\end{align}
for some constant $K$ and $\aalpha > 0$. Thus, on the hypothesis $(\sigma_t^2)_{t\in [0,1]}$ is $\min(1/2,\aalpha)$-regular.\\
\textbf{Alternative:} For $\td\in(0,1)$, $(\varrho_{t\wedge \td})_{t\ge 0}$ satisfies \eqref{rhosmooth} and \eqref{rhosmooth2} on $[0,\td)$.
For some $\aalpha' <\min(1/2,\aalpha)$, $b_n'>0$ and for all $i\Delta_n \geq \td$, we have
\begin{align}\label{rhosmooth:alternative}
\E\bigl[(\Delta_i^n\varrho)^2\bigr] \geq b_n' n^{-2\aalpha'}\,.
\end{align}
\end{testingproblem}
Contrarily to the setup of Section \ref{sec:3}, the hypothesis allows for a finite number of discontinuities in $(\nu_t)_{t\in[0,1]}$. On the other hand, under the alternative the volatility \emph{permanently `exploits its roughness'} in the sense of violating \eqref{rhosmooth} permanently over $[\td,1]$. Condition \eqref{nusmooth} can be extended also to the case where $\nu$ is an It\^{o} semi-martingale with jumps of finite activity. Assumption \ref{assvolaglobal} $(iv)$ grants that $\E[(\varrho_t-\varrho_s)^2]$ does not vary too much over time. This condition is obsolete for $\aalpha>1/4$.

Our setup covers many stochastic volatility models, and in particular it applies to discriminate fractional volatility models with different Hurst parameters. While most fractional stochastic volatility models include independence of log-price and volatility, possible dependence (leverage) is usually allowed in the literature using a semi-martingale volatility. In this light, the decomposition in Assumption \ref{assvolaglobal} appears natural, where $(\varrho_t)_{t\ge 0}$ is independent of $(W_t)_{t\ge 0}$ and $(\nu_t)_{t\ge 0}$, whereas $(\nu_t)_{t\ge 0}$ comprises leverage. The following simple example reveals the interplay of Assumption \ref{assvolaglobal} $(iv)$ and \eqref{rhosmooth} and the Hurst parameter.
\begin{exa}[Fractional Brownian motion]\label{example:fractional}
Suppose that $(\varrho_t)$ is a fractional Brownian motion with Hurst parameter $H$. Let $\xi_{k} = n^H(\varrho_{k\Delta_n} - \varrho_{(k-1)\Delta_n})$, $k \geq 1$. Then, see e.g.\;\cite{embrechts},
\begin{itemize}
\item[(i)] $\varrho_r - \varrho_s \stackrel{d}{=} \varrho_{r-s} \stackrel{d}{=} (r-s)^H \xi_1$, $r \geq s$,
\item[(ii)] $\bigl|\E\bigl[\xi_{0} \xi_{k} \bigr]\bigr| \le K\, (k+1)^{2H - 2}$, $k \geq 1$, for some constant $K$.
\end{itemize}
The scaling property (i) implies \eqref{rhosmooth} and \eqref{rhosmooth2} with $H=\aalpha$. If $1/4 \le H \leq 1$, then (i) suffices also to guarantee the validity of Assumption \ref{assvolaglobal} $(iv)$. If $0 < H < 1/2$, then
\begin{align*}
\E\Big[\max_{2\le m\le n}\hspace*{-.05cm}\Big|\sum_{i=2}^m \hspace*{-.05cm}\Big(\hspace*{-.1cm}(\Delta_i^n \varrho)^2\hspace*{-.05cm}-\hspace*{-.05cm}\E\bigl[(\Delta_i^n\varrho_{})^2\bigr]\hspace*{-.05cm}\Big)\hspace*{-.025cm}\Big|\Big] \hspace*{-.05cm}\leq \hspace*{-.05cm} n^{-2H} \hspace*{-.1cm}\int_{[0,1]^2} \hspace*{-.05cm}\E\Big[\max_{2\le m\le n}\hspace*{-.05cm}\Big|\sum_{k=1}^{m-1}\hspace*{-.05cm} \Big(\xi_{k+r}\xi_{k+s}\hspace*{-.05cm} -\hspace*{-.05cm} \E\bigl[\xi_{k+r}\xi_{k+s} \bigr]\hspace*{-.05cm}\Big)\hspace*{-.025cm}\Big|\Big]dr ds.
\end{align*}
Then (ii) together with the joint Gaussianity of $(\xi_{k+r},\xi_{k+s})$ implies that $(\xi_{k+r}\xi_{k+s})$ is a short memory sequence. In particular, using the results in \cite{arcones} and \cite{moricz} yields that
\begin{align*}
n^{-2H}\sup_{0 \leq r,s \leq 1}\E\Big[\max_{2\le m\le n}\Big|\sum_{k=1}^{m-1} \Big(\xi_{k+r}\xi_{k+s} - \E\bigl[\xi_{k+r}\xi_{k+s} \bigr]\Big)\Big|\Big] = \mathcal{O}_{}\bigl(n^{-2H+1/2}\bigr)\,,
\end{align*}
and hence the validity of Assumption \ref{assvolaglobal} $(iv)$.
\end{exa}
More generally, our setup includes prominent realistic volatility models, such as fractional Ornstein-Uhlenbeck processes discussed in \cite{comte}, also considered in Section \ref{sec:6}.
   
Because of the different permanent nature of the change, we derive a new statistical device to address Testing problem \ref{testingproblem} which differs from the methods in Section \ref{sec:3}. In particular, we propose a global cusum-type test statistic instead of localized ones. Define for $i=2,\ldots,n$:
\begin{align}\label{q}Q_{n,i}=n^2\Big(\big((\Delta_i^n X)^2-(\Delta_{i-1}^n X)^2\big)^2-\frac{2}{3}\big((\Delta_i^n X)^4+(\Delta_{i-1}^n X)^4\big)\Big)\,.\end{align}
Our cusum-type test statistic based on the statistics \eqref{q} is
\begin{align}\label{testglobal}V_n^{\dagger}=\frac{1}{\sqrt{n-1}}\max_{m=2,\ldots,n}\Big|\sum_{i=2}^m\Big(Q_{n,i}-\frac{\sum_{i=2}^{n}Q_{n,i}}{n-1}\Big)\Big|\,.\end{align}
Intuitively, statistics $Q_{n,i}$ are small if $|\sigma_{i\Delta_n}^2-\sigma_{(i-1)\Delta_n}^2|$ is small and become larger the larger $|\sigma_{i\Delta_n}^2-\sigma_{(i-1)\Delta_n}^2|$. The regularity $\aalpha$ thus directly influences the average behaviour of the $Q_{n,i}$ and a change at time $\td$ can be detected by \eqref{testglobal}.
\begin{theo}\label{upperglobal}On the hypothesis of the Testing problem \ref{testingproblem} and Assumption \ref{assvolaglobal}, the cusum-process associated with statistic \eqref{testglobal} obeys the functional convergence
\begin{align}\label{cusumclt}\frac{\sqrt{3/80}}{\sqrt{n-1}}\Big(\sum_{i=2}^{\lfloor nt\rfloor}\Big(Q_{n,i}\hspace*{-.05cm}-\hspace*{-.05cm}\frac{\sum_{i=2}^{n}Q_{n,i}}{n-1}\Big)\hspace*{-.05cm}\Big)\hspace*{-.05cm}\stackrel{\omega-(st)}{\longrightarrow}\hspace*{-.1cm}\Big(\int_0^t \hspace*{-.1cm}\sigma_s^4\,dB_s\hspace*{-.05cm}-\hspace*{-.05cm}t\hspace*{-.05cm}\int_0^1\hspace*{-.1cm} \sigma_s^4\,dB_s \Big)\,,\end{align}
stable with respect to $\mathcal{F}$, weakly in the Skorokhod space, where $(B_s)$ denotes a Brownian motion independent of $\mathcal{F}$. 
\end{theo}
As an immediate consequence of Theorem \ref{upperglobal}, we obtain 
\begin{align}
\sqrt{\frac{3}{80}}\;{V}_n^{\dagger} \stackrel{\omega-(st)}{\longrightarrow} V^{\dagger}, \quad V^{\dagger} = \sup_{0 \leq t \leq 1}\Bigl|\int_0^t \sigma_s^4 d B_s - t \int_0^1 \sigma_s^4 d B_s \Bigr|\,.
\end{align}
In order to construct a test $\psi^{\dagger}$ based on ${V}_n^{\dagger}$, the key object are the (conditional) quantiles
\begin{align}
{q}_{\alpha}({V}^{\dagger}|\F) = \inf\bigl\{x \geq 0\, : \, \P({V}^{\dagger} \leq x | \F) \geq \alpha \bigr\}.
\end{align}
The latter depend on the unknown volatility $(\sigma_t)_{t\in[0,1]}$, and are therefore not a priori available. One way to circumvent this problem is to locally estimate $(\sigma_t)_{t\in[0,1]}$ and to apply an appropriate bootstrap procedure to approximate ${q}_{\alpha}({V}^{\dagger}|\F)$. Alternatively, a standardized version of \eqref{cusumclt} facilitates a test based on a Kolmogorov-Smirnov limit law which is given in Proposition \ref{corrupperglobal} and requires a slightly stronger additional condition. For both, consider for $K_n\rightarrow\infty$, $K_n/n\rightarrow 0$ 
\begin{align}\label{eq_vola_est_for_boot_1}
\hat{\sigma}_{i\Delta_n}^4=\frac{n^2}{3K_n}\sum_{j=i-K_n}^{i}(\Delta_j^n X)^4\,,K_n+1\le i\le n\,.
\end{align}
For a sequence of i.i.d.\,standard normals $\{Z_i\}_{1 \leq i \leq \lfloor nK_n^{-1}\rfloor }$, denote with
\begin{align}\label{boots}
\hat{V}_n^{\dagger} = \sup_{0 \leq t \leq 1}\bigl|\hat{S}_{nt} - t \hat{S}_{n}\bigr|, \quad \hat{S}_{nt} =  \Big(\frac{K_n}{n}\Big)^{1/2}\sum_{i = 1}^{\lfloor nt/K_n \rfloor} \hat{\sigma}_{(iK_n+1)\Delta_n}^4 Z_i.
\end{align}
Based on $\hat{V}_n^{\dagger}$, we construct the approximative (conditional) quantiles 
\begin{align}
\hat{q}_{\alpha}(\hat{V}_n^{\dagger}|\F) = \inf\bigl\{x \geq 0\, : \, \P(\hat{V}_n^{\dagger} \leq x | \F) \geq \alpha \bigr\}.
\end{align}
We can compute $\hat{q}_{\alpha}(\hat{V}_n^{\dagger}|\F)$ as accurately as we want using Monte Carlo approximations. Testing problem \ref{testingproblem} is now addressed  with the test:
\begin{align}\label{eq:defn:test:global:boot}
\psi_{\alpha}^{\dagger}\bigl((X_{i\Delta_n})_{0\le i\le n}\bigr) = \left\{\begin{array}{cl} 1, & \mbox{if $\sqrt{\frac{3}{80}}\,{V}_n^{\dagger} > \hat{q}_{1-\alpha}(\hat{V}_n^{\dagger}|\F)$,}\\ 0, & \mbox{otherwise.} \end{array}\right.
\end{align}
Observe that this test does not require any pre-specification of $\aalpha,\aalpha'$. It reacts to the change under the alternative for any $\aalpha,\aalpha'$. From a statistical perspective, it is important that 
\begin{itemize}
\item[(A)] $\psi_{\alpha}^{\dagger}$ (asymptotically) correctly controls the type I error under $H_0$,
\item[(B)] $\psi_{\alpha}^{\dagger}$ provides optimal power (in minimax sense),
\end{itemize}
established by Theorem \ref{thm:bootstrap:global} and Theorem \ref{lowerglobal}.
\begin{theo}\label{thm:bootstrap:global}
Grant Assumption \ref{assvolaglobal}.
Then, for any fixed $\alpha > 0$,
$\begin{aligned}
\bigl|\P\bigl({V}_n^{\dagger} \leq \hat{q}_{\alpha}(\hat{V}_n^{\dagger}|\F) \bigr) - \alpha \bigr| \to 0.
\end{aligned}$
\end{theo}
Before discussing property (B) of the test, let us first touch on a second approach that avoids a bootstrap. Slightly modified assumptions allow for the following standardized version of \eqref{testglobal}.
\begin{prop}\label{corrupperglobal}Grant Assumption \ref{assvolaglobal} and the hypothesis of Testing problem \ref{testingproblem} with $\aalpha>1/4$.  In this case \eqref{cusumclt} holds true. Moreover, 
$(3/80)^{1/2}\,\bar V_n^{\dagger}$ with
\[\bar V_n^{\dagger}=n^{-1/2}\max_{K_n+1\le m\le n}\big|\sum_{i=K_n+1}^m (\hat\sigma_{(i-2)\Delta_n}^4)^{-1}Q_{n,i}-(m/n)\sum_{i=K_n+1}^n(\hat\sigma_{(i-2)\Delta_n}^4)^{-1}Q_{n,i}\big|\] weakly converges to a Kolmogorov-Smirnov law and the associated limiting process in \eqref{cusumclt} becomes a standard Brownian bridge.
\end{prop}
\begin{theo}\label{lowerglobal}For this testing problem, the minimax distinguishable boundary satisfies
\begin{align}
b_n \propto\,\big(n^{-1/2+2\aalpha'}+n^{-2(\aalpha-\aalpha')}\big)\,.\end{align}
In particular, for $b_n =\KLEINO\big(n^{-1/2+2\aalpha'}+n^{-2(\aalpha-\aalpha')}\big)$ a consistent test does not exist:
\[\lim_{n\rightarrow\infty} \inf_{\psi} \gamma_{\psi}\bigl(\aalpha,b_n \bigr) = 1\,.\]
\end{theo}
\begin{rem}The lower bound in Theorem \ref{lowerglobal} reveals that detecting an alternative with ``too smooth'' volatility, $\aalpha'>1/4$, is not possible in a high-frequency setting.\qed\end{rem}
The boundary $b_n$ in Theorem \ref{lowerglobal} is of slightly different nature than the one in Theorem \ref{thm_lowerbound}. Roughly speaking, the testing problem in Section \ref{sec:4.1} can be associated with a high-dimensional statistical experiment, whereas the one here is attached to a univariate statistical experiment. In this case, an optimal test $\psi = \psi_{\alpha}$ can only reach the lower bound up to a pre-specified nominal level $1-\alpha$, $0 < \alpha < 1$, see e.g.\,\cite{ingster}. Equivalently, we call a sequence of tests $\psi_n$ minimax-optimal if for any $b_n'$ with $n^{-1/2+2\aalpha'}+n^{-2(\aalpha-\aalpha')} = \KLEINO(b_n')$ 
\begin{align}\label{defn:optimal:global}
\lim_{n\rightarrow\infty} \gamma_{\psi_n}\bigl(\aalpha,b_n'\bigr) = 0.
\end{align}
\begin{prop}\label{prop:optimal:test:global}
The test $\psi_{\alpha}^{\dagger}$ in \eqref{eq:defn:test:global:boot} is minimax-optimal. A corresponding test based on Proposition \ref{corrupperglobal} is also minimax-optimal if $\aalpha > 1/4$. In particular, under the alternative of Testing problem \ref{testingproblem} we have ${V}_n^{\dagger}\pn\infty$, when
\begin{align}\label{eq:prop:optimal:test:global:b_n}
n^{-1/2+2\aalpha'}+n^{-2(\aalpha-\aalpha')} = \KLEINO(b_n')\,. 
\end{align}

\end{prop}
\begin{rem}
Truncation in \eqref{testglobal}, analogously as in Section \ref{sec:3.2}, facilitates a method robust to jumps of $X_t$. The time of change $\td$ can be estimated using the $\operatorname{argmax}$. Precise results on the latter aspects are left for future research.
\qed
\end{rem}
We have established a minimax-optimal test for global changes. The methods from Section \ref{sec:3} react under some conditions also to global changes, but forfeit optimality. Combining both approaches provides the statistician with suitable devices to analyze volatility dynamics.

\section{Simulations\label{sec:6}}

\begin{figure}[ht!]
\fbox{
\includegraphics[width=7.5cm]{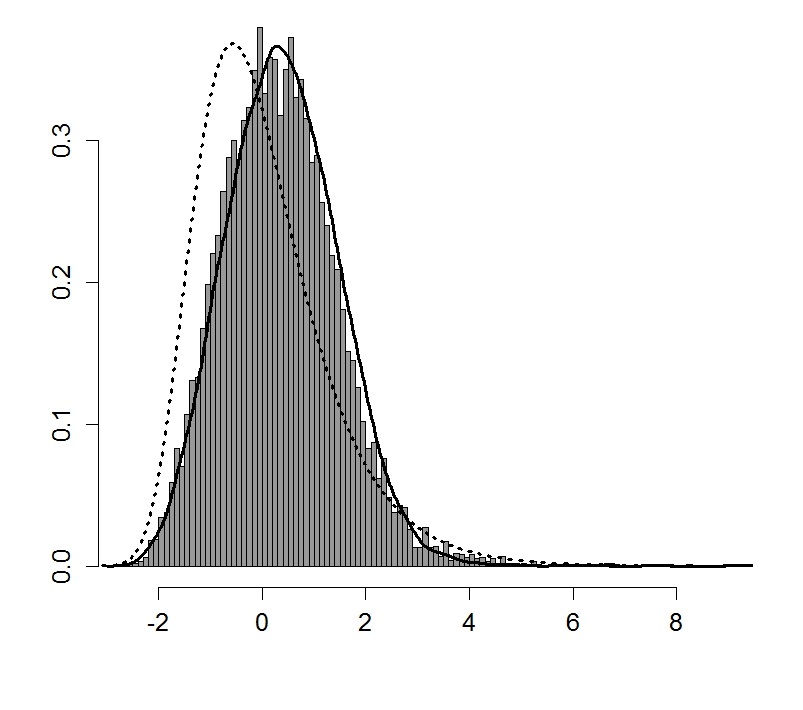}~\includegraphics[width=7.5cm]{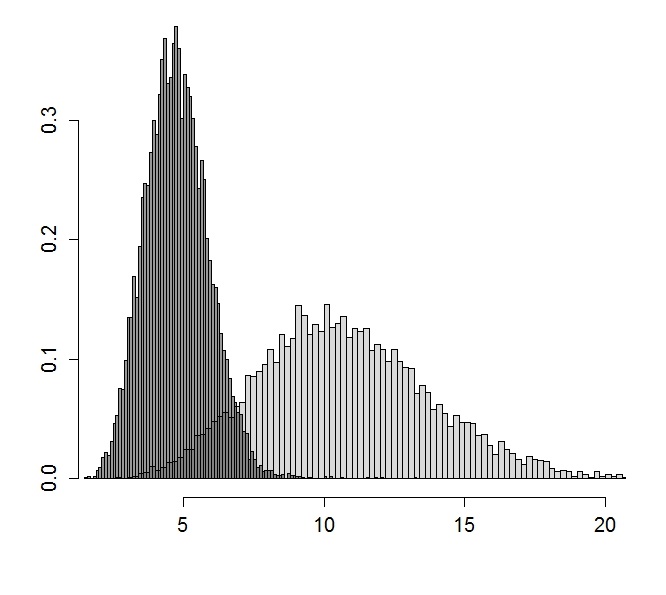}}
\fbox{
\includegraphics[width=7.5cm]{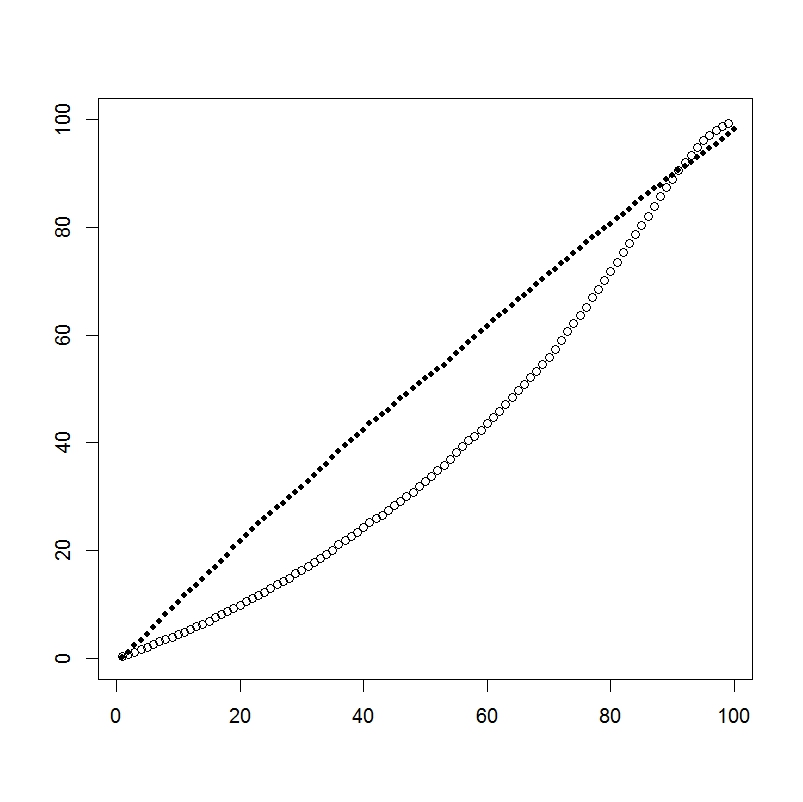}~\includegraphics[width=7.5cm]{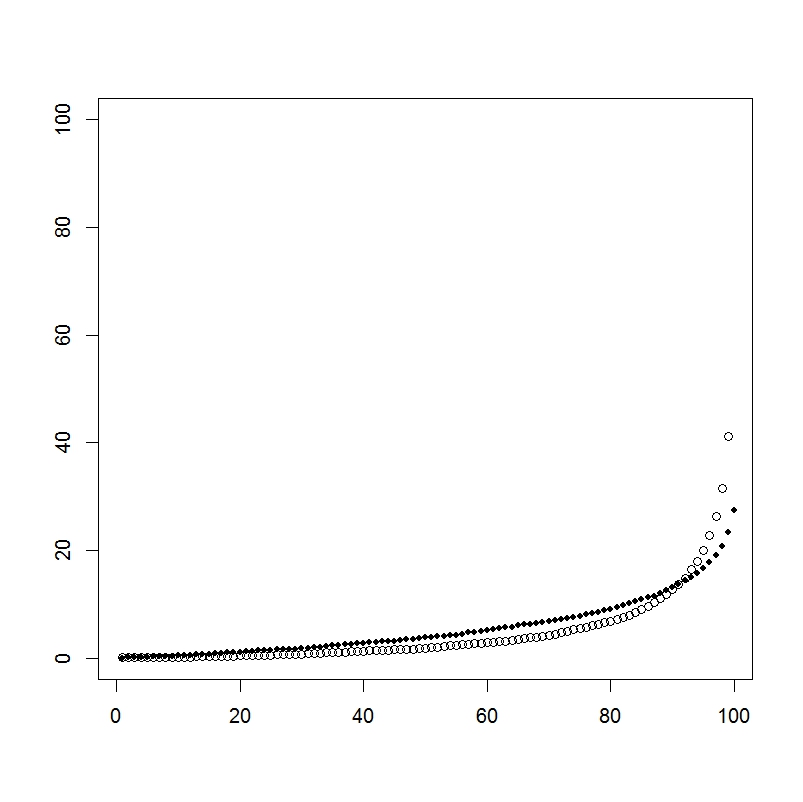}}
\caption{\label{Fig:2}Top: Histograms of \eqref{statV2tr} for $k_{1000}=275$ under hypothesis and alternative (right) and rescaled version comparing left hand side and limit law of \eqref{thmj2} and bootstrapped law (left); limit law density marked by dotted and bootstrapped by solid line. Bottom: Empirical size (left) and power (right) of the test by comparing empirical percentiles to ones of limit law under $H_0$ (light points) and to the bootstrapped percentiles (dark points).}
\end{figure}

We examine finite-sample properties of the proposed methods in a simulation study. First, consider $n=1000$ observations at regular times of \eqref{sm} with a stochastic semi-martingale volatility model
\begin{align}\label{volasim2}\sigma_t=\left(\int_0^t c\cdot\rho \,dW_s+\int_0^t \sqrt{1-\rho^2}\cdot c \,dW_s^{\bot}\right)\cdot v_t\,\end{align}
which fluctuates around a deterministic seasonality function
\begin{align}\label{volasim}v_t=1-0.2\sin\big(\tfrac{3}{4}\pi\,t),\,t\in[0,1]\,,\end{align}
with $c=0.1$ and $\rho=0.5$, where $W^{\bot}$ is a standard Brownian motion independent of $W$. We set the start value $X_0=4$ and the constant drift $a=0.1$. \eqref{volasim} mimics a realistic volatility shape with strong decrease after opening and slight increase before closing and the model poses an intricate setup to discriminate jumps from continuous motion based on the $n=1000$ discrete recordings of $X$. 

Under the local alternative, we add one jump of size $0.2$ at fixed time $t=2/3$ to $\sigma_t$, which equals the range of the continuous movement and shifts the volatility back to its maximum start value. This is in line with effects evoked by surprise elements from macroeconomic news in the financial context; see for instance Figure \ref{Fig:Intro}. Changing the time of the volatility jump does not affect the results substantially, though. One jump of $X$ at a uniformly drawn jump arrival time is implemented for the hypothesis and the alternative as well, and under the alternative $X$ additionally exhibits a common jump of $X$ and $\sigma$ at $t=2/3$. All these jumps are $N(0.5,0.1)$ distributed. 

Since $X$ comprises jumps, we apply the test statistic \eqref{statV2tr}. We focus on $V_{n,u_n}^*$ with overlapping blocks as it significantly outperforms the test with non-overlapping blocks. For the truncation sequence we set $u_n=\sqrt{2\log{(n)}}n^{-1/2}\approx 3.72\,n^{-1/2}$; see Remark \ref{trunc}. In all cases, we iterate $10,000$ Monte Carlo runs.

Figure \ref{Fig:2} illustrates simulation results for $k_{1000}=275$, but minor modifications of $k_{n}$ do not change the results substantially. Null and alternative are reasonably well distinguished, but the approximation of the limit law is somewhat imprecise, which is typical for limit theorems with extreme value distributions in finite-sample applications. Therefore, it is common practice in change-point literature to apply bootstrap procedures; see e.g.\,\cite{wuzhao2007}. We apply here a wild bootstrap-type procedure and use the statistics \eqref{statXtr} to pre-estimate the (in practice) unknown volatility and a smoothed version applying a linear filter with equal weights and $k_n$ lags to derive an estimated volatility shape. Then the statistics \eqref{statV2} are iteratively simulated, with $X$ being a discretized It\^{o} process without jumps and drift and with the pre-estimated volatility, to obtain critical values from the bootstrapped distribution. In light of the intricate setup, Figure \ref{Fig:2} confirms a high finite-sample accuracy. The density curve of the bootstrapped law in Figure \ref{Fig:2} is obtained from a kernel density estimate with R's standard bandwidth selection using Silverman's rule of thumb. While the size based on critical values from the limit law is not so precise, the bootstrapping works considerably well. The power for typical testing levels is reasonably large. 

If one is interested in the approximation of the critical values by the limit law, Figure \ref{Fig:4} provides more insight. Here, we have visualized the simulation results for $n=10,000$ observations, using different block widths. Setting $k_{10000}=500$, the limiting law is very well approximated and the power looks good as well. However, similar effects as before can be seen for a too large choice of $k_n$. Therefore it appears reasonable to use a bootstrap procedure in such a setting as well.

\begin{figure}[t]
\fbox{\parbox{.98\linewidth}{{\begin{center}Testing fractional OU log-volatility against jump\end{center}}
\includegraphics[width=7.5cm]{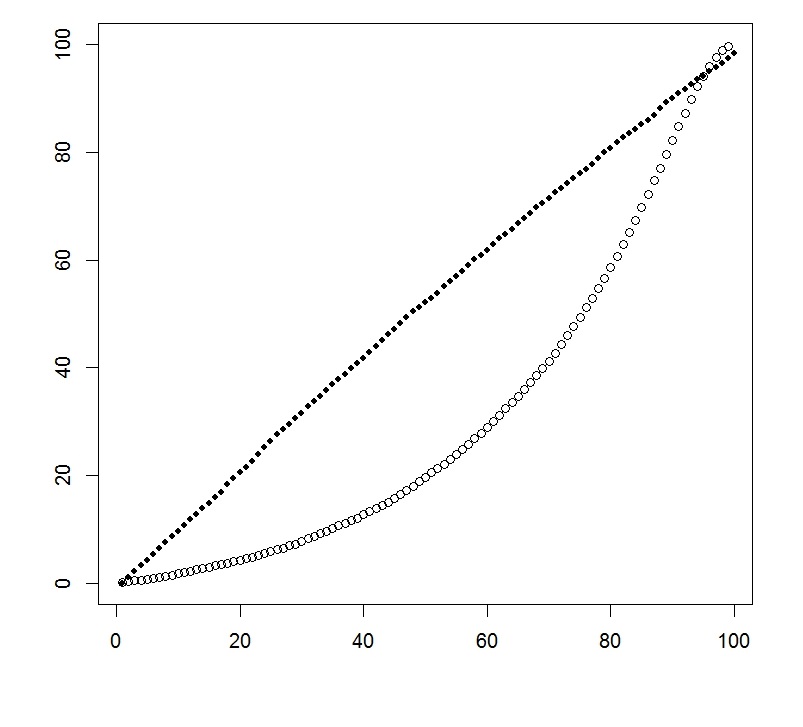}~\includegraphics[width=7.5cm]{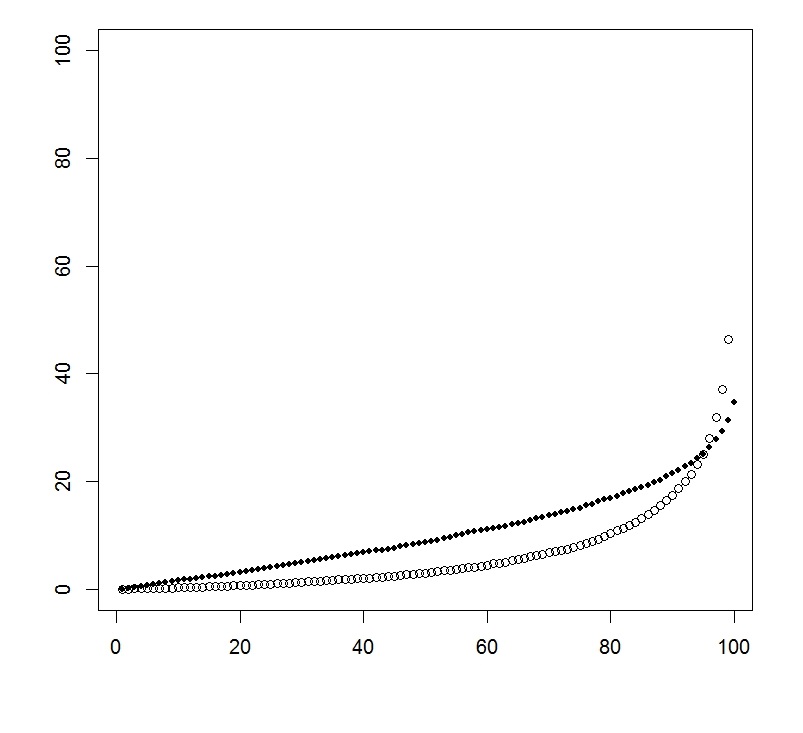}
}}
\fbox{\parbox{.98\linewidth}{{\begin{center}Testing semi-martingale against fractional volatility\end{center}}
\includegraphics[width=7.5cm]{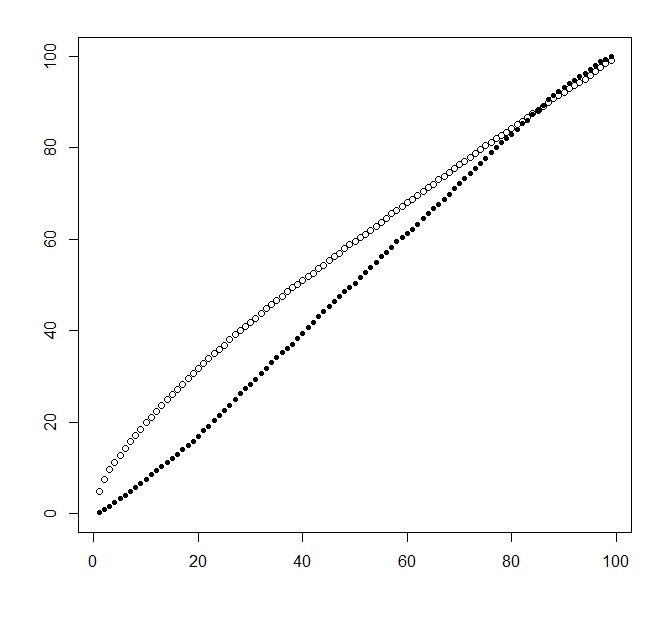}~\includegraphics[width=7.5cm]{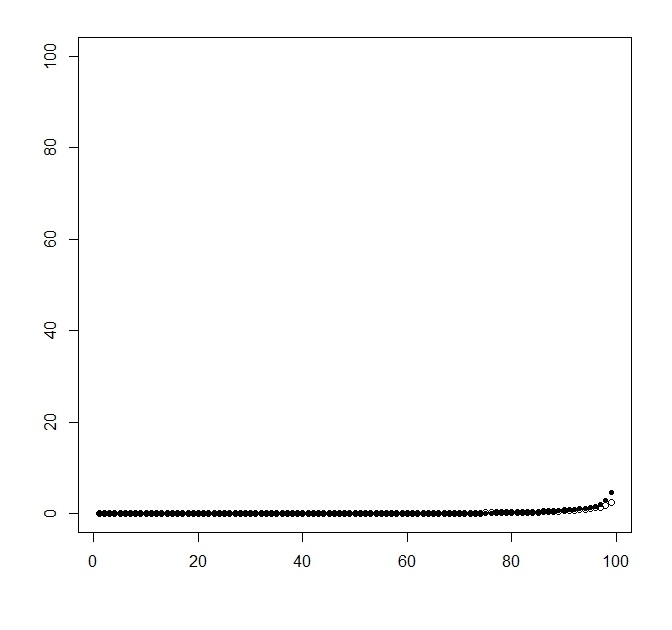}
}}
\caption{\label{Fig:3}Empirical size (left) and power (right) of the tests by comparing empirical percentiles to ones of the limit law under $H_0$ (light points) and to the bootstrapped ones (dark points). Top: Test \eqref{statV2tr} with $k_{1000}=275$; bottom: Test \eqref{testglobal}.}
\end{figure}

In the financial literature, many stochastic volatility models rely on fractional non-semi-martingale processes and recently in particular the interest in volatilities with small regularity has increased; see e.g.\ \cite{gatheral}. To see how our methods perform in such models, we modify our setup using the prominent fractional log-volatility model by \cite{comte}; i.e.\,replacing the semi-martingale above by a fractional OU-process 
\begin{align}\label{volasim3}d(\log(\tilde\sigma_t))=-0.1\log(\tilde\sigma_t)\,dt+0.1\,dB_t^H\,,\,\sigma_t=\tilde\sigma_t\cdot v_t\end{align}
with a fractional Brownian motion $(B_t^H)_{0\le t\le 1}$ with Hurst parameter $H$. The fractional process is implemented following Choleski's method with a code similar to the one in Appendix A.3 of \cite{fBm}.

The upper part of Figure \ref{Fig:3} presents the finite-sample precision of the test for a small Hurst parameter $H=0.2$. The outcomes are only slightly less accurate than for semi-martingale volatility and broadly give a similar picture. Therefore there is even finite-sample precision for detecting jumps in a fractional volatility process with a small Hurst parameter. Coming back to our introductory data example from Figure \ref{Fig:1} for intra-day prices on March 18th, 2009, the test rejects the null for both 3M and GE with $p$-values very close to zero. The point in time where the difference of adjacent statistics is maximized estimates the time of the structural change under the alternative. In both examples we find grid point 285, corresponding to 02:15 p.m. EST, as the estimated change-point.

\begin{figure}[t]
\includegraphics[width=7.5cm]{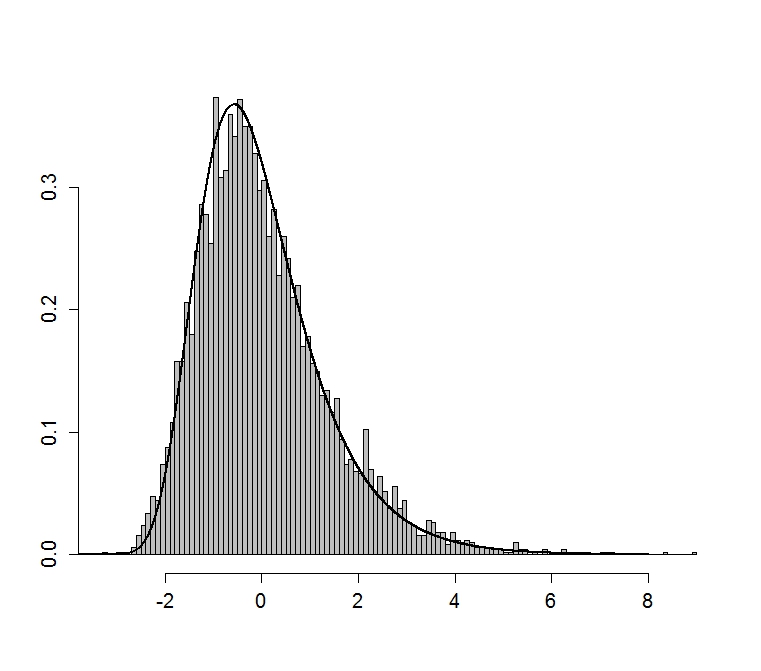}~\includegraphics[width=7.5cm]{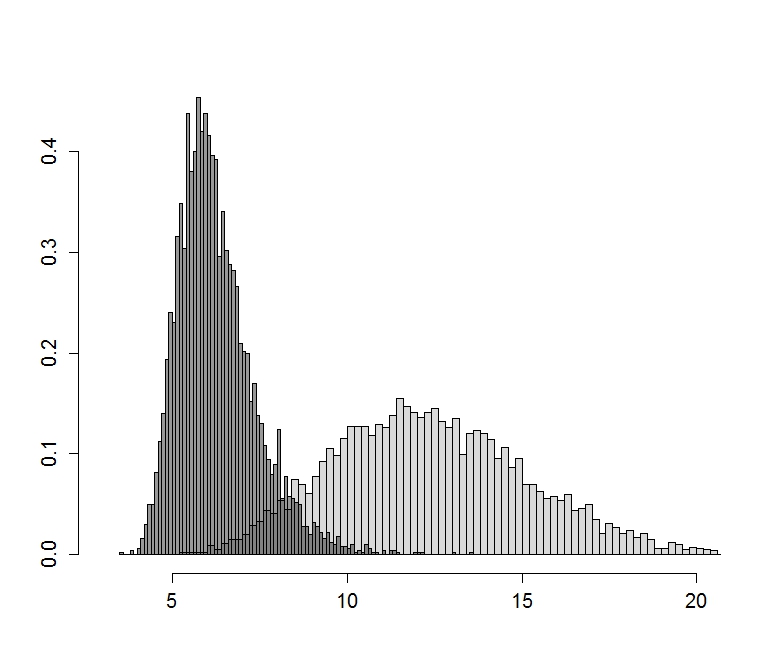}\\
\includegraphics[width=3.9cm]{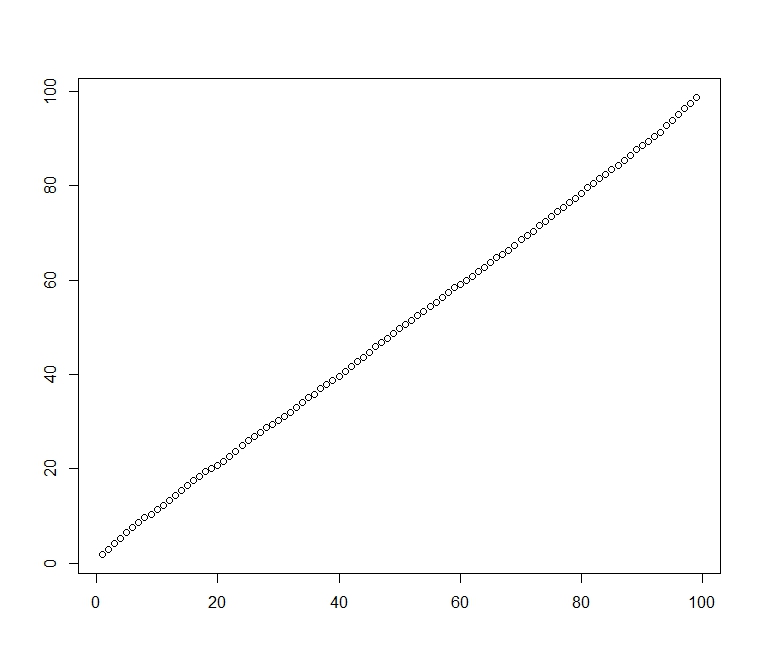}~\includegraphics[width=3.9cm]{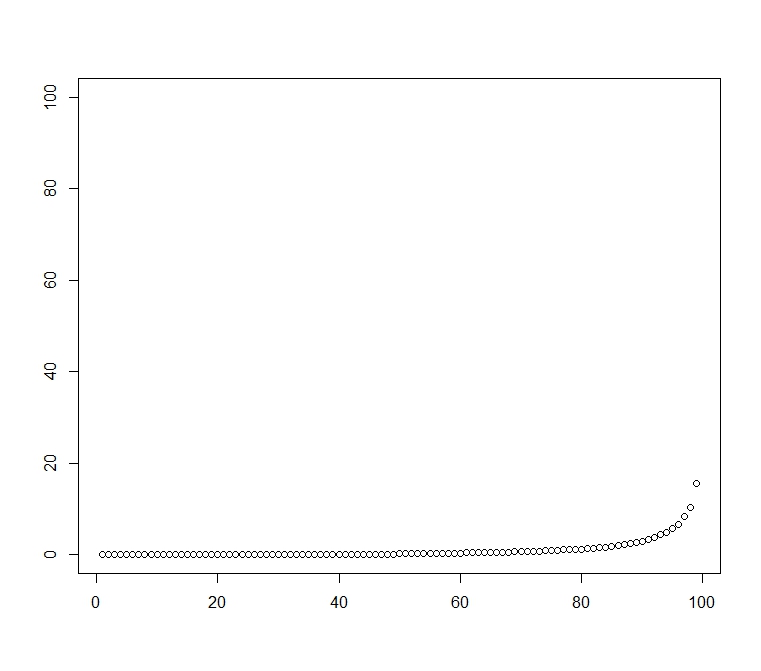}~\includegraphics[width=3.45cm]{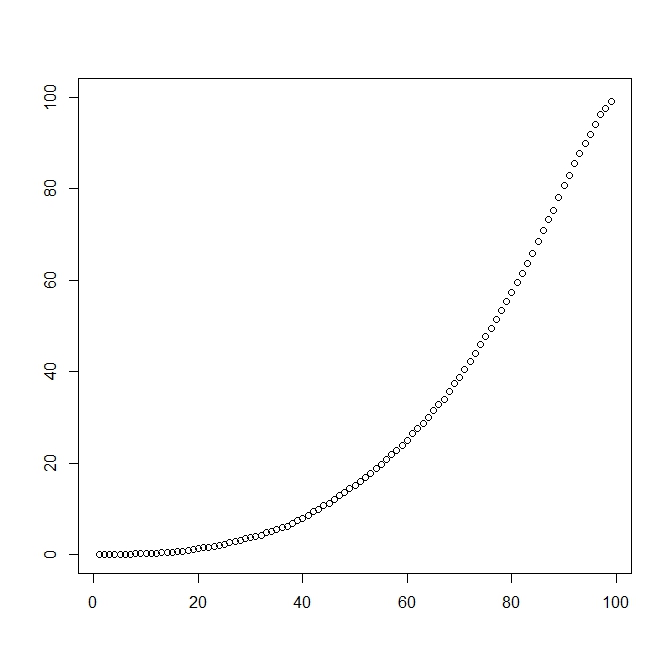}~\includegraphics[width=3.45cm]{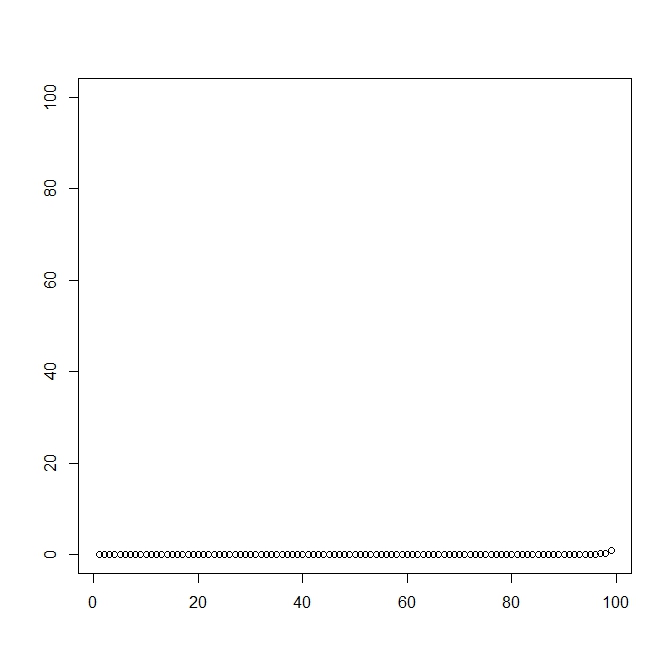}
\caption[caption]{\label{Fig:4}Finite sample precision of nonparametric test \eqref{statV2} for volatility model \eqref{volasim} against jump under the alternative for sample size $n=10,000$.\\ Top: Histograms of \eqref{statV2} for $k_{10000}=500$ under hypothesis and alternative (right) and rescaled version comparing left hand side and limit law of \eqref{thm1lt2} (left); limit law density marked by solid line.\\ Bottom: Empirical size and power of the test by comparing empirical percentiles to ones of limit law under $H_0$ for $k_{10000}=500$ (left) and $k_{10000}=1000$ (right).}
\begin{center}\includegraphics[width=7.3cm]{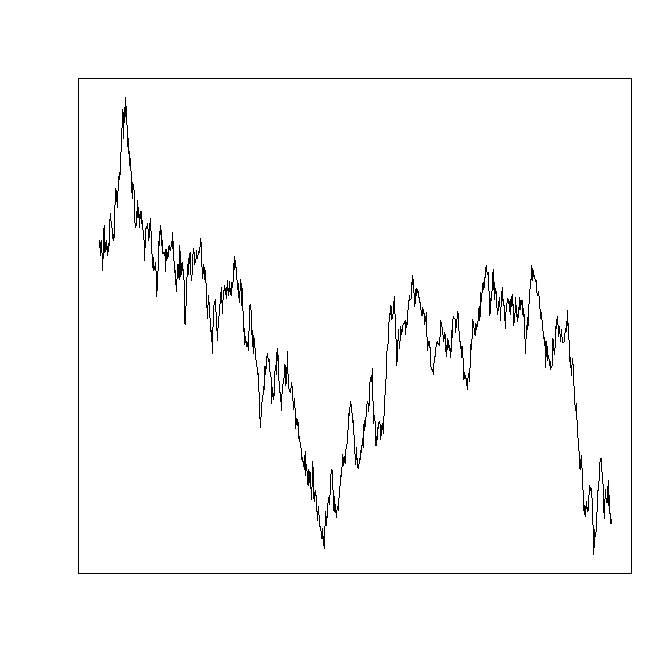}\hspace*{1.2cm}\includegraphics[width=7cm]{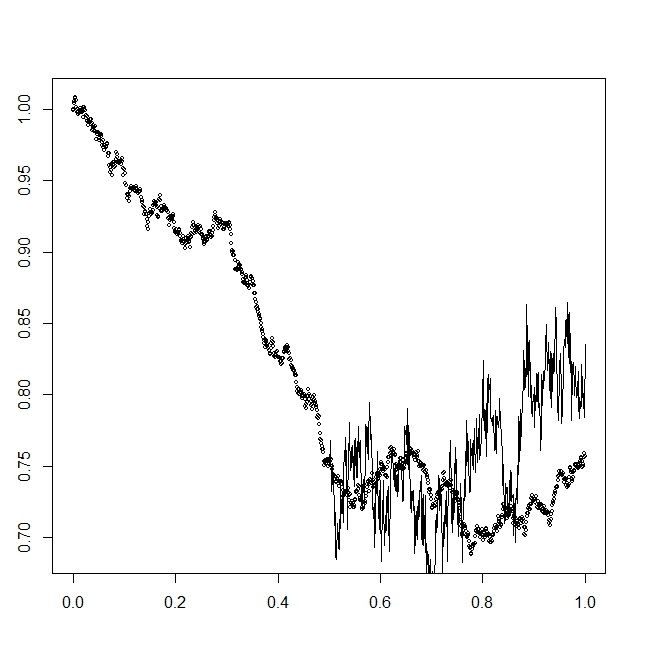}\end{center}
\caption{\label{Fig:5}Test of semi-martingale against fractional volatility: Typical paths of the log-price under the alternative, left, and of the volatility under null (dots) and under alternative (lines), right.}
\end{figure}

Finally, we examine the test for global changes based on \eqref{testglobal} in a simulation. Thereto, consider hypothesis \eqref{volasim2} for the volatility against the alternative that $(\sigma_t)_{t\ge\td}$ follows \eqref{volasim3} with Hurst parameter $H=0.15$, where under the alternative the change in smoothness happens at $\td=0.5$; see Figure \ref{Fig:5} for an illustration. The lower part of Figure \ref{Fig:3} confirms a remarkable finite sample performance of this test. Modifications of $H$ under the alternative do not affect the results substantially. Therefore, we expect that the methodology in Section \ref{sec:5} opens up valuable new ways for inference on volatility's regularity, useful for studies as the one in \cite{gatheral}.  

\clearpage
\appendix
\section{Proof of Theorem 3.2}
First, we reduce the proof of Theorem \ref{thm1} to Propositions \ref{prop1}-\ref{thm_mod_wu}. The main part in the analysis of $V_n$ from \eqref{statV} is to replace it by the statistic
\begin{align} \label{statU}
U_n =  \max_{i=0,\ldots,\lfloor n/k_n \rfloor -2} \vert Y_{n,i}/Y_{n,i+1} - 1 \vert,
\end{align}
in which the original statistics $RV_{n,i}$ from \eqref{statX} are approximated by
\begin{align} \label{statY}
Y_{n,i} = \frac{n}{k_n} \sum_{j=1}^{k_n} \sigma_{ik_n\Delta_n}^2 (\Delta_{ik_n + j}^n W)^2.
\end{align}
Up to different (random) factors in front, the maximum in $U_n$ is constructed from functionals of the i.i.d.\ increments of Brownian motion, which helps a lot in the derivation of its asymptotic behaviour. We start with a result on the approximation error due to replacing $V_n$ by $U_n$.

\begin{prop} \label{prop1}
Suppose that we are under the null. If Assumption \ref{assVola} and \eqref{assk} hold, then we have
\[
\sqrt{\log{(n)}\,k_n}\,\big(V_n - U_n \big)\pn 0.
\]
\end{prop}

Recall that the variables $Y_{n,i}$ are not only computed over different intervals, but come with different volatilities in front as well. In order to obtain a statistic which is independent of $\sigma$ let us define
\begin{align} \label{statYtil}
\WY_{n,i} = \frac{n}{k_n} \sum_{j=1}^{k_n} \sigma_{(i-1)k_n\Delta_n}^2 (\Delta_{ik_n + j}^n W)^2,
\end{align}
where the volatility factor is shifted in time now. Set then
\begin{align} \label{statUtil}
\WU_n =  \max_{i=0,\ldots,\lfloor n/k_n \rfloor -2} \vert Y_{n,i}/\WY_{n,i+1} - 1 \vert.
\end{align}

\begin{prop} \label{prop2}
Suppose that we are under the null. If Assumption \ref{assVola} and \eqref{assk} hold, then we have
\[
\sqrt{\log{(n)}\,k_n}\,\big(U_n - \WU_n \big)\pn 0.
\]
\end{prop}

In the final step, we replace $\WY_{n,i+1}$ in the denominator by its limit $\sigma_{ik_n\Delta_n}^2$. Set
\begin{align} \label{statVtil}
\WV_n =  \max_{i=0,\ldots,\lfloor n/k_n \rfloor -2} \Big\vert \frac{Y_{n,i}-\WY_{n,i+1}}{\sigma_{ik_n\Delta_n}^2} \Big\vert.
\end{align}

\begin{prop} \label{prop3}
Suppose that we are under the null. If condition \eqref{assk} is satisfied, then we have
\[\sqrt{\log{(n)}\,k_n}\,\big(
\WU_n - \WV_n\big) \pn 0.
\]
\end{prop}

From Propositions \ref{prop1} to \ref{prop3} we have $\sqrt{\log{(n)}\,k_n}\,\big(V_n - \WV_n\big) \pn 0$, while
\begin{align}\label{keystatisticno}
\WV_n =  \max_{i=0,\ldots,\lfloor n/k_n \rfloor -2} \Big \vert \frac 1{k_n} \sum_{j=1}^{k_n} (\sqrt{n}\Delta_{ik_n + j}^n W)^2  - \frac 1{k_n} \sum_{j=1}^{k_n} (\sqrt{n}\Delta_{(i+1)k_n + j}^n W)^2  \Big \vert
\,.\end{align}
This statistic corresponds to the statistic $D_n$ given in (13) of \cite{wuzhao2007}; see as well Proposition \ref{thm_mod_wu}. Precisely, after subtracting the mean on both sides above, their $(X_k)_{1\le k\le n}$ correspond to $\big((\sqrt{n}\Delta_{k}^n W)^2) -1\big)_{1\le k\le n}$, which forms an i.i.d\ sequence of shifted $\chi^2_1$-variables.

In the same fashion we can prove that the asymptotics of $V_n^*$ in \eqref{statV2} can be traced back to the statistics $D_n^*$ in (12) of \cite{wuzhao2007}, see Proposition \ref{thm_mod_wu}.
\begin{prop}\label{prop4}We have that $\sqrt{\log{(n)}\,k_n}\,\big(V_n^*-\tilde V_n^*\big)\pn 0$, with
\begin{align}\label{keystatistic}
\WV^*_n =  \max_{i=k_n,\ldots, n-k_n } \Big \vert \frac 1{k_n} \sum_{j=i+1}^{i+k_n} \big((\sqrt{n}\Delta_{j}^n W)^2  - (\sqrt{n}\Delta_{ j-k_n}^n W)^2 \big) \Big \vert
\,.\end{align}
\end{prop}
Theorem 1 of \cite{wuzhao2007} establishes limit theorems of the form \eqref{thm1lt} and \eqref{thm1lt2} under more restrictive assertions on $k_n$ than \eqref{assk}, as they consider the behavior for a class of weakly dependent random sequences $(X_k)_{k\ge 1}$. The next proposition provides a more specific limit theorem tailored to the asymptotic analysis of the statistics \eqref{keystatisticno} and \eqref{keystatistic}. In particular, instead of using the strong approximation theory under weak dependence from \cite{wu2006} employed by \cite{wuzhao2007} to prove their Theorem 1, we rely on classical bounds for the approximation error in the invariance principle for i.i.d.\,variables with existing moments. This is applicable in a more general setup with much smaller block lengths $k_n$.
\begin{prop}\label{thm_mod_wu}
Consider a sequence $(X_k)_{k\in\N}$ of i.i.d.\,random variables with $\Var\bigl[X_k\bigr] = \varsigma^2$ and $\E\bigl[|X_k|^p\bigr] < \infty$ for some $p \geq 4$. If
\begin{align}\label{eq_condi_kn}
k_n^{-p/2} n =  \KLEINO\bigl((\log (n))^{-p/2}\bigr),
\end{align}
then with $m_n=\lfloor n/k_n\rfloor$ the statistic
\begin{align*}
D_n^* = \frac{1}{k_n}\max_{k_n \leq i \leq n - k_n}\biggl|\sum_{j = i +1}^{k_n + i} X_j  - \sum_{j = i- k_n+1}^{i} X_j\biggr|.
\end{align*}
obeys the weak convergence:
\begin{align*}
\sqrt{\log (m_n)} (k_n^{1/2}\varsigma^{-1}) D_n^* - 2 \log{(m_n)} - \frac{1}{2}\log \log (m_n) - \log 3 \xrightarrow{w} V\,,
\end{align*}
where $V$ is distributed according to \eqref{V}. The statistic
\[
D_n=\max_{1\le i\le \lfloor n/k_n\rfloor-2}\big|\sum_{j=1}^{k_n}X_{ik_n+j}-X_{(i+1)k_n+j}\big|
\]
using non-overlapping blocks satisfies under the same assumptions
\begin{align*}\sqrt{\log(m_n)} \big(\big(k_n^{1/2}\varsigma^{-1}\big) D_{n} - [4 \log(m_n) - 2 \log(\log (m_n))]^{1/2}\big) \weak V\,.\end{align*}
\end{prop}

As all moments of the $\chi_1^2$ distribution exist and $k_n$ is at least of polynomial growth in $n$, Proposition \ref{thm_mod_wu} applied to \eqref{keystatisticno} and \eqref{keystatistic} implies Theorem \ref{thm1}. We start with the proof of Proposition \ref{thm_mod_wu} and then show by proving Propositions \ref{prop1}-\ref{prop4} that the preliminary reductions are in order.

\vspace*{.25cm}

\bf{Proof of Proposition \ref{thm_mod_wu}.}
\rm By a simple rescaling argument, we can restrict ourselves to the case $\Var\bigl[X_k\bigr] =1$. The Donsker-Prokhorov invariance principle guarantees weak convergence of partial sums of $(X_k)_{k\in\N}$, rescaled with $\sqrt{n}$, to the law of the standard Brownian motion as $n\rightarrow\infty$.
Let $\bigl(Z_j\bigr)_{j \in \N}$ be a sequence of centered i.i.d.\, Gaussian random variables with $\E\bigl[Z_j^2\bigr] = \E\bigl[X_j^2\bigr]=1$.
Observe that
\begin{align*}
\max_{k_n \leq i \leq n - k_n}\biggl|\sum_{j = i +1}^{k_n + i} \bigl(X_j - Z_j\bigr) - \sum_{j = i- k_n+1}^{i} \bigl(X_j - Z_j\bigr) \biggr| \leq 4 \max_{k_n \leq i \leq n}\biggl|\sum_{j = 1}^{i} \bigl(X_j - Z_j \bigr) \biggr|.
\end{align*}
We exploit the classical theory on bounds for the approximation error of partial sums of the above type associated with the invariance principle provided by the seminal works of \cite{kmt1}, \cite{kmt2}, \cite{zaitsev}, and related literature.
Let $(x_n)_{n\in\N}$ be a sequence with $x_n \geq 0$ for all $n$. According to Theorem 4 of \cite{kmt2} or equivalently (1.6) of \cite{sakhanenko} with Markov inequality, the sequence $\bigl(Z_j\bigr)_{j \in \N}$ can be constructed in such a manner that
\begin{align*}
\P\biggl(\max_{k_n \leq i \leq n }\biggl|\sum_{j = 1}^{i} \bigl(X_j - Z_j \bigr) \biggr| \geq x_n \biggr) \le C_1 \frac{1}{x_n^{p}} \sum_{j = 1}^n \E\bigl[|X_j|^{p}\bigr] \le C_2\frac{n}{x_n^{p}}\,,
\end{align*}
with constants $C_1,C_2$ which may depend on $p$.
Selecting $x_n = \sqrt{k_n} \delta_n$ with $\delta_n = (\log (n))^{-1/2}$, the conditions to apply Theorem 4 of \cite{kmt2} are in order and we get from condition \eqref{eq_condi_kn} and the above that
\begin{align}\label{eq_thm_wu_mod_2}
\max_{k_n \leq i \leq n}\biggl|\sum_{j = i +1}^{k_n + i} \bigl(X_j - Z_j\bigr) - \sum_{j = i- k_n+1}^{i} \bigl(X_j - Z_j\bigr) \biggr| = \KLEINO_{\P}\bigl(\sqrt{k_n} (\log (n))^{-1/2} \bigr).
\end{align}
Denote with $\mathbb{B}(k) = \sum_{j = 1}^k Z_j$ and define
\begin{align*}
H(u) = \bigl(\1(0 \leq u < 1) - \1(-1 < u < 0) \bigr)/\sqrt{2}.
\end{align*}
Then by \eqref{eq_thm_wu_mod_2}, it follows that
\begin{align*}
\frac{\sqrt{k_n}D_n^*}{\sqrt{2}} &= \frac{1}{\sqrt{2 k_n}} \max_{k_n \leq i \leq n - k_n}\bigl|\mathbb{B}(i + k_n) - 2 \mathbb{B}(i) + \mathbb{B}(i - k_n)\bigr| + \frac{\KLEINO_{\P}(1)}{\sqrt{\log (n)}}\\&= \frac{1}{\sqrt{k_n}} \sup_{s \in [k_n,n-k_n]}\biggl|\int_{\R} H\left(\frac{s-u}{k_n}\right)d\mathbb{B}(u)\biggr| + \frac{\mathcal{O}(R_n)}{\sqrt{k_n}} + \frac{\KLEINO_{\P}(1)}{\sqrt{\log (n)}},
\end{align*}
where $R_n = \sup\bigl\{|\mathbb{B}(u) - \mathbb{B}(u')|\,:\, u,u' \in [0,n], \, |u-u'| \leq 1 \bigr\} = \KLEINO_{\P}\bigl(\sqrt{\log (n)}\bigr)$ by standard properties of Brownian motion. Then, since $(\log (n))^6 = \KLEINO\bigl(k_n)$ by condition \eqref{eq_condi_kn}, we may apply the limit theorem from Lemma 2 in \cite{wuzhao2007} with $\alpha = 1$, $D_{H,1} = 3$ and $b_n =m_n^{-1}$ for $m_n = \lfloor n/k_n \rfloor$ ; see Definition 1 and Lemma 2 of \cite{wuzhao2007}. In the same manner, we can use that
\begin{align*}\sqrt{k_n}\Bigg(\sum_{j = i +1}^{k_n + i} X_j  -\hspace*{-.15cm} \sum_{j = i- k_n+1}^{i}\hspace*{-.15cm} X_j\hspace*{-.075cm}\Bigg)\hspace*{-.1cm} =\hspace*{-.1cm} \sqrt{k_n}\Bigg(\sum_{j = i +1}^{k_n + i}Z_j  - \hspace*{-.15cm} \sum_{j = i- k_n+1}^{i} \hspace*{-.15cm}Z_j\hspace*{-.075cm} \Bigg)\hspace*{-.1cm}+\KLEINO_{\P}\big(\hspace*{-.05cm}\sqrt{\log{\hspace*{-.05cm}(n)}}\big),\end{align*}
for $\E[X_j^2]=1$ with $\bigl(Z_j\bigr)_{j \in \N}$ again a sequence of centered i.i.d.\,Gaussian random variables. Lemma 1 of \cite{wuzhao2007} then ensures the limit theorem for non-overlapping blocks.
This completes the proof of Proposition \ref{thm_mod_wu}.\qed \\[.2cm]

\bf{Proof of Proposition \ref{prop1}.}
\rm First, a standard argument as e.g.\, laid out in Section 4.4.1 in \cite{JP} allows us to strengthen Assumption \ref{assVola} and to assume that all local conditions are in fact global. That is, we assume without loss of generality that $|a_s| \leq K$, $0 < \sigma_-^2 < \sigma_s^2 < K$ and $w_\delta(\sigma)_{1} \leq K \delta^\aalpha$ for a generic constant $K$.

Let $(a_i)_{i=1,\ldots,m}$ and $(b_i)_{i=1,\ldots,m}$ be arbitrary reals. Obviously, for an arbitrary $i$ we have
\[
|a_i| \leq |a_i - b_i| + |b_i| \leq \max_{i=1,\ldots,m} |a_i - b_i| + \max_{i=1,\ldots,m} |b_i|.
\]
Therefore the inequality holds with the left hand side replaced with $ \max_{i=1,\ldots,m} |a_i|$ as well. Applied to $V_n$ and $U_n$ it is simple to deduce
\begin{align} \label{VmU}
|V_n - U_n| \leq&  \max_{i=0,\ldots,\lfloor n/k_n \rfloor -2} \vert RV_{n,i}/RV_{n,i+1} - 1 - (Y_{n,i}/Y_{n,i+1} - 1) \vert  \\ \nonumber \leq&  \max_{i=0,\ldots,\lfloor n/k_n \rfloor -2} \Big \vert RV_{n,i} \left(\frac1{RV_{n,i+1}} - \frac1{Y_{n,i+1}}\right) \Big \vert +   \max_{i=0,\ldots,\lfloor n/k_n \rfloor -2} \Big \vert \frac{RV_{n,i} - Y_{n,i}}{Y_{n,i+1}} \Big \vert.
\end{align}
Let us begin with the second term on the right hand side above. For all $\epsilon > 0$ and all constants $D > 0$, we have
\begin{align}
&\notag\P\left( \max_{i=0,\ldots,\lfloor n/k_n \rfloor -2} \Big \vert \frac{\sqrt{k_n\log(n)}(RV_{n,i} - Y_{n,i})}{Y_{n,i+1}} \Big \vert > \epsilon \right) \\&\notag\leq \P\left( \max_{i=0,\ldots,\lfloor n/k_n \rfloor -2} \sqrt{k_n\log(n)}|RV_{n,i} - Y_{n,i}| \cdot \max_{i=0,\ldots,\lfloor n/k_n \rfloor -2} 1/\vert {Y_{n,i+1}} \vert > \epsilon \right) \\&  \leq\P\left( \max_{i=0,\ldots,\lfloor n/k_n \rfloor -2} \hspace*{-.05cm}\sqrt{k_n\log(n)}|RV_{n,i} - Y_{n,i}| > \frac{\epsilon}{D} \right)\hspace*{-.05cm} + \hspace*{-.05cm}\P\left(\max_{i=0,\ldots,\lfloor n/k_n \rfloor -2} 1/\vert {Y_{n,i+1}} \vert > D \right)\hspace*{-.05cm}.\label{twoprob}
\end{align}
To keep the notation readable, here and below we use standard probabilities and expectations without an extra indication that we are on the set $\Omega^c$.

Since we have $\sigma_t^2 \ge \sigma_-^2>0$, we can use the same arguments as in the proof of equation (22) in \cite{vett2012} to derive
\[
\P\left(\max_{i=0,\ldots,\lfloor n/k_n \rfloor -2} 1/\vert {Y_{n,i+1}} \vert > D \right) = \P\left(\min_{i=0,\ldots,\lfloor n/k_n \rfloor -2} \vert {Y_{n,i+1}} \vert < D^{-1} \right) \to 0
\]
with e.g.\, $D^{-1}=\sigma_-^2/2$. The intuition behind this result is that the probability of a mean of $k_n$ i.i.d.\ variables with all moments deviating too much from its expectation becomes exponentially small in $k_n$. A similar argument will be given in (\ref{step19}) later. Also, here we need that $k_n$ is of (at least) polynomial growth, which is included in \eqref{assk}.

On the other hand, using It\^o formula we obtain
\begin{align} \label{decomp}
& \sqrt{k_n\log(n)}(RV_{n,i} - Y_{n,i}) = \frac{n\sqrt{\log(n)}}{\sqrt{k_n}} \sum_{j=1}^{k_n} \big((\Delta_{ik_n+j}^n X)^2 - \sigma_{ik_n\Delta_n}^2 (\Delta_{ik_n + j}^n W)^2\big) \\ \nonumber &= \frac{n\sqrt{\log(n)}}{\sqrt{k_n}}\hspace*{-.05cm} \Bigg(\sum_{j=1}^{k_n} 2 \hspace*{-.1cm} \int_{(ik_n+j-1)\Delta_n}^{(ik_n+j)\Delta_n} (X_s \hspace*{-.05cm} -\hspace*{-.05cm}  X_{(ik_n+j-1)\Delta_n}) a_s \,ds +  \sum_{j=1}^{k_n} \int_{(ik_n+j-1)\Delta_n}^{(ik_n+j)\Delta_n} \hspace*{-.05cm} (\sigma_s^2 \hspace*{-.05cm} -\hspace*{-.05cm}  \sigma_{ik_n\Delta_n}^2)\, ds\hspace*{-.1cm} \Bigg) \\ &\quad + \frac{n\sqrt{\log(n)}}{\sqrt{k_n}} \sum_{j=1}^{k_n} 2 \hspace*{-.1cm} \int_{(ik_n+j-1)\Delta_n}^{(ik_n+j)\Delta_n} \hspace*{-.05cm} \big((X_s \hspace*{-.05cm} -\hspace*{-.05cm}  X_{(ik_n+j-1)\Delta_n}) \sigma_s \hspace*{-.05cm} - \hspace*{-.05cm} (W_s \hspace*{-.05cm} - \hspace*{-.05cm} W_{(ik_n+j-1)\Delta_n}) \sigma_{ik_n\Delta_n}^2\big)\, dW_s\,.  \nonumber
\end{align}
Using this decomposition, we split the discussion of \[\P\left( \max_{i=0,\ldots,\lfloor n/k_n \rfloor -2} \sqrt{k_n\log(n)}|RV_{n,i} - Y_{n,i}| > \epsilon/D \right)\] into three parts. For the first term, observe that
\begin{align*}
& \P\bigg( \max_{i=0,\ldots,\lfloor n/k_n \rfloor -2} \frac{n\sqrt{\log(n)}}{\sqrt{k_n}} \Big \vert \sum_{j=1}^{k_n} 2 \int_{(ik_n+j-1)\Delta_n}^{(ik_n+j)\Delta_n} (X_s - X_{(ik_n+j-1)\Delta_n}) a_s\, ds \Big \vert > \epsilon/(3D) \bigg) \\ &\leq \sum_{i=0}^{\lfloor n/k_n \rfloor -2} \P\bigg(\frac{n\sqrt{\log(n)}}{\sqrt{k_n}} \Big \vert \sum_{j=1}^{k_n} 2 \int_{(ik_n+j-1)\Delta_n}^{(ik_n+j)\Delta_n} (X_s - X_{(ik_n+j-1)\Delta_n}) a_s ds \Big \vert > \epsilon/(3D)\bigg) \\  &\leq (\epsilon/(3D))^{-r} \sum_{i=0}^{\lfloor n/k_n \rfloor -2} \E\bigg[\bigg \vert \frac{n\sqrt{\log(n)}}{\sqrt{k_n}} \sum_{j=1}^{k_n} 2 \int_{(ik_n+j-1)\Delta_n}^{(ik_n+j)\Delta_n} (X_s - X_{(ik_n+j-1)\Delta_n}) a_s\, ds \bigg \vert^r \bigg],
\end{align*}
for all integers $r$. Applying a standard bound based on Jensen's and Minkowski's inequalities yields
\begin{align*}
&\E\bigg[\bigg \vert \frac{n\sqrt{\log(n)}}{\sqrt{k_n}} 2 \int_{(ik_n+j-1)\Delta_n}^{(ik_n+j)\Delta_n} (X_s - X_{(ik_n+j-1)\Delta_n}) a_s\, ds \bigg \vert^r \bigg] \\
   & \leq K_r \Bigg(\frac{n\sqrt{\log(n)}}{\sqrt{k_n}} \int_{(ik_n+j-1)\Delta_n}^{(ik_n+j)\Delta_n} \E[|X_{s} - X_{(ik_n+j-1)\Delta_n}|^r]^{1/r} ds \Bigg)^r\,.
\end{align*}
$K_r$ here and below denotes a generic constant depending on $r$. Burkholder-Davis-Gundy inequality gives for any $s\in[(ik_n+j-1)\Delta_n,(ik_n+j)\Delta_n]$:
\begin{align*}
\E[|X_{s} - X_{(ik_n+j-1)\Delta_n}|^r] &\leq K_r n^{-r/2}\,,\\
\E\Big[\Big \vert \frac{n\sqrt{\log(n)}}{\sqrt{k_n}} 2 \int_{(ik_n+j-1)\Delta_n}^{(ik_n+j)\Delta_n} (X_s - X_{(ik_n+j-1)\Delta_n}) a_s \,ds \Big \vert^r \Big] &\leq K_r (nk_n)^{-r/2}\log^{r/2}(n)\,.
\end{align*}
We conclude that
\begin{align}
& \label{mink1}\P\bigg( \max_{i=0,\ldots,\lfloor n/k_n \rfloor -2} \frac{n\sqrt{\log(n)}}{\sqrt{k_n}} \Big \vert \sum_{j=1}^{k_n} 2 \int_{(ik_n+j-1)\Delta_n}^{(ik_n+j)\Delta_n} (X_s - X_{(ik_n+j-1)\Delta_n}) a_s \,ds \big \vert > \epsilon/(3D) \bigg) \\ \notag&\quad \leq (\epsilon/(3D))^{-r} \lfloor n/k_n \rfloor K_r k_n^{r/2} n^{-r/2}\log^{r/2}(n) \to 0
\end{align}
for $r>2$ arbitrary. Regarding the second term in (\ref{decomp}), on $\Omega^c$ we have
\begin{align*}
& \max_{i=0,\ldots,\lfloor n/k_n \rfloor -2} \Big \vert \frac{n\sqrt{\log(n)}}{\sqrt{k_n}} \sum_{j=1}^{k_n} \int_{(ik_n+j-1)\Delta_n}^{(ik_n+j)\Delta_n} (\sigma_s^2 - \sigma_{ik_n\Delta_n}^2) \,ds \Big \vert \\ & \leq\max_{i=0,\ldots,\lfloor n/k_n \rfloor -2} \frac{n\sqrt{\log(n)}}{\sqrt{k_n}} \sum_{j=1}^{k_n} \int_{(ik_n+j-1)\Delta_n}^{(ik_n+j)\Delta_n} |\sigma_s^2 - \sigma_{ik_n\Delta_n}^2| ds \\
&\leq \sqrt{k_n} w_{k_n\Delta_n}(\sigma)_{1}\sqrt{\log(n)} \leq K \sqrt{k_n} (k_n\Delta_n)^{\aalpha}\sqrt{\log(n)}\,,
\end{align*}
which converges to zero by \eqref{assk}. Observe that addends above involve interlacing time intervals such that for $\sigma$ a continuous It\^{o} semi-martingale the bound above applies with $\aalpha=1/2$ and is sharp.

Finally, we have the further decomposition
\begin{align} \label{decomp2}
&(X_s - X_{(ik_n+j-1)\Delta_n}) \sigma_s - (W_s - W_{(ik_n+j-1)\Delta_n}) \sigma_{ik_n\Delta_n}^2 =\sigma_s \int_{(ik_n+j-1)\Delta_n}^s a_u \,du \\
&\quad\quad + (\sigma_s - \sigma_{ik_n\Delta_n}) \int_{(ik_n+j-1)\Delta_n}^s
\sigma_u \,  dW_u + \sigma_{ik_n \Delta_n} \int_{(ik_n+j-1)\Delta_n}^s (\sigma_u - \sigma_{ik_n\Delta_n})\, dW_u\,. \nonumber
\end{align}
We proceed in a similar way as above:
\begin{align} \nonumber
& \P\bigg( \max_{i=0,\ldots,\lfloor n/k_n \rfloor -2} \frac{n\sqrt{\log(n)}}{\sqrt{k_n}} \Big \vert \sum_{j=1}^{k_n} 2 \int_{(ik_n+j-1)\Delta_n}^{(ik_n+j)\Delta_n} \sigma_s \int_{(ik_n+j-1)\Delta_n}^s a_u du dW_s \Big \vert > \epsilon/(9D) \bigg) \\ & \leq(\epsilon/(9D))^{-r} \sum_{i=0}^{\lfloor n/k_n \rfloor -2} \E\bigg[\Big \vert \frac{n\sqrt{\log(n)}}{\sqrt{k_n}} \sum_{j=1}^{k_n} 2 \int_{(ik_n+j-1)\Delta_n}^{(ik_n+j)\Delta_n} \sigma_s \int_{(ik_n+j-1)\Delta_n}^s a_u \,du \,  dW_s \Big \vert^r \bigg]\,. \label{step17}
\end{align}
Precisely, let $r=2m$ and set
\[
c_s = \sum_{j=1}^{k_n} \sigma_s \int_{(ik_n+j-1)\Delta_n}^s a_u\, du 1_{[(ik_n+j-1)\Delta_n, (ik_n+j)\Delta_n)}(s)\,.
\]
Then we have in a similar way as before
\begin{align*}
& \E\Big[\Big \vert \frac{n\sqrt{\log(n)}}{\sqrt{k_n}} \sum_{j=1}^{k_n} 2 \int_{(ik_n+j-1)\Delta_n}^{(ik_n+j)\Delta_n} \sigma_s \int_{(ik_n+j-1)\Delta_n}^s a_u \,du \,  dW_s \Big \vert^{2m} \Big] \\ &= 2^{2m} \Big(\frac{n\sqrt{\log(n)}}{\sqrt{k_n}}\Big)^{2m} \E\Big[\Big \vert \int_{ik_n\Delta_n}^{(i+1)k_n\Delta_n} c_s \,dW_s \Big \vert^{2m} \Big] \\ &=  2^{2m}
\Big(\frac{n\sqrt{\log(n)}}{\sqrt{k_n}}\Big)^{2m} \E\Big[\Big( \int_{ik_n\Delta_n}^{(i+1)k_n\Delta_n} c_s^2 ds \Big)^{m} \Big] \\ & \leq K_m \Big(\frac{n\sqrt{\log(n)}}{\sqrt{k_n}}\Big)^{2m} \Big(  \int_{ik_n\Delta_n}^{(i+1)k_n\Delta_n} \E[c_s^{2m}]^{1/m} ds \Big)^{m}.
\end{align*}
With
\[
\E[c_s^{2m}] = \sum_{j=1}^{k_n} \E\Big[\sigma_s^{2m} \Big(\int_{(ik_n+j-1)\Delta_n}^s a_u du\Big)^{2m}\Big] 1_{[(ik_n+j-1)\Delta_n, (ik_n+j)\Delta_n)}(s) \leq K_m \Delta_n^{2m}\,,
\]
we obtain
\begin{align}\label{mink2}
\Big(\frac{n\sqrt{\log(n)}}{\sqrt{k_n}}\Big)^{2m} \Big(  \int_{ik_n\Delta_n}^{(i+1)k_n\Delta_n} \E[c_s^{2m}]^{1/m} ds \Big)^{m} &\leq K_m \Big(\frac{n\sqrt{\log(n)}}{\sqrt{k_n}}\Big)^{2m} (k_n \Delta_n^3)^m\\ &= K_m \Delta_n^{m}\log^{2m}(n)\notag .
\end{align}
By choosing $m$ large enough, the term in \eqref{step17} converges to zero. Similarly,
\begin{align*}
& \E\Big[\Big \vert \frac{n\sqrt{\log(n)}}{\sqrt{k_n}} \sum_{j=1}^{k_n} 2 \int_{(ik_n+j-1)\Delta_n}^{(ik_n+j)\Delta_n}(\sigma_s - \sigma_{ik_n\Delta_n}) \int_{(ik_n+j-1)\Delta_n}^s
\sigma_u\, dW_u  dW_s \Big \vert^{2m} \Big] \\ & \leq K_m \Big(\frac{n\sqrt{\log(n)}}{\sqrt{k_n}}\Big)^{2m} \Big(  \int_{ik_n\Delta_n}^{(i+1)k_n\Delta_n} \Big( \sum_{j=1}^{k_n} \E \Big[(\sigma_s - \sigma_{ik_n\Delta_n})^{2m} \\ &\hspace*{4cm} \times \Big(\int_{(ik_n+j-1)\Delta_n}^s
\sigma_u dW_u\Big)^{2m} 1_{[(ik_n+j-1)\Delta_n, (ik_n+j)\Delta_n)}(s)\Big] \Big)^{1/m}  ds \Big)^{m} \\ & \leq K_m \Big(\frac{n\sqrt{\log(n)}}{\sqrt{k_n}}\Big)^{2m} \Big(  \int_{ik_n\Delta_n}^{(i+1)k_n\Delta_n} \Big( \sum_{j=1}^{k_n} \E \Big[w_{k_n\Delta_n}(\sigma)_{1}^{2m}\\ &\hspace*{4cm} \times \Big(\int_{(ik_n+j-1)\Delta_n}^s
\sigma_u dW_u\Big)^{2m} 1_{[(ik_n+j-1)\Delta_n, (ik_n+j)\Delta_n)}(s)\Big] \Big)^{1/m}  ds \Big)^{m}  \\ & \leq K_m \Big(\frac{n\sqrt{\log(n)}}{\sqrt{k_n}}\Big)^{2m} \Big(  \int_{ik_n\Delta_n}^{(i+1)k_n\Delta_n} \big( (k_n\Delta_n)^{2m\aalpha} \Delta_n^m \big)^{1/m} \,ds \Big)^{m} \leq K_m (k_n\Delta_n)^{2m\aalpha}.
\end{align*}
The same upper bound is obtained for the third term in (\ref{decomp2}). Again, choosing $m$ large enough yields convergence to zero. Altogether, we conclude
\[
\P\left( \max_{i=0,\ldots,\lfloor n/k_n \rfloor -2} \sqrt{k_n\log(n)}|RV_{n,i} - Y_{n,i}| > \epsilon/D \right) \to 0,
\]
and we are done with the second term on the right hand side of \eqref{VmU}.
Next, consider the first term on the right hand side of \eqref{VmU}. For any $\epsilon > 0$ and any $D>0$:
\begin{align*}
&\P\left( \max_{i=0,\ldots,\lfloor n/k_n \rfloor -2} \sqrt{k_n\log(n)}\big \vert RV_{n,i} \left(\frac1{RV_{n,i+1}} - \frac1{Y_{n,i+1}}\right) \big \vert > \epsilon \right) \\ & \leq\P\left( \max_{i=0,\ldots,\lfloor n/k_n \rfloor -2} \sqrt{k_n\log(n)}|RV_{n,i} (Y_{n,i+1} - RV_{n,i+1})| > \epsilon/D \right)\\
&\hspace*{7cm}  + \P\left(\max_{i=0,\ldots,\lfloor n/k_n \rfloor -2} 1/\vert {Y_{n,i+1}RV_{n,i+1}} \vert > D \right)\,.
\end{align*}
Observe that
\begin{align*}
& \P\left(\min_{i=0,\ldots,\lfloor n/k_n \rfloor -2} \vert {Y_{n,i+1}RV_{n,i+1}} \vert < D^{-1} \right) \\ & \leq\P\left(\min_{i=0,\ldots,\lfloor n/k_n \rfloor -2} \vert Y_{n,i+1} \vert < D^{-1/2} \right) + \P\left(\min_{i=0,\ldots,\lfloor n/k_n \rfloor -2} \vert {RV_{n,i+1}} \vert < D^{-1/2} \right) \\ & \leq\P\left(\min_{i=0,\ldots,\lfloor n/k_n \rfloor -2} \vert Y_{n,i+1} \vert < D^{-1/2} \right) + \P\left(\min_{i=0,\ldots,\lfloor n/k_n \rfloor -2} \vert {Y_{n,i+1}} \vert < 2D^{-1/2} \right) \\&\hspace*{6.4cm}+ \P\left(\max_{i=0,\ldots,\lfloor n/k_n \rfloor -2} \vert {RV_{n,i+1}} - Y_{n,i+1} \vert > D^{-1/2} \right).
\end{align*}
All three terms on the right hand side have already been discussed above for an appropriate choice of $D$, the latter term even with an additional factor $\sqrt{k_n \log(n)}$. Similarly, for all $\Gamma > 0$ we have
\begin{align*}
& \P\left( \max_{i=0,\ldots,\lfloor n/k_n \rfloor -2} \sqrt{k_n\log(n)}|RV_{n,i} (Y_{n,i+1} - RV_{n,i+1})| > \epsilon/D \right) \\ &\leq \P\left( \max_{i=0,\ldots,\lfloor n/k_n \rfloor -2} |RV_{n,i}| > \Gamma \right) + \P\left( \max_{i=0,\ldots,\lfloor n/k_n \rfloor -2} \sqrt{k_n\log(n)}|Y_{n,i+1} - RV_{n,i+1}| > \epsilon/(D\Gamma) \right).
\end{align*}
Here we only have to focus on the first term, for which we use
\begin{align} \label{step18}
& \P\left( \max_{i=0,\ldots,\lfloor n/k_n \rfloor -2} |RV_{n,i}| > \Gamma \right) \\ & \leq\P\left( \max_{i=0,\ldots,\lfloor n/k_n \rfloor -2} |Y_{n,i}| > \Gamma/2 \right) + \P\left( \max_{i=0,\ldots,\lfloor n/k_n \rfloor -2} |Y_{n,i+1} - RV_{n,i+1}| > \Gamma/2 \right).
\nonumber
\end{align}
The same arguments which were leading to equation (22) in \cite{vett2012} show that the first probability becomes arbitrarily small for large enough $\Gamma$, this time because we may assume $\sigma$ is bounded from above. The second probability has already been discussed above. \qed \\[.2cm]

\bf{Proof of Proposition \ref{prop2}.} \rm
We have to show convergence in probability to zero of
\[
 \sqrt{k_n\log(n)} \max_{i=0,\ldots,\lfloor n/k_n \rfloor -2} \Big \vert \frac{Y_{n,i}}{Y_{n,i+1}} - \frac{Y_{n,i}}{\WY_{n,i+1}} \Big \vert =  \max_{i=0,\ldots,\lfloor n/k_n \rfloor -2} \Big \vert \frac{\sqrt{k_n\log(n)} Y_{n,i}(\WY_{n,i+1}-Y_{n,i+1})}{Y_{n,i+1}\WY_{n,i+1}} \Big \vert.
\]
Using equation (22) in \cite{vett2012} again, we may focus on the numerator above only, and for the same reason as in \eqref{step18} it suffices to prove convergence to zero of
\begin{align} \label{step20}
& \P\left(\max_{i=0,\ldots,\lfloor n/k_n \rfloor -2} \sqrt{k_n\log(n)}|\WY_{n,i+1}-Y_{n,i+1}| > \epsilon \right) \\ \nonumber
&= \P\bigg(\max_{i=0,\ldots,\lfloor n/k_n \rfloor -2} \sqrt{k_n\log(n)} \big|\sigma_{ik_n\Delta_n}^2 - \sigma_{(i+1)k_n\Delta_n}^2\big| \bigg \vert \frac n{k_n} \sum_{j=1}^{k_n} (\Delta_{(i+1)k_n + j}^n W)^2 \bigg \vert > \epsilon \bigg) \\ \nonumber & \leq\P\left(\max_{i=0,\ldots,\lfloor n/k_n \rfloor -2} \sqrt{k_n\log(n)} \big|\sigma_{ik_n\Delta_n}^2 - \sigma_{(i+1)k_n\Delta_n}^2\big| > \epsilon/2 \right)\\
 &\nonumber\hspace*{5cm}+ \P\left(\max_{i=0,\ldots,\lfloor n/k_n \rfloor -2} \Big \vert \frac n{k_n} \sum_{j=1}^{k_n} (\Delta_{(i+1)k_n + j}^n W)^2 \Big \vert > 2 \right)
\end{align}
for all $\epsilon > 0$. Regarding the first quantity, recall that on $\Omega^c$ by \eqref{assk}
\begin{align*}
\max_{i=0,\ldots,\lfloor n/k_n \rfloor -2} \sqrt{k_n\log(n)} |\sigma_{ik_n\Delta_n}^2 - \sigma_{(i+1)k_n\Delta_n}^2| &\leq \sqrt{k_n\log(n)} w_{k_n\Delta_n}(\sigma)_{1} \\ &\leq K \sqrt{k_n} (k_n\Delta_n)^{\aalpha}\sqrt{\log(n)} \to 0\,.
\end{align*}
On the other hand,
\begin{align} \label{step19}
&\P\Big(\max_{i=0,\ldots,\lfloor n/k_n \rfloor -2} \Big \vert \frac{n\sqrt{\log(n)}}{k_n} \sum_{j=1}^{k_n} (\Delta_{(i+1)k_n + j}^n W)^2 \Big \vert >  2 \Big)
\\ \nonumber &\leq\sum_{i=0}^{\lfloor n/k_n \rfloor -2}\hspace*{-.1cm} \P \Big(\Big \vert \frac{n\sqrt{\log(n)}}{k_n} \sum_{j=1}^{k_n} (\Delta_{(i+1)k_n + j}^n W)^2 \Big \vert  >  2 \Big) \\ \nonumber &\leq \sum_{i=0}^{\lfloor n/k_n \rfloor -2} \P \Big(\Big \vert\frac{\sqrt{\log(n)}}{k_n} \sum_{j=1}^{k_n} \big((\sqrt{n}\Delta_{(i+1)k_n + j}^n W)^2 - 1\big) \Big \vert > 1 \Big)  \\ \nonumber &\leq \sum_{i=0}^{\lfloor n/k_n \rfloor -2} \E\Big[ \Big \vert \frac{\sqrt{\log(n)}}{k_n} \sum_{j=1}^{k_n} \big((\sqrt{n}\Delta_{(i+1)k_n + j}^n W)^2 - 1\big) \Big \vert^{2m} \Big]
\end{align}
for all integers $m$. Due to the i.i.d.\ structure, the latter term is bounded by $K_m (n/k_n) k_n^{-m}\log^m(n)$, which converges to zero for $m$ large enough. \qed\\[.2cm]

\bf{Proof of Proposition \ref{prop3}.} \rm
We have to show convergence to zero in probability of
\begin{align*}
&\sqrt{k_n\log(n)} \max_{i=0,\ldots,\lfloor n/k_n \rfloor -2} \Big \vert \frac{Y_{n,i}-\WY_{n,i+1}}{\WY_{n,i+1}} - \frac{Y_{n,i}-\WY_{n,i+1}}{\sigma_{ik_n\Delta_n}^2} \Big \vert\\
 &\hspace*{4cm}= \sqrt{k_n\log(n)} \max_{i=0,\ldots,\lfloor n/k_n \rfloor -2} \Big \vert \frac{(Y_{n,i}-\WY_{n,i+1})(\WY_{n,i+1}-\sigma_{ik_n\Delta_n}^2)}{\WY_{n,i+1}\sigma_{ik_n\Delta_n}^2}   \Big \vert\,.
\end{align*}
It is sufficient to focus on the numerator, and we discuss two terms separately, using
\begin{align*}
& \P\left(\max_{i=0,\ldots,\lfloor n/k_n \rfloor -2} \sqrt{k_n\log(n)} \vert (Y_{n,i}-\WY_{n,i+1})(\WY_{n,i+1}-\sigma_{ik_n\Delta_n}^2) \vert > \epsilon \right) \\ & \leq \P\left(\max_{i=0,\ldots,\lfloor n/k_n \rfloor -2} \sqrt{k_n\log(n)} \vert Y_{n,i}-\WY_{n,i+1} \vert > \sqrt{\epsilon} \right)\\
&\hspace*{3.5cm} + \P\left(\max_{i=0,\ldots,\lfloor n/k_n \rfloor -2} \vert \WY_{n,i+1}-\sigma_{ik_n\Delta_n}^2 \vert > \sqrt{\epsilon} \right)\,.
\end{align*}
The first term has already been discussed in (\ref{step20}), while
\begin{align*}
& \P\left(\max_{i=0,\ldots,\lfloor n/k_n \rfloor -2} \vert \WY_{n,i+1}-\sigma_{ik_n\Delta_n}^2 \vert > \sqrt{\epsilon} \right)  \\ &= \P \left(\max_{i=0,\ldots,\lfloor n/k_n \rfloor -2} \sigma_{ik_n\Delta_n}^2 \Big \vert \frac 1{k_n} \sum_{j=1}^{k_n} ((\sqrt{n}\Delta_{(i+1)k_n + j}^n W)^2 - 1) \Big \vert > \sqrt \epsilon \right) \\&  \leq\P \left(\max_{i=0,\ldots,\lfloor n/k_n \rfloor -2} \Big \vert \frac 1{k_n} \sum_{j=1}^{k_n} ((\sqrt{n}\Delta_{(i+1)k_n + j}^n W)^2 - 1) \Big \vert > \sqrt \epsilon/K \right)
\end{align*}
using $\sigma^2 \leq K$. The claim follows from (\ref{step19}). \qed\\[.2cm]

\bf{Proof of Proposition \ref{prop4}.} \rm
For the test statistic $V_n^*$ from \eqref{statV2} our proof follows the same stages as the one for $V_n$ via Propositions \ref{prop1}, \ref{prop2} and \ref{prop3} above. We start proving $$\sqrt{k_n\,\log(n)}\,\big(V_n^*-U_n^*\big)\pn 0$$ for
\begin{align*}U_n^*=\max_{i=k_n,\ldots,n-k_n}\Bigg\vert \frac{\frac{n}{k_n}\sum_{j=i-k_n+1}^{i}\sigma_{(i-k_n)\Delta_n}^2(\Delta_j^nW)^2}{\frac{n}{k_n}\sum_{j=i+1}^{i+k_n}\sigma_{i\Delta_n}^2(\Delta_j^nW)^2}-1\Bigg\vert\,.\end{align*}
Similar to \eqref{VmU}, we find that
\begin{align*}
|V_n^*-U_n^*|\le &\max_{i=k_n,\ldots,n-k_n}\hspace*{-.05cm}\Bigg\vert  \sum_{j=i-k_n+1}^{i}\hspace*{-.05cm}(\Delta_j^nX)^2\bigg(\Big( \sum_{j=i+1}^{i+k_n}(\Delta_j^n X)^2\Big)^{-1}\hspace*{-.05cm}-\Big(\sum_{j=i+1}^{i+k_n}\sigma_{i\Delta_n}^2(\Delta_j^nW)^2\Big)^{-1}\bigg)\Bigg\vert\\
&+\max_{i=k_n,\ldots,n-k_n}\Bigg\vert\frac{\frac{n}{k_n}\sum_{j=i-k_n+1}^{i}\big((\Delta_j^nX)^2-\sigma_{(i-k_n)\Delta_n}^2(\Delta_j^nW)^2\big)}{\frac{n}{k_n}\sum_{j=i+1}^{i+k_n}\sigma_{i\Delta_n}^2(\Delta_j^nW)^2}\Bigg\vert\,.\end{align*}
Following an inequality analogous to \eqref{twoprob}, the key step is to show that
\begin{align}\label{keyprop4}\P\Big(\sqrt{\log(n)k_n}\max_{i=k_n,\ldots,n-k_n} \Big\vert \frac{n}{k_n}\sum_{j=i-k_n+1}^{i}\big((\Delta_j^nX)^2-\sigma_{(i-k_n)\Delta_n}^2(\Delta_j^nW)^2\big) \Big\vert>\epsilon/D \Big)\rightarrow 0\,,\end{align}
while for $\sigma_t$ bounded from below we readily obtain
\begin{align*}\P\Big(\min_{i=k_n,\ldots,n-k_n}\frac{n}{k_n}\Big\vert\sum_{j=i+1}^{i+k_n}\sigma_{i\Delta_n}^2(\Delta_j^nW)^2\Big\vert<D^{-1}\Big)\rightarrow 0\,.\end{align*}
For the proof of \eqref{keyprop4} we proceed with a decomposition analogous to \eqref{decomp} and for the first term along the same lines as above leading to \eqref{mink1}. However, the maximum extends now over the larger set of all indices $i=k_n,\ldots,n-k_n$, and thus instead of \eqref{mink1} the upper bound yields
\((\epsilon/(3D))^{-r}n^{1-\frac{r}{2}}k_n^{r/2}\) \(\log^{r/2}(n)\), which is a factor $k_n$ larger than above. Still, choosing $r$ sufficiently large the term tends to zero. The same reasoning applies to all terms for which we have used Jensen's and generalized Minkowski's inequalities above as \eqref{mink2}.

Upper bounds exploiting the smoothness of the volatility remain as before, for instance
\begin{align*} \sqrt{\log(n)k_n} \max_{i=k_n,\ldots,n-k_n}\Big\vert \frac{n}{k_n}\sum_{j=1}^{k_n}\int_{(j+i-1)\Delta_n}^{(j+i+k_n-1)\Delta_n}(\sigma_s^2-\sigma^2_{i\Delta_n})\,ds\Big\vert &\le K \frac{n\sqrt{\log(n)}}{\sqrt{k_n}}\sum_{j=1}^{k_n}(j\Delta_n)^{\aalpha}\Delta_n\\ &\le K\sqrt{k_n}(k_n\Delta_n)^{\aalpha}\sqrt{\log(n)}\end{align*}
with a constant $K$ on Assumption \ref{assVola}. In this fashion, all terms generalizing the expressions in the proofs of Propositions \ref{prop1}, \ref{prop2} and \ref{prop3} are controlled and we conclude Proposition \ref{prop4}. \qed

\section{Proof of Proposition 3.5}
Recall the definition of the general It\^{o} semi-martingale in \eqref{sm2}. Again, by the usual localization procedure, we can work under the reinforced assumption that the process $X_t$ and its jumps $\Delta X_t$ are bounded as well. We will then work with the decomposition $X_t=X_0+C_t+J_t$, where $J_t$ denotes the pure jump martingale
\[
J_t = \int_0^t\int_{\mathds{R}}\delta(s,x) (\mu-\nu)(ds,dx)
\]
and the continuous part becomes
\[
C_t=\int_0^t \tilde a_s\,ds+\int_0^t\sigma_s\,dW_s
\]
with $\tilde a_s = a_s + \int_{\mathds{R}}\bar\kappa(\delta(s,x))\lambda(dx)$. The latter integral is finite for bounded jumps.

We shall prove only \eqref{thmj2} of Proposition \ref{propjumps} by showing that
\begin{align}\notag\label{resjumps}&\sqrt{k_n\log(n)}\Bigg(\max_{i=k_n,\ldots,n-k_n}\Bigg|\frac{\frac{n}{k_n} \sum_{j=i-k_n+1}^{i}(\Delta_j^n X)^2\1_{\{|\Delta_j^n X|\le u_n\}}}{\frac{n}{k_n} \sum_{j=i+1}^{i+k_n}(\Delta_j^n X)^2\1_{\{|\Delta_j^n X|\le u_n\}}}-1\Bigg|\\
&\hspace*{6cm}-\max_{i=k_n,\ldots,n-k_n}\Bigg|\frac{\frac{n}{k_n} \sum_{j=i-k_n+1}^{i}(\Delta_j^n C)^2}{\frac{n}{k_n} \sum_{j=i+1}^{i+k_n}(\Delta_j^n C)^2}-1\Bigg|\Bigg)\pn 0\,.
\end{align}

Following a decomposition of the error term as in \eqref{VmU}, we have to show that
\begin{align*}&\sqrt{k_n\log(n)}\max_{i=k_n,\ldots,n-k_n}\hspace*{-.05cm}\Bigg|\hspace*{-.05cm}\frac{\frac{n}{k_n}\hspace*{-.05cm} \sum_{j=i-k_n+1}^{i}(\Delta_j^n X)^2\1_{\{|\Delta_j^n X|\le u_n\}}}{\frac{n}{k_n} \sum_{j=i+1}^{i+k_n}(\Delta_j^n X)^2\1_{\{|\Delta_j^n X|\le u_n\}}}-\frac{\frac{n}{k_n} \hspace*{-.05cm}\sum_{j=i-k_n+1}^{i}(\Delta_j^n X)^2\1_{\{|\Delta_j^n X|\le u_n\}}}{\frac{n}{k_n} \sum_{j=i+1}^{i+k_n}(\Delta_j^n C)^2}\hspace*{-.05cm}\Bigg|\\
&+\sqrt{k_n\log(n)}\max_{i=k_n,\ldots,n-k_n}\Bigg|\frac{\frac{n}{k_n} \sum_{j=i-k_n+1}^{i}\Big((\Delta_j^n X)^2\1_{\{|\Delta_j^n X|\le u_n\}}-(\Delta_j^n C)^2\Big)}{\frac{n}{k_n} \sum_{j=i+1}^{i+k_n}(\Delta_j^n C)^2}\Bigg|\pn 0\,.\end{align*}
Both terms are handled similarly and we restrict to the second one. It suffices to prove that
\begin{align}\P\Bigg(\max_{i=k_n,\ldots,n-k_n}\frac{n\sqrt{\log(n)}}{\sqrt{k_n}}\Big|\sum_{j=i-k_n+1}^{i}\Big((\Delta_j^n X)^2\1_{\{|\Delta_j^n X|\le u_n\}}-(\Delta_j^n C)^2\Big)\Big|>\frac{\epsilon}{D}\Bigg)\to 0\,,\end{align}
for all $\epsilon>0$ and constants $D>0$, as \eqref{resjumps} then follows with \eqref{twoprob} and the same bound for the second probability as in the proof of Proposition \ref{prop1}. As $\max_{1\le i\le n}|\Delta_i^n C|=\KLEINO_{a.s.}(u_n)$ by basic extreme value theory we can work on a subset of $\Omega$ where $\max_{1\le i\le n}|\Delta_i^n C|=\KLEINO(u_n)$. Observe that on this subset
\begin{align*}&\max_{i=k_n,\ldots,n-k_n}\Big|\sum_{j=i-k_n+1}^{i}(\Delta_j^n X)^2\1_{\{|\Delta_j^n X|\le u_n\}}-(\Delta_j^n C)^2\Big)\Big|\\ &\le K  \max_{i=k_n,\ldots,n-k_n} \Big(\sum_{j=i-k_n+1}^{i}\hspace*{-.15cm}\1_{\{|\Delta_j^n X|>u_n\}}(\Delta_j^n C)^2 + \hspace*{-.1cm}\sum_{j=i-k_n+1}^{i}\hspace*{-.1cm}\big((|\Delta_j^n J|\wedge u_n)^2\hspace*{-.05cm}+\hspace*{-.05cm}(|\Delta_j^n J| \wedge u_n)|\Delta_j^n C|\big)\hspace*{-.05cm}\Big)\end{align*}
with some constant $K$.

Pertaining the first addend and using $\max_i(\Delta_i^n C)^2=\mathcal{O}_{\P}(\Delta_n\log(n))$, we have to ensure that
\begin{align*}\max_{i=k_n,\ldots,n-k_n}\sum_{j=i-k_n+1}^i \1_{\{|\Delta_j^n X|> u_n\}}=\KLEINO_{\P}\big(\sqrt{k_n}/\log^{3/2}{(n)}\big)\,.\end{align*}
Let $p$ with $1 < p < (2r\tau)^{-1}$ be arbitrary. We use the decomposition $X = X^{'n} + X^{''n}$ with
\[
X^{''n}_t = \int_0^t\int_{\mathds{R}}\delta(s,x) \1_{\{\gamma(x) > u_n^p\}} \mu(ds,dx), \qquad X^{'n}_t = X_t- X^{''n}_t,
\]
and define $A_j^n = \{ |\Delta_j^n X^{'n}| \leq u_n/2 \}$. Finally, $N^n$ is the counting process
\[
N^n_t = \int_0^t\int_{\mathds{R}} \1_{\{\gamma(x) > u_n^p\}} \mu(ds,dx).
\]
We know from (13.1.10) in \cite{JP}  that
\[
\E \Big[\max_{i=k_n,\ldots,n-k_n}\sum_{j=i-k_n+1}^i \1_{\{|\Delta_j^n X|> u_n\}} \1_{\{(A_j^n)^{\complement}\}}\Big] \leq \sum_{j=1}^n \P\big((A_j^n)^{\complement}\big) \to 0
\]
for all such $p$. Then, using $\1_{\{|\Delta_j^n X|> u_n\}} \1_{\{A_j^n\}} \leq \1_{\{|\Delta_j^n X^{''n}|> u_n/2\}}$, all we have to show are conditions under which
\begin{align*}\max_{i=k_n,\ldots,n-k_n}\sum_{j=i-k_n+1}^i \1_{\{|\Delta_j^n N^n| \geq 1\}}=\KLEINO_{\P}\big(\sqrt{k_n}/\log^{3/2}{(n)}\big)\,.\end{align*}
Obviously,
\[
\sum_{j=i-k_n+1}^i \1_{\{|\Delta_j^n N^n| \geq 1\}} \leq N^n_{i\Delta_n} - N^n_{(i-k_n)\Delta_n},
\]
and $N^n$ is a Poisson process with parameter $\int_{\mathds{R}} \1_{\{\gamma(x) > u_n^p\}} \lambda(dx) \leq K u_n^{-rp}$; see (13.1.14) in \cite{JP}. It is clearly enough if the probability of more than $l<\infty$ jumps on one block converges to zero, i.e.\,
\begin{align*}
\P\Big(\bigcup_{j=k_n}^n\big\{N_{(j+1)\Delta_n}^{n}-N^{n}_{(j-k_n+1)\Delta_n}\hspace*{-.05cm}\ge l\big\}\hspace*{-.05cm}\Big)\hspace*{-.05cm}\le \hspace*{-.05cm} n\P\big(N^n_{k_n\Delta_n}\ge l\big)\hspace*{-.05cm}\le \hspace*{-.05cm}Kn\Delta_n^lk_n^lu_n^{-rpl}\hspace*{-.05cm}=\hspace*{-.05cm}K k_n^l\Delta_n^{l(1-rp\tau)-1}.
\end{align*}
Thus, we need the condition that for some $p > 1$ and some $l<\infty$:
\begin{align}\label{con1}k_n^l\Delta_n^{l(1-rp\tau)-1}\rightarrow 0~~\mbox{and}~~2r\tau<1\,.\end{align}
Bounding the second term above comprising small jumps in case of non-truncation poses a more delicate mathematical problem. We restrict to the quadratic jump terms as the cross terms lead in the same way to an obsolete weaker criterion. For finite activity it is enough to ensure that $nk_n^{-1/2}\sqrt{\log(n)}u_n^2\rightarrow 0$. Else, define the sequence of random variables
\begin{align*}\mathcal{Z}_i=\big(|\Delta_i J|\wedge u_n\big)^2-\E\big[\big(|\Delta_i J|\wedge u_n\big)^2\big], \quad i=1,\ldots,n\,.\end{align*}
Note from equation (54) in \cite{aitjac10} that we can bound moments of $\big(|\Delta_i J|\wedge u_n\big)^2$ in the following way:
\begin{align*}\E\Big[\big(|\Delta_i^n J|\wedge u_n\big)^2\big|\mathcal{F}_{(i-1)\Delta_n}\Big]& \le K \Delta_n u_n^{2-r}\,,\\
\var\Big(\big(|\Delta_i^n J|\wedge u_n\big)^2\big|\mathcal{F}_{(i-1)\Delta_n}\Big)&\le \E\Big[\big(|\Delta_i^n J|\wedge u_n\big)^4\big|\mathcal{F}_{(i-1)\Delta_n}\Big]\le u_n^2 K \Delta_n u_n^{2-r} = K \Delta_n u_n^{4-r}\,,\end{align*}
for all $i=1,\ldots,n$. We decompose
\begin{align*}\max_{i=k_n,\ldots,n-k_n}\Big\vert \sum_{j=i-k_n+1}^i \big(|\Delta_j J|\wedge u_n\big)^2\Big\vert &\le \max_{i=k_n,\ldots,n-k_n}\Big\vert \sum_{j=i-k_n+1}^i \mathcal{Z}_j\Big\vert\\ &\quad +\max_{i=k_n,\ldots,n-k_n} \sum_{j=i-k_n+1}^i \E\Big[\big(|\Delta_j J|\wedge u_n\big)^2\Big]\,,\end{align*}
where the condition
\begin{align}\label{con2}\sqrt{k_n}u_n^{2-r}\sqrt{\log(n)}\rightarrow 0 \end{align}
renders the second term with the expectation asymptotically negligible. Yet, the derivation of the maximum in the first term from its expectation can in general become much larger. Observe that
\begin{align*}\max_{i=k_n,\ldots,n-k_n}\Big\vert \sum_{j=i-k_n+1}^i \mathcal{Z}_j\Big\vert= \max_{i=k_n,\ldots,n-k_n}\Big\vert \sum_{j=1}^i \mathcal{Z}_j-\sum_{j=1}^{i-k_n}\mathcal{Z}_j\Big\vert\le 2\max_{i=k_n,\ldots,n}\Big\vert \sum_{j=1}^i \mathcal{Z}_j\Big\vert\,.\end{align*}
Having a sequence of independent and centered random variables, we can apply Kolmogorov's maximal inequality:
\begin{align}\P\Big(\max_{i=k_n,\ldots,n-k_n}\Big\vert \sum_{j=i-k_n+1}^i \mathcal{Z}_j\Big\vert>\lambda \Big)\le \frac{n}{\lambda^2}\var(\mathcal{Z}_1)\le \lambda^{-2}u_n^{4-r}\,.
\end{align}
Thereby we conclude that $\max_{i=k_n,\ldots,n-k_n}\Big\vert \sum_{j=i-k_n+1}^i \mathcal{Z}_j\Big\vert=\mathcal{O}_{\P}\big(u_n^{2-r/2}\big)$. We obtain the condition
\begin{align}\label{con3}\frac{n\sqrt{\log(n)}}{\sqrt{k_n}}u_n^{2-r/2}\rightarrow 0\,.\end{align}
In conclusion, the conditions \eqref{con1}, \eqref{con2} and \eqref{con3} ensure \eqref{resjumps}. A careful computation finally proves that \eqref{con2} is in fact obsolete, which yields our claim.
\hfill\qed

\section{Proof of the lower bound and consistency for the local problem}
{\bf{Proof of Theorem \ref{thm_lowerbound}.}} The proof is based on equivalences of statistical experiments in the strong Le Cam sense. After information-theoretic reductions, we subsequently move to statistical experiments that allow a simpler treatment; see \eqref{eq_iinformation_relation} below. Our final experiment $\mathcal{E}_4$ is a special high-dimensional signal detection problem, from which we will deduce the lower bound by classical arguments.

First consider alternatives with a jump as in \eqref{hypo}. Here, throughout this proof, we set 
\begin{align}\label{kproof}k_n = c_k \bigl(\sqrt{\log (m_n)}n^{\aalpha}/\KK_n\bigr)^{\frac{2}{2\aalpha + 1}}\,,\end{align}
with a constant $c_k > 0$. In the preliminary step, we first grant the experimenter additional knowledge. We restrict to a sub-class of $\mathcal{S}^J_{\theta}\bigl(\aalpha,b_n,\KK_n\bigr)$ from \eqref{jumpalternatives}, where we have one jump at time $\td \in (0,1)$ in the volatility, $|\sigma_{\td}^2-\sigma_{\td-}^2|\ge b_n$. Then, we assume that $\td nk_n^{-1} \in \{1,2,\ldots,\lfloor n/k_n\rfloor-1\}$, such that the jump time is in the set of observation grid points which are multiples of $k_n$. Furthermore, we can stick to $X_0 = 0$ and $a_s = 0, s \in [0,1]$. From an information-theoretic view, obtaining this additional knowledge can only decrease the lower boundary on minimax distinguishability. Consequently, a lower bound derived for the sub-model carries over to the less informative general situation.

To ease the exposition, we first set $\sigma_-^2=1$ and $\KK_n=1$ and generalize the result at the end of this proof. Next, denote with $[a]_{b} = a \mod b$ and let
\begin{align}
\sigma_{j\Delta_n}^2= \begin{cases} 1+(k_n - [j]_{k_n})^{\aalpha}n^{-\aalpha}, & \td n \leq j < \td n + k_n,\\
1, & \text{else}. \end{cases}
\end{align}
The discretized squared volatility exhibits a jump (resp.\,change-point) of order $b_n$ at $\td$ and then decays on the window $[\td,\td+k_n\Delta_n]$ smoothly with regularity $\aalpha$ and is constant elsewhere. It suffices to consider the sub-class $\varSigma_{\theta}\subset \mathcal{S}_{\theta}^J\bigl(\aalpha,b_n,\KK_n\bigr)$ of squared discretized volatility processes of the above form for which it remains unknown on which window the jump occurs.

Introduce a sequence $r_n$ with $r_n\rightarrow\infty$ such that $r_nk_n^{-1}\rightarrow 0$ as $n\to \infty$. We specify the following stepwise approximation of $(\sigma^2_{j\Delta_n})_{0\le j\le n}\in\varSigma_{\theta}$:
\begin{align*}
\widetilde{\sigma}_{j\Delta_n}^2 = \begin{cases}1+ (k_n - i r_n)^{\aalpha}n^{-\aalpha}, & \td n + (i-1) r_n \leq j \leq \td n + i r_n,\, 1 \leq i \leq \tfrac{k_n}{r_n},\\
1, & \text{else}. \end{cases}
\end{align*}
Denote the observations by $\eta_j = \sigma_{(j-1)\Delta_n}\big(W_{j\Delta_n}-W_{(j-1)\Delta_n}\big)$ and \(\widetilde{\eta}_j = \widetilde{\sigma}_{(j-1)\Delta_n}\)\\ \(\big(W_{j\Delta_n}-W_{(j-1)\Delta_n}\big)$, $j=1,\ldots,n\), respectively, with $W$ the Wiener process in \eqref{sm}.
In the sequel, it is convenient to distinguish the two cases where $\aalpha > 1/2$ and $\aalpha \leq 1/2$.\\
{\bf Case} $\aalpha > 1/2$: As alluded to above, we relate different experiments:
\begin{description}
\item[$\bf{\mathcal{E}_1}$]: Observe $\big(\eta_j\big)_{1\le j\le n}$ and information $\td nk_n^{-1} \in \{1,2,\ldots,\lfloor n/k_n\rfloor-1\}$ is provided.
\item[$\bf{\mathcal{E}_2}$]: Observe $\big(\widetilde{\eta}_j\big)_{1\le j\le n}$ and information $\td nk_n^{-1} \in \{1,2,\ldots,\lfloor n/k_n\rfloor-1\}$ is provided.
\item[$\bf{\mathcal{E}_3}$]: Observe $\chiv = \big(\bigl(\widetilde{\sigma}^2_{ik_n\Delta_n}\chi_i\bigr)_{i\in\mathcal{I}_1},\bigl(\widetilde{\sigma}^2_{\td+(i-1)r_n\Delta_n}\tilde\chi_i\bigr)_{i\in\mathcal{I}_2}\big)$, where indices $(ik_n,i\in\mathcal{I}_1)$ expand over all multiples of $k_n$, except the one where the jump is located, i.e.\ $\mathcal{I}_1=\{1,\ldots,\td nk_n^{-1}-1,\td n k_n^{-1}+1,\ldots,\lfloor n/k_n\rfloor-1\}$, and $(\td n+(i-1)r_n,i\in\mathcal{I}_2)$ over all multiples of $r_n$ in the window of length $k_n\Delta_n$ where $(\sigma_j^2)$ is non-constant, i.e.\ $\mathcal{I}_2=\{1,2,\ldots,k_nr_n^{-1}\}$. $(\chi_{i})_{i\in\mathcal{I}_1}$ and $(\tilde{\chi}_{i})_{i\in\mathcal{I}_2}$ are i.i.d.\,random variables having chi-square distribution with degrees of freedom $k_n$ for $i \in \mathcal{I}_1$ and $r_n$ for $i \in \mathcal{I}_2$. Moreover, information $\td nk_n^{-1} \in \{1,2,\ldots,\lfloor n/k_n\rfloor-1\}$ is provided.
\item[$\bf{\mathcal{E}_4}$]: We observe $\xiv= \big(\bigl(k_n^{-1/2}\xi_i\widetilde{\sigma}_{ik_n\Delta_n}^2 + \widetilde{\sigma}_{ik_n\Delta_n}^2\bigr)_{i \in \mathcal{I}_1}, \bigl(r_n^{-1/2}\tilde\xi_{i}\widetilde{\sigma}_{\td+(i-1)r_n\Delta_n}^2 + \widetilde{\sigma}_{\td+(i-1)r_n\Delta_n}^2\bigr)_{i \in \mathcal{I}_2}\big)$  where $(\xi_i,\tilde \xi_i)$ are i.i.d.\,standard normal random variables. Moreover, information $\td nk_n^{-1} \in \{1,2,\ldots,\lfloor n/k_n\rfloor-1\}$ is provided. 
\end{description}
When considering the above experiments, we always have $(\sigma_{j\Delta_n}^2)\in\varSigma_{\theta}$ (or $(\widetilde{\sigma}^2_{j\Delta_n})\in\varSigma_{\theta}$) as unknown parameter that index a family of probability measures $\{\P_{(\sigma_{j\Delta_n}^2)}\}$. For the sake of readability, we move this formalism to the background and omit subscripts indicating the parameter space.
We show the following relations for the experiments, where $\thicksim$ marks strong Le Cam equivalence and $\approx$ asymptotic equivalence:
\begin{align}\label{eq_iinformation_relation}
\mathcal{E}_1 \approx \mathcal{E}_2 \thicksim \mathcal{E}_3 \approx \mathcal{E}_4.
\end{align}
Finally, we shall derive the lower bound in $\bf{\mathcal{E}_4}$ which carries over to $\bf{\mathcal{E}_1}$ by the above relations and thus also to our general model. The proof is now divided into four main steps.\\

\vspace{-0.3cm}
{\bf Step 1} $\mathcal{E}_1 \approx \mathcal{E}_2$: For random variables $U, V$ and their laws $\P_U,\P_V$, we denote the Kullback-Leibler divergence ${\bf D}\bigl(U\| V\bigr)={\bf D}\bigl(\P_U\| \P_V\bigr) = \int \log\big(d \P_U/d \P_V\big) d \P_U$. For normal families with unknown variance $\P_{\theta}=N(0,\theta)$, it is known that
\begin{align*}{\bf D}\big(\P_{\theta}\| \P_{\theta'}\big)=\E_{\theta}\Big[\log\Big(\frac{d\P_{\theta}}{d\P_{\theta'}}\Big)\Big]=-\frac12 \Big(\log\Big(\frac{\theta}{\theta'}\Big)+1-\frac{\theta}{\theta'}\Big)\,,\end{align*}
such that for $\theta=\theta'+\delta$ and considering asymptotics where $\delta\rightarrow 0$, we obtain
\begin{align}\label{kl}{\bf D}\big(\P_{\theta'+\delta}\| \P_{\theta'}\big)&=-\frac12 \Big(\log\Big(1+\frac{\delta}{\theta'}\Big)-\frac{\delta}{\theta'}\Big)= \frac{\delta^2}{4(\theta')^2}+\mathcal{O}\big(\delta^3\big)\,.\end{align}
As $\mathcal{E}_1$ and $\mathcal{E}_2$ share a common space on which the considered random variables are accommodated, asymptotic equivalence holds if $\|\P_{(\eta_j)}-\P_{(\tilde\eta_j)} \|_{TV}\rightarrow 0$ as $n\rightarrow\infty$ where $\|\cdot\|_{TV}$ denotes the total variation distance and $\P_{(\eta_j)}$ the law of observations $(\eta_j)$. We exploit Pinsker's inequality
\begin{align}
\bigl\|\P_{(\eta_j)} -  \P_{(\tilde{\eta}_j)}\bigr\|_{TV}^2 \leq \frac{1}{2}{\bf D}\bigl((\eta_j)\| (\tilde{\eta}_j)\bigr)\,.
\end{align}
By Gaussianity and independence of Brownian increments, implying additivity of the Kullback-Leibler divergences, it follows with \eqref{kl} for a piecewise constant approximation of a function with regularity $\aalpha$ on $k_nr_n^{-1}$ intervals of length $r_n\Delta_n$: 
\begin{align*}
 {\bf D}\bigl((\eta_j)\| (\tilde{\eta}_j)\bigr)= \mathcal{O}\bigl(1\bigr) \sum_{i = 1}^{k_nr_n^{-1}} \sum_{j = 1}^{r_n}  \big(j\Delta_n\big)^{2\aalpha} = \mathcal{O}\bigl(n^{-2\aalpha} k_n r_n^{2\aalpha}\bigr)\,,
\end{align*}
which tends to zero for $r_nk_n^{-1}=\mathcal{O}(n^{-\epsilon})$ for some $\epsilon>0$. \\

\vspace{-0.3cm}
{\bf Step 2} $\mathcal{E}_2 \thicksim \mathcal{E}_3$: The vector of averages
\begin{align*}\bigg(\Big(k_n^{-1} \sum_{j=1}^{k_n}\widetilde \eta_{ik_n+j-1}^2\Big)_{i\in \mathcal{I}_1}\, ,\,\Big(r_n^{-1} \sum_{j =1}^{r_n}\widetilde \eta_{\td n+(i-1)r_n+j-1}^2\Big)_{i \in \mathcal{I}_2}\bigg)\,
\end{align*}
forms a sufficient statistic for $(\widetilde{\sigma}_{j-1}^2)_{1\le j\le n}$. Thereby we conclude, see e.g.\, Lemma 3.2 of \cite{brownlow}, the strong Le Cam equivalence.\\

\vspace{-0.3cm}
{\bf Step 3} $\mathcal{E}_3 \approx \mathcal{E}_4$:~
Let $\chiv^{\diamond} = \big(k_n^{-1/2}\bigl(\widetilde{\sigma}^2_{ik_n\Delta_n}(\chi_i-k_n)\bigr)_{i\in\mathcal{I}_1},r_n^{-1/2}\bigl(\widetilde{\sigma}^2_{\td+(i-1)r_n\Delta_n}(\tilde\chi_i-r_n)\bigr)_{i\in\mathcal{I}_2}\big)$ and $\xiv^{\diamond} = \big(\bigl(\xi_i\widetilde{\sigma}_{ik_n\Delta_n}^2 \bigr)_{i \in \mathcal{I}_1}, \bigl(\tilde\xi_{i}\widetilde{\sigma}_{\td+(i-1)r_n\Delta_n}^2 \bigr)_{i \in \mathcal{I}_2}\big)$. In both experiments random variables are accommodated on the same space. Rescaling and a location shift yield with Pinsker's inequality
\begin{align*}
\bigl\|\P_{\chiv} -  \P_{\xiv}\bigr\|^2_{TV} = \bigl\|\P_{\chiv^{\diamond}} -  \P_{\xiv^{\diamond}}\bigr\|^2_{TV}\leq \frac{1}{2} {\bf D}\bigl(\chiv^{\diamond}\| \xiv^{\diamond}\bigr)\,.
\end{align*}
By independence, it follows that
\begin{align*}
{\bf D}\bigl(\chiv^{\diamond}\|\xiv^{\diamond}\bigr)\hspace*{-.00cm} &\leq \hspace*{-.00cm}\sum_{i \in \mathcal{I}_1} \hspace*{-.00cm}{\bf D}\Big(k_n^{-1/2}\widetilde{\sigma}^2_{\frac{ik_n}{n}}(\chi_i-k_n)\big\| \xi_i\widetilde{\sigma}_{\frac{ik_n}{n}}^2\Big) \hspace*{-.00cm}\\
&\quad +\hspace*{-.00cm} \sum_{i \in \mathcal{I}_2} \hspace*{-.00cm}{\bf D}\Big(r_n^{-1/2}\widetilde{\sigma}^2_{\td+\frac{(i-1)r_n}{n}}(\tilde\chi_i-r_n)\big\| \tilde\xi_{i}\widetilde{\sigma}_{\td+\frac{(i-1)r_n}{n}}^2\Big).
\end{align*}
An application of Theorem 1.1 in \cite{goetze} yields
\begin{align*}
&\sum_{i \in \mathcal{I}_1} {\bf D}\big(k_n^{-1/2}\widetilde{\sigma}^2_{ik_n\Delta_n}(\chi_i-k_n)\big\| \xi_i\widetilde{\sigma}_{ik_n\Delta_n}^2\big)=\mathcal{O}\big(nk_n^{-2}\big)\,, \\
&\sum_{i \in \mathcal{I}_2} {\bf D}\big(r_n^{-1/2}\widetilde{\sigma}^2_{\td+(i-1)r_n\Delta_n}(\tilde\chi_i-r_n)\big\| \tilde\xi_{i}\widetilde{\sigma}_{\td+(i-1)r_n\Delta_n}^2\big)=\mathcal{O}\big(k_nr_n^{-2}\big)\,.
\end{align*}
For $\aalpha > 1/2$, we have $nk_n^{-2} = \KLEINO(1)$. Choosing $r_n$ sufficiently large such that $k_nr_n^{-2} = \KLEINO(1)$, it follows that
\begin{align}
\bigl\|\P_{\chiv} -  \P_{\xiv}\bigr\|_{TV}  = \KLEINO\bigl(1\bigr)\,,
\end{align}
what ensures the claimed asymptotic equivalence.\\

\vspace{-0.3cm}
{\bf Step 4}: By the previous steps, it suffices to establish a lower bound for the distinguishability in experiment $\mathcal{E}_4$. Adding an additional drift, which gives clearly an equivalent experiment, we consider observations $\xiv= \big(\bigl(k_n^{-1/2}\xi_i\widetilde{\sigma}_{ik_n\Delta_n}^2 + \widetilde{\sigma}_{ik_n\Delta_n}^2-1\bigr)_{i \in \mathcal{I}_1}, \bigl(r_n^{-1/2}\tilde\xi_{i}\widetilde{\sigma}_{\td+(i-1)r_n\Delta_n}^2 + \widetilde{\sigma}_{\td+(i-1)r_n\Delta_n}^2-1\bigr)_{i \in \mathcal{I}_2}\big)$. Then, the testing problem can be interpreted as a high dimensional location signal detection problem in the sup-norm. More precisely, we test the hypothesis
\begin{align}\label{signaltest}
H_0:\sup_j(\tilde\sigma_j^2-1) = 0
~~\text{
against the alternative}~~
H_1:  \sup_j(\tilde\sigma_j^2-1) \ge  b_n\,,
\end{align}
and we are interested in the maximal value $b_n \to 0$ such that the hypothesis $H_0$ and $H_1$ are non-distinguishable in the minimax sense.
Non-distinguishability in the minimax sense is formulated as
\begin{align}\label{eq_lower_bound_highfim_setup}
\lim_{n \to \infty}\inf_{\psi}\gamma_{\psi}\bigl(\aalpha,b_n\bigr) = 1,
\end{align}
and the detection boundary here is $b_n\propto(k_n\Delta_n)^{\aalpha}\propto n^{-\frac{\aalpha}{2\aalpha+1}}$. In order to show \eqref{eq_lower_bound_highfim_setup}, we proceed in the fashion of Section 3.3.7 of \cite{ingster}. Let $\P_{\xiv} $ be the law of the observations. We consider the probability measures
\begin{align*}
\P_0 = \P_{\xiv} \times \P_{\td_{0}} \quad \text{and} \quad \P_1 = \P_{\xiv} \times \P_{\td_{1}},
\end{align*}
where $\P_{\td_0}$ means the hypothesis of the test applies (no jump) and $\P_{\td_{1}}$ draws a jump-time $\td$ with $\td nk_n^{-1}\in\{1,\ldots,\lfloor n/k_n\rfloor-1\}$ uniformly from this set. Therefore, $\P_0$ represents the probability measure without signal, and $\P_1$ the measure where a signal is present. It then follows that
\begin{align*}
\inf_{\psi }\gamma_{\psi}\bigl( \aalpha,b_n\bigr) &\geq 1 - \frac{1}{2}\bigl\|\P_1 - \P_0 \bigr\|_{TV} \geq 1 - \frac{1}{2}\bigl|\E_{P_0}\bigl[L_{0,1}^2 - 1\bigr]\bigr|^{1/2},
\end{align*}
with $L_{0,1} = d\P_1/d\P_0$ the likelihood ratio of the measures $\P_1$ and $\P_0$. For the validity of \eqref{eq_lower_bound_highfim_setup}, it thus suffices to establish
\begin{align}\label{eq_likelohood_ratio_to_one}
\E_{\P_0}\bigl[L_{0,1}^2\bigr] \to 1 \quad \text{as $n \to \infty$.}
\end{align}
To this end, for given $\td$ we denote with $u_{i}^{\td nk_n^{-1}}  = \widetilde{\sigma}_{\td+(i-1)r_n\Delta_n}^4$ for $i \in \mathcal{I}_2$, $v_i^{\td nk_n^{-1}} =(u_i^{1/2}-1)r_n^{1/2}$. We first perform some preliminary computations. Denote with $\varphi_Y(x)$ the density function of a Gaussian random variable $Y$, not necessarily standard normal, and
for $a,b \in \{1,\ldots, \lfloor n/k_n \rfloor -1\}$
\begin{align*}
I_{a,b}(x,y) := \prod_{i \in \mathcal{I}_2} \frac{\varphi_{\tilde\xi_i (u_i^{a})^{1/2} + v_i^{a}}(x_i)}{\varphi_{\tilde\xi_{i}}(x_i)}\,\prod_{i \in \mathcal{I}_2} \frac{\varphi_{\tilde\xi_i (u_i^{b})^{1/2} + v_i^{b}}(y_i)}{\varphi_{\tilde\xi_{i}}(y_i)}\,.
\end{align*}
Then, we have that \(I_{a,b} := \int I_{a,b}(x,y) \prod_{i  \in \mathcal{I}_2} \varphi_{\tilde\xi_{i}}(x_i) d x_i \prod_{i \in \mathcal{I}_2} \varphi_{\tilde\xi_{i}}(y_i) d y_i = 1.\)
Next, for $a \in \{1,\ldots, \lfloor n/k_n \rfloor -1\}$, consider
\begin{align*}
II_a(x) := \prod_{i\in \mathcal{I}_2}\left(\frac{\varphi_{\tilde\xi_i (u_i^{a})^{1/2} + v_i^{a}}(x_i)}{\varphi_{\tilde\xi_{i}}(x_i)}\right)^2.
\end{align*}
Observe that for a standard Gaussian random variable $Z$ and $s,t \in \R$, $|s| < 1/2$:
\begin{align}
\E\bigl[\exp(sZ^2 + t Z)\bigr] = (1-2s)^{-1/2}\,\exp\left(\frac{t^2}{2 - 4 s}\right)\,. 
\end{align}
This, together with the inequality 
\begin{align*}
C_0 k_n \left(k_n\Delta_n\right)^{2\aalpha} \leq r_n \sum_{i = 0}^{k_n/r_n-1} \left((k_n - ir_n)\Delta_n\right)^{2\aalpha} \leq k_n  \left(k_n\Delta_n\right)^{2\aalpha}
\end{align*}
for some constant $C_0 > 0$ and routine calculations yield for some $C_0 \leq C_1 \leq 1$:
\begin{align*}
II_a := \int II_a(x) \prod_{i\in \mathcal{I}_2} \varphi_{\tilde \xi_i}(x_i) d x_i \leq e^{C_1 k_n  \left(k_n\Delta_n\right)^{2\aalpha}}\bigl(1 + \KLEINO(1)\bigr).
\end{align*}
With all the preliminary calculations completed, we are now ready to derive a bound for
\begin{align*}
\E_{\P_0}\bigl[L_{0,1}^2\bigr] - 1&= \sum_{\substack{a,b = 1\\ a \neq b}}^{\lfloor n/k_n\rfloor - 1} \P\bigl(\td n k_n^{-1} = a\bigr)\P\bigl(\td n k_n^{-1} = b\bigr)\bigl(I_{a,b} - 1\bigr)\\&\quad + \sum_{a = 1}^{\lfloor n/k_n\rfloor - 1} \P\bigl(\td n k_n^{-1} = a\bigr)^2\bigl(II_a -1 \bigr)
\end{align*}
where the first sum vanishes. For an appropriate choice of $c_k > 0$ in \eqref{kproof}, we have that $k_n (k_n\Delta_n)^{2\aalpha} = C_2 \log(n/k_n)$ for some $C_2 < C_1^{-1}$.
Since $\P\bigl(\td n k_n^{-1} = a\bigr) = k_n\Delta_n$, we thus obtain
\begin{align}\nonumber
\bigl|\E_{P_0}\bigl[L_{0,1}^2\bigr]-1\bigr| &\le \sum_{a=1}^{\lfloor n/k_n\rfloor } \P\bigl(\td n k_n^{-1} = a\bigr)^2 \big(e^{C_1 k_n (k_n\Delta_n)^{2\aalpha}}-1\big) \\
\label{eq_likelihood_ratio_equation_1} &= (1+\KLEINO(1))k_n\Delta_n\, e^{C_1 k_n \left(k_n\Delta_n\right)^{2\aalpha}}\,.
\end{align}
We conclude \eqref{eq_likelohood_ratio_to_one} using
\begin{align*}
k_n\Delta_n e^{C_1 k_n (k_n\Delta_n)^{2\aalpha}} = k_n\Delta_n\exp\big({C_1\,C_2 \log\big(n/k_n\big)}\big) = (k_n\Delta_n)^{1 - C_1\,C_2} = \KLEINO(1)\,.
\end{align*}
{\bf Case} $\aalpha \leq 1/2$: The only time we make use of the condition $\aalpha > 1/2$ above is in Step 3 to obtain $n/k_n^2 = \KLEINO(1)$. The necessity of this relation is due to the large number of blocks $n/k_n$, when operating with the entropy bounds. To establish the lower bound, this constraint can be removed by granting the experimenter even more additional information what is briefly sketched in the following. Indeed, suppose we know in addition that $\td n \in \bigl\{k_n,2 k_n ,\ldots, l_n k_n \bigr\}$ where $l_n = n^{\ld}\ll n/k_n$, $\ld > 0$ arbitrarily small but strictly positive and such that $l_n \in \N$. Using the sufficiency argument of Step 2, we can gather all the information contained in $(\eta_i)_{l_n k_n < i \leq n}$ in one single average $(n - (l_n+1)k_n)^{-1} \sum_{i = l_n k_n+1}^{n} \eta_i^2$. Then, one can repeat Steps 3 and 4, subject to the weaker condition $l_n/k_n = \KLEINO(1)$. Selecting $\ld > 0$ sufficiently small for each $0 < \aalpha \leq 1$, this is always possible. Substituting $nk_n^{-1}$ by $l_n$ in the sum and (squared) probability in Step 4, we obtain instead of \eqref{eq_likelihood_ratio_equation_1}
\begin{align*}\bigl|\E_{P_0}\bigl[L_{0,1}^2\bigr]-1\bigr|= (1+\KLEINO(1))l_n^{-1}\, e^{C_1 k_n \left(k_n\Delta_n\right)^{2\aalpha}}\,.\end{align*}
For an appropriate choice of $c_k > 0$ in \eqref{kproof}, $k_n (k_n\Delta_n)^{2\aalpha} = C_2 \log(n/k_n)$ with $C_2 < \mathfrak{l} \,C_1^{-1}$. Hence, we conclude that the term tends to zero and the lower bound in Step 4 gives the same minimax detection boundary.\\

\vspace{-0.3cm}
Let us now touch on the general case with some $\sigma_{-}^2 > 0$ and sequences $\KK_n$. We can divide formulas in \eqref{signaltest} by $\sigma_{-}^2$ to rescale. Exactly the same arguments lead to $\lim_{n\rightarrow\infty}\inf_{\psi}\gamma_{\psi}(\aalpha,b_n)=1$ for $k_n$ given in \eqref{kproof} with $b_n\le L_n (k_n\Delta_n)^{\aalpha}\sigma^2_-$, which gives the general result.
Finally, we remark on the regularity alternative $H_1^{\text{R}}$. The proof is along the same lines as for jumps where instead of a jump of size $\KK_n(k_n\Delta_n)^{\aalpha}$ at unknown location, we observe a sudden, more regular increase in $\sigma_t^2$ of size $\KK_n(k_n\Delta_n)^{\aalpha+\aalpha'}$, where we exploit the regularity $\aalpha'$ in sets $\mathcal{S}_{\td}^{\text{R}}$. Hence, the jump gets replaced with a gradual regular increase. 
Then, the arguments are almost identical. This also highlights the fact that at (or below) the boundary $b_n$, the different alternatives $H_1$ and $H_1^{\text{R}}$ are not distinguishable.\hfill\qed\\[.2cm]

{\bf{Proof of Theorem \ref{thm_upper1}.}}
Using similar arguments as in the proof of Theorem \ref{thm1} and in particular Proposition \ref{prop4}
, one derives for
\begin{align*}
\overline{V}_{n,i} =   \frac{ \big\vert\frac{n}{k_n^{\diamond}}\sum_{j=i-k_n^{\diamond}+1}^{i}\big((\Delta_{j}^n X)^2- \E\bigl[(\Delta_{j}^n X)^2\bigr]\big)-\frac{n}{k_n^{\diamond}}\sum_{j=i+1}^{i+k_n^{\diamond}} \big((\Delta_{j}^n X)^2- \E\bigl[(\Delta_{j}^n X)^2\bigr]\big)\big\vert}{\frac{n}{k_n^{\diamond}}\sum_{j=i+1}^{i+k_n^{\diamond}} (\Delta_{j}^n X)^2}\,, \end{align*}
$\quad k_n^{\diamond} \leq i \leq n - k_n^{\diamond}$, that under the alternatives $H_1$ and $H_1^{\text{R}}$, it holds that
\begin{align}\label{eq_thm_upper1_1}
\sqrt{k_n^{\diamond}}\overline{V}_{n,i} = \mathcal{O}_P\bigl(1\bigr), \quad  k_n^{\diamond} \leq i \leq n - k_n^{\diamond}.
\end{align}
Based on \eqref{eq_thm_upper1_1}, a simple estimate yields
\begin{align}\label{eq_thm_upper 1_2}\nonumber
V_{n}^* &\geq -\overline{V}_{n,\lfloor n\td\rfloor } + \frac{n}{k_n^{\diamond}}\biggl|\int_{\td - k_n^{\diamond}\Delta_n}^{\td} \sigma_s^2\, ds - \int_{\td}^{\td + k_n^{\diamond}\Delta_n} \sigma_s^2\, ds\biggr| \frac{\bigl(1 - \KLEINO_{\P}(1)\bigr)}{\sigma_{\td}^2}\\& \geq - \mathcal{O}_{\P}\bigl(\big(k_n^{\diamond}\big)^{-1/2}\bigr) + \frac{n}{k_n^{\diamond}}\biggl|\int_{\td - k_n^{\diamond}\Delta_n}^{\td} \sigma_s^2\, ds - \int_{\td}^{\td + k_n^{\diamond}\Delta_n} \sigma_s^2\, ds\biggr| \frac{\bigl(1 - \KLEINO_{\P}(1)\bigr)}{\sup_{0 \leq t \leq 1}\sigma_t^2}\,.
\end{align}
Observe that in order to prove $\gamma_{\psi^{\diamond}}(\aalpha,b_n^{\diamond}) \to 0$, it suffices to show that
\begin{align}\label{eq_thm_upper 1_3}
\P\bigl(V_n^{*} \geq 2 C^{\diamond} \sqrt{2 \log(m_n^{\diamond})/k_n^{\diamond}}\bigr) &\to 1 \quad \text{under $H_1$ or $H_1^{\text{R}}$,}\\
\label{eq_thm_upper 1_3.5}\mbox{and}~~
\P\bigl(V_n^{*} < 2 C^{\diamond} \sqrt{2 \log(m_n^{\diamond})/k_n^{\diamond}}\bigr) &\to 1 \quad \text{under $H_0$}.
\end{align}

{\bf Case} $H_1$: Using that $(\sigma_t^2 - \Delta \sigma_t^2)_{t \in [0,1]} \in \Sigma(\aalpha,\KK_n)$ and $\Delta \sigma_t \geq 0$, we get
\begin{align*}
\frac{n}{k_n^{\diamond}}\sup_{t \geq \td}\biggl|\int_{\td- k_n^{\diamond}\Delta_n}^{\td} \sigma_s^2 \,ds - \int_{\td}^{\td + k_n^{\diamond}\Delta_n} \sigma_s^2\, ds\biggr| \geq b_n - 2 \KK_n \left(k_n^{\diamond}\Delta_n\right)^{\aalpha}.
\end{align*}
Hence \eqref{eq_thm_upper 1_3} follows for $\aalpha'=0$ in \eqref{ass_b_n_diamond} with \eqref{eq_thm_upper 1_2}.\\

\vspace{-0.3cm}
{\bf Case } $H_1^{\text{R}}$: For $\sigma_t^2 \in \mathcal{S}_{\td}^{\text{R}}\bigl(\aalpha,\aalpha',b_n^{\diamond},\KK_n, k_n\bigr)$, we have that
\begin{align*}
\sigma^2_{\td + h} \geq \sigma^2_{\td} +  b_n h^{\aalpha'} \quad \text{or} \quad \sigma^2_{\td + h} \leq \sigma^2_{\td} -  b_n h^{\aalpha'}, \quad \text{for $0 \le h \leq 2k_n\Delta_n$}.
\end{align*}
It follows that 
\begin{align*}
\frac{n}{k_n^{\diamond}}\biggl|\int_{\td + k_n^{\diamond}\Delta_n}^{\td+2 k_n^{\diamond}\Delta_n} \big(\sigma_s^2-\sigma^2_{s-k_n^{\diamond}\Delta_n}\big) \,ds \biggr| &\geq b_n \big(k_n^{\diamond}\Delta_n\big)^{\aalpha'}.
\end{align*}
Therefore \eqref{eq_thm_upper 1_3} follows for \eqref{ass_b_n_diamond} with \eqref{eq_thm_upper 1_2}.\\

\vspace{-0.3cm}
{\bf Case } $H_0$: Under the hypothesis we employ the upper bound
\begin{align*}
V_{n}^* \leq \max_{k_n^{\diamond} \leq i \leq n -k_n^{\diamond}} \overline{V}_{n,i} + \KK_n (k_n^{\diamond}\Delta_n)^{\aalpha}
\end{align*}
to prove \eqref{eq_thm_upper 1_3.5}. \eqref{eq_kn_and_mn} implies
\begin{align*}
2 C^{\diamond} \sqrt{2 \log(m_n^{\diamond})/k_n^{\diamond}} \geq 2 \sqrt{2 \log(m_n^{\diamond})/k_n^{\diamond}} + \KK_n (k_n^{\diamond}\Delta_n)^{\aalpha},
\end{align*}
and hence it suffices to show that $\P\bigl(\max_{k_n^{\diamond} \leq i \leq n -k_n^{\diamond}} \overline{V}_{n,i} \leq  \sqrt{2 \log(m_n^{\diamond})/k_n^{\diamond}}\bigr) \to 1$. This, however, follows from a direct adaption of Theorem \ref{thm1}, which completes the proof.\hfill\qed\\[.2cm]

\section{Proofs of Section 4.2}
{\bf{Proof of Proposition \ref{propest}.}} We use the following elementary lemma to prove Proposition \ref{propest}.
\begin{lem}\label{lem_argmax_bound}
Let $f(t)$ and $g(t)$ be functions on $[0,\td]$ such that $f(t)$ is increasing. As long as $f(\td) - f(\td - \gamma) \geq \sup_{0 \leq t \leq \td}|g(t)|$ for some $\gamma \in [0,\td]$, we have that
\begin{align*}
\operatorname{argmax}_{0 \leq t \leq \td}\bigl(f(t)+g(t)\bigr) \geq \td - \gamma.
\end{align*}
An analogous result holds if $f(t)$ and $g(t)$ are functions on $[\td,1]$ and $f(t)$ is decreasing.
\end{lem}

For $\td\in(0,1)$ define $i^*=\lceil \td n\rceil$, the smallest integer such that $i^* \Delta_n$ is larger or equal than $\td$. While $(\sigma_t^2)_{t\in[0,1]}$ is the squared volatility process containing one jump at time $\td$, denote by $(\tilde{\sigma}_t^2)_{t\in[0,1]}$ the same path without jump, such that
\begin{align*}
\sigma_{i\Delta_n}^2 = \tilde\sigma^2_{i\Delta_n} + \delta \1(i \geq i^*)
\end{align*}
with jump size $\delta$. Without loss of generality, we assume $\delta > 0$. Define
\begin{align} \label{42}
f\bigl(i\Delta_n\bigr)=\left\{\begin{array}{cl} 0, &\mbox{if $i +k_n < i^*$,}\\ (i+k_n-i^*)k_n^{-1/2}\delta &\mbox{for}~i=i^*-k_n,\ldots,i^*\,,\\ \sqrt{k_n}\delta &\mbox{if $i>i^*$,}\end{array}\right.
\end{align}
and $(f(t))_{t\in[0,1]}$ the associated piecewise constant increasing step function. For $i=k_n,\ldots,n-k_n$:
\begin{align*}
&\sum_{j = i- k_n + 1}^{i}n(\Delta_j^n X)^2-\sum_{j = i + 1}^{i + k_n}n(\Delta_j^n X)^2\\
&=\big\{\sum_{j = i- k_n + 1}^{i}\bigl(n(\Delta_j^n X)^2 - \E[n(\Delta_j^n X)^2] \bigr)- \sum_{j = i + 1}^{i + k_n} \bigl(n(\Delta_j^n X)^2 - \E[n(\Delta_j^n X)^2] \bigr)\big\} \\
&+ \big\{ \sum_{j = i- k_n + 1}^{i}\bigl(\E[n(\Delta_j^n X)^2] - \tilde{\sigma}_{j\Delta_n}^2\bigr)- \sum_{j = i + 1}^{i + k_n} \bigl(\E[n(\Delta_j^n X)^2-\sigma_{j\Delta_n}^2] \bigr) \big\}\\
&+ \big\{\hspace*{-.05cm}\sum_{j = i- k_n + 1}^{i}\hspace*{-.05cm}\tilde{\sigma}_{j\Delta_n}^2 \hspace*{-.05cm}- \hspace*{-.1cm}\sum_{j = i + 1}^{i + k_n} \tilde{\sigma}_{j\Delta_n}^2 \big \}-\hspace*{-.05cm}\sum_{j = i + 1}^{i + k_n} \bigl(\sigma_{j\Delta_n}^2 \hspace*{-.1cm} - \tilde{\sigma}_{j\Delta_n}^2\bigr)\hspace*{-.05cm} =: A_i^n + B_i^n + C_i^n -\hspace*{-.05cm} \sum_{j = i + 1}^{i + k_n} \bigl(\sigma_{j\Delta_n}^2 \hspace*{-.05cm}- \tilde{\sigma}_{j\Delta_n}^2\bigr),
\end{align*}
with the obvious definition using the curly brackets. Thus, for the step function $(g(t))_{t\in[0,1]}$, with
\begin{align*}g\bigl(i\Delta_n\bigr)=k_n^{-1/2}\Bigg(\sum_{j = i- k_n + 1}^{i}n(\Delta_j^n X)^2-\sum_{j = i + 1}^{i + k_n}n(\Delta_j^n X)^2+\sum_{j = i + 1}^{i + k_n}\big(\sigma^2_{i^*\Delta_n}-\tilde{\sigma}^2_{i^*\Delta_n}\big)\Bigg)\end{align*}
for $i=k_n,\ldots,n-k_n$ and $g(i\Delta_n)=0$ else, we have that
\begin{align*}
&\sqrt{k_n}\,g\bigl(i\Delta_n\bigr)=A_{n,i}+B_{n,i}+C_{n,i}+D_{n,i}\,,\\
&~~~~~~~~-D_{n,i}=\sum_{j = i + 1}^{i + k_n}\big(\sigma^2_{j\Delta_n}-\tilde{\sigma}^2_{j\Delta_n}\big)-\big(\sigma^2_{i^*\Delta_n}-\tilde{\sigma}^2_{i^*\Delta_n}\big)\\
&\phantom{-D_{n,i}}~~~~~~~~~~=\sum_{j = i + 1}^{i + k_n}\big(\sigma^2_{j\Delta_n}-\sigma^2_{i^*\Delta_n}\big)+\big(\tilde{\sigma}^2_{i^*\Delta_n}-\tilde{\sigma}^2_{j\Delta_n}\big)\,.
\end{align*}
Exploiting the smoothness of $(\sigma_t)_{t\in[0,1]}$ by Assumption \ref{assVola}, we obtain
\begin{align*}
\max_{i=1,\ldots,n}\bigl|D_{n,i}\bigr| &= \max_{i^*-k_n \leq i \leq i^*}\bigl|D_{n,i}\bigr| \leq K \sup_{0 \leq s \leq 1}|\sigma_s| \max_{i^*-k_n \leq i \leq i^*}\sum_{j = i + 1}^{i + k_n}|\sigma_{j\Delta_n} - \sigma_{i^*\Delta_n} |
\\&\leq K\sup_{0 \leq s \leq 1}|\sigma_s| k_n^{1+\aalpha}n^{-\aalpha} = \mathcal{O}_{a.s.}\bigl(\sqrt{k_n \log (n)}\bigr)\,,
\end{align*}
with some constant $K$ by \eqref{assk}. Proceeding similarly as in the proof of Theorem \ref{thm1}, it follows that
\begin{align}
\max_{i=1,\ldots,k_n}\bigl|A_{n,i}+B_{n,i}+C_{n,i}\bigr| = \mathcal{O}_{\P}\bigl(\sqrt{k_n \log (n)}\bigr)\,.
\end{align}
Altogether, we conclude that
\begin{align}\sup_{t\in[0,\td]}\bigl|g(t)\bigr|=\mathcal{O}_{\P}\big(\sqrt{\log(n)}\big)\,.\end{align}

Finally, using (\ref{42}), we see that $f(i\Delta_n)>|g(i\Delta_n)|>0$ holds for each $i=i^*-k_n/2,\ldots,i^*$, with probability tending to 1. In particular,
\begin{align}\label{eq_repres_Vdiamond}
V_{n,i}^{\diamond} = \bigl|f\bigl(i\Delta_n\bigr)+g\bigl(i\Delta_n\bigr)\bigr| = f\bigl(i\Delta_n\bigr) + \operatorname{sign}\bigl({g}(i\Delta_n)\bigr)\bigl|{g}\bigl(i\Delta_n\bigr)\bigr|\,
\end{align}
for those $i$. Furthermore,
$$f(i^*\Delta_n)-f(i^*\Delta_n-\gamma_n)=\lfloor \gamma_n n\rfloor \delta k_n^{-1/2}~~\mbox{for}~~\gamma_n\in[0,k_n/(2n)]\,.$$
Thus, we choose $\gamma_n$ such that
\begin{align}
\frac{\sqrt{k_n \log (n)}}{\delta n} = \KLEINO\bigl(\gamma_n\bigr).
\end{align}
Now the assumptions of Lemma \ref{lem_argmax_bound} are fulfilled, and we obtain, with probability tending to one,
\begin{align*}
i^*\Delta_n \geq \operatorname{argmax}_{i=k_n,\ldots, i^*} V_{n,i}^{\diamond}\; \Delta_n \geq i^*\Delta_n - \gamma_n,
\end{align*}
through an application of Lemma \ref{lem_argmax_bound} together with \eqref{eq_repres_Vdiamond}. A similar argument for $i > i^*$ shows
\begin{align*}
i^*\Delta_n \leq \operatorname{argmax}_{i=i^*,\ldots, n-k_n} V_{n,i}^{\diamond}\; \Delta_n \leq i^*\Delta_n + \gamma_n,
\end{align*}
from which one obtains $|\widehat{\td}_n - i^* \Delta_n | = \mathcal{O}_{\P}\bigl(\gamma_n\bigr)$, which completes the proof by definition of $i^*$.\hfill\qed \\[.2cm]

{\bf{Proof of Proposition \ref{prop_time_est_optimal}.}} Suppose that
\begin{align}\label{eq_prop_time_est_optimal_1}
\sqrt{k_n} \delta_n = \KLEINO\bigl(\sqrt{\log (n)}\bigr)
\end{align}
and we have a consistent estimator $\widehat{\td}^*$ for $\td$. Define
\begin{align}
T_{\widehat{\td}^* } = \frac{n}{k_n}\bigg(\sum_{j = \widehat{\td}^* n- k_n}^{\widehat{\td}^* n-1} \big(\Delta_j^n X\big)^2 - \sum_{j = \widehat{\td}^* n + 1}^{\widehat{\td}^* n+k_n} \big(\Delta_j^n X\big)^2 \bigg)\,.
\end{align}
Using the statistic $T_{\widehat{\td}^*}$, we can now test for jumps in the volatility $\sigma_t^2$. Note that due to \eqref{eq_prop_time_est_optimal_1}, it readily follows that this new test has a detection boundary $b_n = \KLEINO\bigl(\sqrt{\log (n)}/\sqrt{k_n}\bigr)$. This, however, is a contradiction to Theorem \ref{thm_lowerbound}, and hence such an estimator $\widehat{\td}^*$ cannot exist.\hfill\qed\\
Finally, to prove Proposition \ref{thm_multiple:change-points} we may proceed exactly as in the proof of Proposition \ref{propest}.
\section{Proofs of Section \ref{sec:5}}
{\bf{Proof of Theorem \ref{upperglobal}.}}
\begin{lem}\label{propR}
Decompose $X_t$ from \eqref{sm} in $A_t=\int_0^t a_s\,ds$ and $M_t=X_0+\int_0^t \sigma_s\,dW_s$. Define for $i=2,\ldots,n$:
\begin{align}\label{R}R_{n,i}=n^2\Big(\big((\Delta_i^n M)^2-(\Delta_{i-1}^n M)^2\big)^2-\frac{2}{3}\big((\Delta_i^n M)^4+(\Delta_{i-1}^n M)^4\big)\Big)\,,\end{align}
\begin{align}\label{Udagger}U_n^{\dagger}=\frac{1}{\sqrt{n-1}}\max_{m=2,\ldots,n}\Big|\sum_{i=2}^m\Big(R_{n,i}-\frac{\sum_{i=2}^{n}R_{n,i}}{n-1}\Big)\Big|\,.\end{align}
Then $|V_n^{\dagger}-U_n^{\dagger}|\pn 0$.
\end{lem}
\begin{proof}
The inequality \eqref{VmU} implies that
\begin{align*}|V_n^{\dagger}-U_n^{\dagger}|\hspace*{-.05cm}\le \hspace*{-.05cm}\frac{1}{\sqrt{n-1}}\Big(\max_{m=2,\ldots,n}\Big|\sum_{i=2}^m\big(Q_{n,i}- R_{n,i}\big)\Big|+ \Big|\sum_{i=2}^{n}\big(Q_{n,i}- R_{n,i}\big)\Big|\Big)\,,\end{align*}
and it suffices to discuss the first addend. To this end, we start with the decomposition
\begin{align*}n^2\big(\hspace*{-.025cm}(\Delta_j^n X)^4\hspace*{-.025cm}-\hspace*{-.025cm}(\Delta_j^n M)^4\hspace*{-.025cm}\big)&\hspace*{-.025cm}=\hspace*{-.025cm}n^2\Big(\hspace*{-.025cm}4(\Delta_j^n M)^3\Delta_j^n A\hspace*{-.025cm}+\hspace*{-.025cm}6(\Delta_j^n M)^2(\Delta_j^n A)^2
+4\Delta_j^n M(\Delta_j^n A)^3+(\Delta_j^n A)^4\hspace*{-.025cm}\Big),\end{align*}
and from $|\Delta_j^n A| \leq Kn^{-1}$ and $\E[|\Delta_i^n M|^p] \leq Kn^{-p/2}$ by boundedness of all coefficients, it is easy to see that the only term to discuss is the leading one, i.e.\, $n^2 (\Delta_j^n M)^3\Delta_j^n A$. Under Assumption \ref{assvolaglobal}, one may approximate both processes inside $\Delta_j^n A$ and $\Delta_i^n M$ by their value at $(j-1)\Delta_n$, and the error due to this approximation has exactly the correct small order. For the remaining approximated versions of $n^2 (\Delta_j^n M)^3\Delta_j^n A$, standard martingale techniques prove the required bound. The term involving $n^2\big((\Delta_{j-1}^n X)^4-(\Delta_{j-1}^n M)^4\big)$ is discussed analogously. Finally, 
\begin{align*}&n^2\big((\Delta_{j-1}^n X\Delta_j^n X)^2-(\Delta_{j-1}^n M\Delta_j^n M)^2\big)=n^2\Big(\Delta_j^n A\big((\Delta_{j-1}^n M)^2\Delta_j^n A\\
&\hspace*{1cm}+2(\Delta_{j-1}^n M)^2\Delta_j^n M+2\Delta_{j-1}^n M\Delta_{j-1}^n A\Delta_{j}^n M+2\Delta_{j-1}^n M\Delta_{j-1}^n A\Delta_{j}^n A\big)\\
& \hspace*{1cm}+\Delta_{j-1}^n A\big((\Delta_{j}^n M)^2\Delta_{j-1}^n A+2(\Delta_{j}^n M)^2\Delta_{j-1}^n M+2\Delta_{j}^n M\Delta_{j-1}^n A\Delta_{j}^n A\big)\\
&\hspace*{1cm}+\Delta_j^n A\Delta_{j-1}^n A\big(\Delta_j^n A\Delta_{j-1}^n A+2\Delta_j^n M\Delta_{j-1}^n M\big)\Big).
\end{align*}
Again, the only relevant terms are those involving one factor of increments of $A$ and three factors of increments of $M$. The same reasoning as above gives the required result. 
\end{proof}
\begin{prop}\label{propupperglobal}On the hypothesis of\, Testing problem \ref{testingproblem} and Assumption \ref{assvolaglobal}, we have the functional convergence
\begin{align}\label{ucusumclt}
\frac{1}{\sqrt{n-1}}\sum_{i=2}^{\lfloor nt\rfloor}\Big(R_{n,i}-\E\bigl[(\Delta_i^n\varrho)^2\bigr]\Big)\stackrel{\omega-(st)}{\longrightarrow}\int_0^t v_s\,dB_s,
\end{align}
weakly in the Skorokhod space with $v_s^2=(80/3)\sigma_s^8$ and $(B_s)$ a Brownian motion independent of $\mathcal{F}$. 
\end{prop}
\begin{proof}
Denote the set of time instants $[i\Delta_n,(i+1)\Delta_n]$ subject to discontinuities in $(\sigma_t)_{t\ge 0}$ with $\mathcal{J}_n$. Since $|\mathcal{J}_n|<\infty$ for $n\rightarrow\infty$, we have that
\[\frac{1}{\sqrt{n-1}}\sum_{j \in \mathcal{J}_n}\,R_{n,j}\pn 0\]
as $n\rightarrow\infty$ on Assumption \ref{assvolaglobal} $(ii)$. Therefore, we may assume throughout this proof that \eqref{nusmooth} is satisfied, for simpler notation on $[0,1]$. Standard localization arguments allow us to assume that \eqref{rhosmooth} and \eqref{nusmooth} apply on $[0,1]$ with global constant $K$. It\^{o}'s formula gives
\begin{align*}C_i=(\Delta_i^n M)^2-\int_{(i-1)\Delta_n}^{i\Delta_n}\sigma_s^2\,ds=2\,\int_{(i-1)\Delta_n}^{i\Delta_n}(M_s-M_{(i-1)\Delta_n})\,dM_s\,,\end{align*}
which forms a martingale difference sequence in $i=1,\ldots,n$. For its quadratic variation, we obtain
\begin{align*}[C,C]_i&=4\int_{(i-1)\Delta_n}^{i\Delta_n}(M_s-M_{(i-1)\Delta_n})^2\,d[M,M]_s\\
&=4\int_{(i-1)\Delta_n}^{i\Delta_n}\Big(\int_{(i-1)\Delta_n}^{s}\sigma_t\,dW_t\Big)^2\,\sigma_s^2\,ds\,.\end{align*}
Applying It\^{o}'s formula to $(\Delta_i^n M)^4$ yields
\begin{align*}\frac{2}{3} (\Delta_i^n M)^4\hspace*{-.05cm}&=\hspace*{-.05cm}\frac{8}{3}\int_{(i-1)\Delta_n}^{i\Delta_n}\hspace*{-.15cm}\Big(\int_{(i-1)\Delta_n}^{s}\hspace*{-.15cm}\sigma_t\,dW_t\Big)^3\sigma_s\,dW_s\hspace*{-.05cm}+\hspace*{-.05cm}4\int_{(i-1)\Delta_n}^{i\Delta_n}\hspace*{-.15cm}\Big(\int_{(i-1)\Delta_n}^{s}\hspace*{-.15cm}\sigma_t \,dW_t\Big)^2\sigma_s^2\,ds\\[.1cm]
&=N_i+[C,C]_i\,,\end{align*}
with a martingale difference sequence $(N_i)_{1\le i\le n}$. We decompose
\begin{align}\nonumber R_{n,i}&=n^2\Big(C_i-C_{i-1}+\int_{(i-1)\Delta_n}^{i\Delta_n}(\sigma_s^2-\sigma_{s-\Delta_n}^2)\,ds\Big)^2-\frac{2}{3}\big((\Delta_i^n M)^4+(\Delta_{i-1}^n M)^4\big)\\
&\label{decomposition} =I_i+II_i+III_i\,,\\
&\hspace*{-.25cm}\mbox{with}\nonumber~~~I_i=n^2\big(C_i^2-[C,C]_i+C_{i-1}^2-[C,C]_{i-1}-2C_iC_{i-1}-N_i-N_{i-1}\big)\,,\\
&\nonumber~~~~~\,II_i=2n^2(C_i-C_{i-1})\int_{(i-1)\Delta_n}^{i\Delta_n}(\sigma_s^2-\sigma_{s-\Delta_n}^2)\,ds\,,\\
&\nonumber~~~~III_i=n^2\Big(\int_{(i-1)\Delta_n}^{i\Delta_n}(\sigma_s^2-\sigma_{s-\Delta_n}^2)\,ds\Big)^2\,.
\end{align}
The terms $(II_i)_{1\le i\le n}$ and $(III_i)_{1\le i\le n}$ are asymptotically negligible. We start proving that
\begin{align}\label{term2}\sup_{t\in[0,1]}\Big|\sum_{i=2}^{\lfloor n t\rfloor }
II_{i}\Big|=\KLEINO_{\P}(\sqrt{n})\,.\end{align}
It suffices to prove that
\begin{align}\label{term2a}n^{-1/2}\,\E\Big[\sup_{t\in[0,1]}\Big|\sum_{i=2}^{\lfloor n t\rfloor }2n^2(C_i-C_{i-1})\int_{(i-1)\Delta_n}^{i\Delta_n}(\sigma_s^2-\sigma_{s-\Delta_n}^2)\,ds\Big|\Big]\rightarrow 0\,.\end{align}
As by Assumption \ref{assvolaglobal} we have that
\begin{align*}\int_{(i-1)\Delta_n}^{i\Delta_n}(\sigma_s^2-\sigma_{s-\Delta_n}^2)\,ds&=\int_{(i-1)\Delta_n}^{i\Delta_n}(\nu_s-\nu_{s-\Delta_n})\,ds +\int_{(i-1)\Delta_n}^{i\Delta_n}(\varrho_s-\varrho_{s-\Delta_n})\,ds\,,\end{align*}
\eqref{term2a} is implied by 
\begin{subequations}
\begin{align}\label{term2b}n^{-1/2}\,\E\Big[\sup_{t\in[0,1]}\Big|\sum_{i=2}^{\lfloor n t\rfloor }2n^2(C_i-C_{i-1})\int_{(i-1)\Delta_n}^{i\Delta_n}(\nu_s-\nu_{s-\Delta_n})\,ds\Big|\Big]\rightarrow 0\,,\\ \label{term2c}n^{-1/2}\,\E\Big[\sup_{t\in[0,1]}\Big|\sum_{i=2}^{\lfloor n t\rfloor }2n^2(C_i-C_{i-1})\int_{(i-1)\Delta_n}^{i\Delta_n}(\varrho_s-\varrho_{s-\Delta_n})\,ds\Big|\Big]\rightarrow 0\,.\end{align}
\end{subequations}
By Cauchy-Schwarz and Jensen's inequality we derive that
\begin{align*}&2n^2\E\Big[\sup_{t\in[0,1]}\Big|\sum_{i=2}^{\lfloor n t\rfloor }(C_i-C_{i-1})\int_{(i-1)\Delta_n}^{i\Delta_n}(\nu_s-\nu_{s-\Delta_n})\,ds\Big|\Big]\\
&\quad \le2n^2\E\Big[\Big|\sum_{i=2}^{n}(C_i-C_{i-1})^2\sum_{i = 2}^n\Bigl(\int_{(i-1)\Delta_n}^{i\Delta_n}(\nu_s-\nu_{s-\Delta_n})\,ds\Bigr)^2\Big|^{1/2}\Big]\\
&\quad \le2n^2\E\Big[\Big|\sum_{i=2}^{n}(C_i-C_{i-1})^2\sum_{i = 2}^n\Delta_n\int_{(i-1)\Delta_n}^{i\Delta_n}(\nu_s-\nu_{s-\Delta_n})^2\,ds\Big|^{1/2}\Big]\\
&\quad \le2n^2\Big(\sum_{i,j=2}^{n}\bigl(\E\bigl[(C_i-C_{i-1})^4\bigr]\bigr)^{1/2}\Delta_n \E\Big[\Bigl(\int_{(j-1)\Delta_n}^{j\Delta_n}(\nu_s-\nu_{s-\Delta_n})^2\,ds\Bigr)^2\Big]^{1/2}\Big)^{1/2}\\
&\quad \le2n^2\Big(\sum_{i,j=2}^{n}\bigl(\E\bigl[(C_i-C_{i-1})^4\bigr]\bigr)^{1/2}\Delta_n^{3/2} \E\Big[\int_{(j-1)\Delta_n}^{j\Delta_n}(\nu_s-\nu_{s-\Delta_n})^4\,ds\Big]^{1/2}\Big)^{1/2}\\
&\quad =\mathcal{O}\big(n^{2} \Delta_n^{3/2 +\epsilon} \bigr)=\KLEINO(n^{1/2})\,,
\end{align*}
since $\sup_{|s-t|\le \Delta_n}\bigl(\E\bigl[(\nu_s-\nu_t)^4\bigr]\bigr)^{1/4}\le \Delta_n^{1/2+\epsilon}$ for some $\epsilon>0$ and \(\E[(C_i-C_{i-1})^4]=\mathcal{O}(\Delta_n^4)\).
This proves \eqref{term2b}. In order to verify \eqref{term2c}, we exploit independence of $(\varrho_s)_{s \in [0,1]}$ and $(W_s)_{s \in [0,1]}$ by Assumption \ref{assvolaglobal}. In particular, we have
\begin{align}\label{eq_martingale_prop_condi}\nonumber
&\E\Big[C_i\int_{(i-1)\Delta_n}^{i\Delta_n}(\varrho_s-\varrho_{s-\Delta_n})\,ds \Big| \sigma(\varrho_s,{s\in[0,1]})\cup \mathcal{F}_{(i-1)\Delta_n}\Big] \\&= \int_{(i-1)\Delta_n}^{i\Delta_n}(\varrho_s-\varrho_{s-\Delta_n})\,ds\;\E\Big[C_i \Big| \sigma(\varrho_s,{s\in[0,1]})\cup \mathcal{F}_{(i-1)\Delta_n}\Big] = 0\,,
\end{align}
and analogously for the term with $C_{i-1}$, such that the martingale property follows by iterated expectations.
Thus, by Burkholder and Cauchy-Schwarz inequality, with some constant $K$:

\begin{align*}
&\E\Big[\sup_{t\in[0,1]}\Big(\sum_{i=2}^{\lfloor n t\rfloor }2n^2(C_i-C_{i-1})\int_{(i-1)\Delta_n}^{i\Delta_n}(\varrho_s-\varrho_{s-\Delta_n})\,ds\Big)^2\Big]\\
&\quad\le K\,n^4 \sum_{i=2}^{ n } \E\Big[(C_i-C_{i-1})^2\Big(\int_{(i-1)\Delta_n}^{i\Delta_n}(\varrho_s-\varrho_{s-\Delta_n})\,ds\Big)^2\Big]\\&\quad\le
K\,n^4\sum_{i=2}^{ n } \big(\E\big[(C_i-C_{i-1})^4]\E\big[(\Delta_i^n\varrho)^4\big]\big)^{1/2} =\mathcal{O}( n^{1-2\aalpha}\big)\,,
\end{align*}
where we used \eqref{rhosmooth2} in Hypothesis \ref{testingproblem} in the last step. This implies \eqref{term2c} and thus \eqref{term2a}. 
Next, we show that
\begin{align}\label{term3}\sup_{t\in[0,1]}\Big|\sum_{i=2}^{\lfloor n t\rfloor } \Big(
III_{i} - \E\bigl[(\Delta_i^n \varrho)^2\bigr]\Big)\Big|=\KLEINO_{\P}(\sqrt{n})\,.\end{align}
Decomposing $\sigma_s^2=\nu_s+\varrho_s$ again, the main point is to guarantee
\begin{align}\label{term3a}\sup_{t\in[0,1]}\Big|\sum_{i=2}^{\lfloor n t\rfloor }\Big((\Delta_i^n\varrho)^2-\E\big[(\Delta_i^n\varrho)^2\big]\Big)\Big|=\KLEINO_{\P}(\sqrt{n})\,,
\end{align}
as the other terms may be treated as for $(II_i)_{1\le i\le n}$ above. \eqref{term3a} is ensured by Assumption \ref{assvolaglobal} $(iv)$.\\
Denote by $U_i=C_i^2-[C,C]_i-N_i$, which forms a martingale difference sequence $(U_i)_{1\le i\le n}$. Observe that
\[\sum_{i=2}^m I_i=n^2\sum_{i=2}^m (U_i+U_{i-1}-2C_i C_{i-1})=n^2 \sum_{i=2}^{m-1}(2U_i-2C_iC_{i-1})+U_m-U_1\,,\]
which is a martingale sequence plus an asymptotically negligible remainder. We establish a functional stable central limit theorem for
\begin{align}\label{mart}\sum_{i=2}^{\lfloor nt\rfloor } Z_{n,i}=\frac{n^2}{\sqrt{n}} \sum_{i=2}^{\lfloor nt\rfloor }(2U_i-2C_iC_{i-1})\end{align}
based on Theorem 3--1 by \cite{jacod1}. Denote with $\mathcal{G}_{i,n}=\mathcal{F}_{i\Delta_n}=\sigma\big(\nu_s,W_s,\varrho_s;s\le i\Delta_n\big)$. Since we consider a martingale directly, the drift condition in Theorem 3--1 by \cite{jacod1} is trivial. Thus, conditions
\begin{subequations}
\begin{align}\label{J2}
\sum_{i=2}^{\lfloor nt\rfloor } \E\big[Z_{n,i}^2|\mathcal{G}_{i-1,n}\big]\pn \int_0^t v_s^2\,ds
\end{align}
with a predictable process $(v_s)_{s\ge 0}$, for all $\epsilon>0$:
\begin{align}\label{J3}
\sum_{i=2}^{\lfloor nt\rfloor } \E\big[Z_{n,i}^2\1_{\{Z_{n,i}>\epsilon\}}|\mathcal{G}_{i-1,n}\big]&\pn 0\,,\\
\label{J4}
\sum_{i=2}^{\lfloor nt\rfloor } \E\big[Z_{n,i}^2 (\mathcal{M}_{i\Delta_n}-\mathcal{M}_{(i-1)\Delta_n})|\mathcal{G}_{i-1,n}\big]&\pn 0\,,
\end{align}
\end{subequations}
for all bounded $\mathcal{G}_t$-martingales $(\mathcal{M}_t)_{t\ge 0}$ with $\mathcal{M}_0=0$ and $[W,M]=0$ or for $\mathcal{M}_t=W_t$,
imply the functional $(\mathcal{G}_t)$-stable limit theorem
\begin{align}\label{fsclt}\sum_{i=2}^{\lfloor nt\rfloor } Z_{n,i}\stackrel{\omega-(st)}{\longrightarrow}\int_0^t v_s\,dB_s\,,\end{align}
with $(B_t)_{t\ge 0}$ a Brownian motion defined on an orthogonal extension of $(\Omega,\mathcal{G},(\mathcal{G}_t),\P)$.
It\^{o}'s formula can be used to determine the relation
\begin{align}
\notag\var\big((\Delta_i^n M)^4|\mathcal{G}_{i-1,n}\big)=96\,\Delta_n^4\big(\sigma_{(i-1)\Delta_n}^8+\mathcal{O}_{a.s.}\big(n^{-(\aalpha\wedge 1/2)}\big)\big)\,,
\end{align}
and similar ones for other powers, which also readily obtained by the mixed normality $\Delta_i^n M\sim MN (0,\int_{(i-1)\Delta_n}^{i\Delta_n}\sigma_s^2\,ds)$ and approximation of $\sigma$ to obtain measurability with respect to $\mathcal{G}_{i-1,n}$. 
Using the form of $Z_{n,i}$ from \eqref{mart}, we derive for the left hand side of \eqref{J2}
\begin{align*}\sum_{i=2}^{\lfloor nt\rfloor}\E\big[Z_{n,i}^2|\mathcal{G}_{i-1,n}\big]=n^3\sum_{i=2}^{\lfloor nt\rfloor}\Big(4\E\big[U_i^2|\mathcal{G}_{i-1,n}\big]+4\E\big[C_i^2C_{i-1}^2|\mathcal{G}_{i-1,n}\big]-8\E\big[U_iC_iC_{i-1}|\mathcal{G}_{i-1,n}\big]\Big)\,.\end{align*}
For a simpler notation we consider the three terms consecutively. First,
\begin{align*}n^3\sum_{i=2}^{\lfloor nt\rfloor}4\E\big[C_i^2C_{i-1}^2|\mathcal{G}_{i-1,n}\big]&=n^3\sum_{i=2}^{\lfloor nt\rfloor}4\Big((\Delta_{i-1}^n M)^2-\int_{(i-2)\Delta_n}^{(i-1)\Delta_n }\sigma_s^2\,ds\Big)^2\E\big[C_i^2|\mathcal{G}_{i-1,n}\big]\\
&=8n\sum_{i=2}^{\lfloor nt\rfloor}\big(\sigma_{(i-2)\Delta_n}^4+\mathcal{O}_{a.s.}\big(n^{-(\aalpha\wedge 1/2)}\big)\big)\Big((\Delta_{i-1}^n M)^2-\int_{(i-2)\Delta_n}^{(i-1)\Delta_n }\sigma_s^2\,ds\Big)^2\\
&\stackrel{\P}{\longrightarrow } 16\int_0^t \sigma_s^8\,ds\,.
\end{align*}
The last convergence is very close to the usual analysis of the variance of realized volatility, see Section 5 of \cite{bn}. We frequently use It\^{o}'s formula, especially for the second term:
\begin{align*}&n^3\sum_{i=2}^{\lfloor nt\rfloor}4\E\big[U_i^2|\mathcal{G}_{i-1,n}\big]=n^3\sum_{i=2}^{\lfloor nt\rfloor}4\E\Big[\Big((\Delta_{i}^n M)^2-\int_{(i-1)\Delta_n}^{i\Delta_n }\sigma_s^2\,ds\Big)^4+\frac49 (\Delta_{i}^n M)^8  \Big|\mathcal{G}_{i-1,n}\Big]\\
&\hspace*{6cm}-2 \E\Big[\Big((\Delta_{i}^n M)^2-\int_{(i-1)\Delta_n}^{i\Delta_n }\sigma_s^2\,ds\Big)^2\,\frac23 (\Delta_{i}^n M)^4  \Big|\mathcal{G}_{i-1,n}\Big]\\
&=n^{-1}\sum_{i=2}^{\lfloor nt\rfloor}4\big(\sigma_{(i-1)\Delta_n}^8+\mathcal{O}_{a.s.}\big(n^{-(\aalpha\wedge 1/2)}\big)\big)\Big(60+105\cdot \frac49-156 \cdot \frac23\Big)\\
&\stackrel{\P}{\longrightarrow } \frac{32}{3} \,\int_0^t \sigma_s^8\,ds\,.
\end{align*}
Finally, $\E\big[U_iC_iC_{i-1}|\mathcal{G}_{i-1,n}\big]=0$, since 
\[\E\big[U_iC_i|\mathcal{G}_{i-1,n}\big]=\E\Big[\Big(\hspace*{-.05cm}(\Delta_{i}^n M)^2-\int_{(i-1)\Delta_n}^{i\Delta_n }\hspace*{-.05cm}\sigma_s^2\,ds\Big)^3-\frac23 (\Delta_{i}^n M)^4 \Big(\hspace*{-.05cm}(\Delta_{i}^n M)^2-\int_{(i-1)\Delta_n}^{i\Delta_n }\hspace*{-.05cm}\sigma_s^2\,ds\Big) \Big|\mathcal{G}_{i-1,n}\Big]\]
vanishes. Thereby we derive \eqref{J2} with 
\begin{align}v_s^2=(80/3)\,\sigma_s^8\,.\end{align}
The Lindeberg criterion \eqref{J3} is ensured by the stronger Lyapunov condition
\begin{align*}
\sum_{i=2}^{\lfloor nt\rfloor } \E\big[Z_{n,i}^4|\mathcal{G}_{i-1,n}\big]=n^6\sum_{i=2}^{\lfloor nt\rfloor } \E\big[(2U_i-2C_iC_{i-1})^4|\mathcal{G}_{i-1,n}\big]=\mathcal{O}_{\P}(\Delta_n)&\pn 0\,.
\end{align*}
Condition \eqref{J4} for $\mathcal{M}=W$ follows with It\^{o}'s formula by the fact that $\E[(\Delta_i^n W)^p]=0$ and $\E[(\Delta_i^n M)^p]$ $=\KLEINO_{\P}(1)$ for all odd $p\ge 1$, using the usual approximation. It\^{o}'s formula also yields for any $p$ and bounded martingales $(\mathcal{M})_{t\ge 0}$ that \((\int_{(i-1)\Delta_n}^{i\Delta_n}\sigma_s\,dW_s)^p\) \(\int_{(i-1)\Delta_n}^{i\Delta_n}d\mathcal{M}_s\) equals a negligible remainder plus the term $\int_{(i-1)\Delta_n}^{i\Delta_n}\sigma_s^p\,d[W,\mathcal{M}]_s$ which vanishes, such that \eqref{J4} is satisfied. This completes the proof of \eqref{fsclt}.
\end{proof} 
Theorem \ref{upperglobal} follows with Lemma \ref{propR} and Proposition \ref{propupperglobal} by continuous mapping theorem.\\[.2cm]
{\bf{Proof of Proposition \ref{corrupperglobal}.}}
The condition $\aalpha>1/4$ guarantees that \[n^{-1/2}\,\sum_{i=1}^n \big(\sigma_{(i-1)\Delta_n}^2-\sigma_{(i-2)\Delta_n}^2\big)^2\pn 0\,,\]and analogously with fourth moments, since $n^{1/2-2\aalpha}\rightarrow 0$. Hence, a consistent estimator of $\sigma_{(i-2)\Delta_n}^4$ can be used to standardize $(Q_{n,i})_{K_n+1\le i\le n}$ and $(R_{n,i})_{K_n+1\le i\le n}$. Thereto, we write\\
\(\bar R_{n,i}=(\hat\sigma_{(i-2)\Delta_n}^4)^{-1}R_{n,i}\) with the estimator and $K_n$ from \eqref{eq_vola_est_for_boot_1} and
\[\bar U_n^{\dagger}=\frac{1}{\sqrt{n-1}}\max_{m=K_n+1,\ldots,n}\Big|\sum_{i=K_n+1}^m\Big(\bar R_{n,i}-\frac{\sum_{i=K_n+1}^{n}\bar R_{n,i}}{n-1}\Big)\Big|.\]
The proof is now traced back to the one of Theorem \ref{upperglobal}. Set
\[\mathcal{A}_i=\Big\{\hat\sigma_{(i-2)\Delta_n}^4\ge \sigma_{(i-2)\Delta_n}^4/2\Big\}\,.\]
Since $(\sigma_s^2)_{s\ge 0}$ is bounded from below and for $K_n\rightarrow\infty$ the estimator is consistent, based on Markov's inequality we obtain analogously as in the proof of Equation (22) of \cite{vett2012}, that $\P(\mathcal{A}_i^{\complement})=\KLEINO(n^{-1})$, such that 
\begin{align}\label{lastref}\P\Big(\bigcap_{i=K_n+1}^n\mathcal{A}_i\Big)\ge 1-\sum_{i=K_n+1}^n \P\big(\mathcal{A}_i^{\complement}\big)=1-\KLEINO(1)\,.\end{align}
Then, we use decomposition \eqref{decomposition} for $R_{n,i}$ and an analogous one for $\bar R_{n,i}$. By the fact that $\hat\sigma_{(i-2)\Delta_n}^4$ is $\mathcal{G}_{i-2,n}=\mathcal{F}_{(i-2)\Delta_n}$-measurable, the leading term of the latter has a martingale structure again. Similar estimates as in the proof of Proposition \ref{propupperglobal} yield that $|\tilde U_n^{\dagger}-\bar U_n^{\dagger}|\pn 0$ for
\[\tilde U_n^{\dagger}=\frac{1}{\sqrt{n-1}}\max_{m=K_n+1,\ldots,n}\Big|\sum_{i=K_n+1}^m\Big(\frac{R_{n,i}\1(\mathcal{A}_i)}{\sigma_{(i-2)\Delta_n}^4}-\frac{\sum_{i=K_n+1}^{n}R_{n,i}\1(\mathcal{A}_i)}{(n-1)\sigma_{(i-2)\Delta_n}^4}\Big)\Big|\,.\]
Observe that the analogue of the leading terms $(I_i)$ in the decomposition \eqref{decomposition} for decomposing $\tilde U_n^{\dagger}$ remain a martingale difference sequence (plus a negligible remainder), since $\sigma_{(i-2)\Delta_n}^4$ is $\mathcal{G}_{i-2,n}=\mathcal{F}_{(i-2)\Delta_n}$-measurable. Analogous terms to $II_i$ and $III_i$ in \eqref{term2} are handled as above. This readily proves Proposition \ref{corrupperglobal}.\\[.2cm]

{\bf{Proof of Theorem \ref{thm:bootstrap:global}.}}
First, let us state two preliminary Lemmas \ref{lem:aux:globalboot:1} and \ref{lem:global:statistic:is:cont}.
\begin{lem}\label{lem:aux:globalboot:1}
On Assumption \ref{assvolaglobal} it holds that $\begin{aligned}\E\bigl[|V^{\dagger} - \hat{V}_n^{\dagger}| \bigl|\F\bigr] = \KLEINO_{\P}(1)
\end{aligned}$.
\end{lem}

\begin{proof}[Proof of Lemma \ref{lem:aux:globalboot:1}]
Introduce the interpolated volatility
\begin{align}
\widetilde{\sigma}_t^4  = \left\{\begin{array}{cl} 0, & \mbox{if $t \in [0,K_n\Delta_{n})$,}\\ \hat{\sigma}_{iK_n\Delta_n}^4, & \mbox{if $t \in \bigl(iK_n\Delta_n, (i+1)K_n\Delta_n\bigr]$ and $1 \leq i  \leq n/K_n$.} \end{array}\right.
\end{align}
Observe that with a standard Brownian motion $(B_s)_{s\ge 0}$ independent of $\cal F$, we can write
\begin{align}\label{eq_thm_bootstrap_representation_1}
\hat{S}_{nt} = \int_0^t \widetilde{\sigma}_s^4 \,d B_s, \quad t \in [0,1],
\end{align}
setting in \eqref{boots} $\sqrt{K_n\Delta_n}Z_i = B_{(i+1)K_n\Delta_n} - B_{i K_n\Delta_n}$. Let $\delta_n \to 0$ and put $\mathcal{A}_n = \{\sup_{0 \leq t \leq 1}|\widetilde{\sigma}_t^4 - \sigma_t^4| \leq \delta_n^2 \}$. Then it follows similarly as for \eqref{lastref} that $\P\bigl(\mathcal{A}_n^{\complement}\bigr) \to 0$. Since $\mathcal{A}_n \in \F$, we thus obtain
\begin{align}\label{eq_thm_bootstrap_3}
\P\bigl(\E\bigl[|V^{\dagger} - \hat{V}_n^{\dagger}| \bigl|\F\bigr] \geq \delta_n \bigr) \leq \P\bigl(\E\bigl[|V^{\dagger} - \hat{V}_n^{\dagger}|\1(\mathcal{A}_n)\bigl|\F\bigr] \geq \delta_n \bigr) + \KLEINO(1).
\end{align}
Moreover, by \eqref{eq_thm_bootstrap_representation_1}, the triangle, Markov, Burkholder and Jensen inequality, we obtain with generic constant $K$:
\begin{align}\label{eq_thm_bootstrap_4} \nonumber
\P\bigl(\E\bigl[|V^{\dagger} - \hat{V}_n^{\dagger}|\1(\mathcal{A}_n)\bigl|\F\bigr] \geq \delta_n \bigr) &\leq \P\Bigl(\E\Bigl[\sup_{0 \leq t \leq 1}\Bigl|\int_0^t(\widetilde{\sigma}_s^4 -{\sigma}_s^4 )dB_s \Bigr|\1(\mathcal{A}_n)\bigl|\F\Bigr] \geq \delta_n/2 \Bigr) \\& \hspace*{-.9cm}\nonumber \le K\, \delta_n^{-1} \E\Bigl[\sup_{0 \leq t \leq 1}\Bigl|\int_0^t(\widetilde{\sigma}_s^4 -{\sigma}_s^4 )\1(|\widetilde{\sigma}_s^4 -{\sigma}_s^4 |\leq \delta_n^2)dB_s\Bigr|\Bigr]\\& \hspace*{-.9cm}\le K\, \delta_n^{-1} \biggl(\int_0^1 \E\bigl[(\widetilde{\sigma}_s^4 -{\sigma}_s^4 )^2\1(|\widetilde{\sigma}_s^4 -{\sigma}_s^4 |\leq \delta_n^2)\bigr] ds \biggr)^{1/2} \leq \delta_n\,.
\end{align}
Combining \eqref{eq_thm_bootstrap_3} and \eqref{eq_thm_bootstrap_4}, the claim follows.
\end{proof}

\begin{lem}\label{lem:global:statistic:is:cont}
Let $(B_s)$ be a standard Brownian motion independent of $\F$. Grant Assumption \ref{assvolaglobal}. The distribution of $V^{\dagger}$ conditional on $\F$ is uniformly continuous on compact sets $\mathcal{K}$ with $0 \not \in \mathcal{K}$, i.e; for any $\varepsilon > 0$ there exists a $\delta > 0$ (depending on $\mathcal{K}$) such that
\begin{align}
\sup_{x \in \mathcal{K}}\sup_{|y|\leq\delta}\bigl|\P\bigl(V^{\dagger} \leq x \bigl| \F \bigr) - \P\bigl(V^{\dagger} \leq x + y \bigl| \F \bigr) \bigr| \leq \varepsilon.
\end{align}
\end{lem}
\begin{proof}[Proof of Lemma \ref{lem:global:statistic:is:cont}]
Denote with
\begin{align*}
w_t = \int_{0}^{t} \sigma_s^4\, d B_s - t \int_0^{1} \sigma^4_s \,d B_s,  \quad t \in [0,1],
\end{align*}
and for $r > 0$ we partition $[0,1] = [0,r] \cup (r,1-r] \cup (1-r,1]$, and $(r,1-r) = \cup_{i = 0}^{d-1} (t_i, t_{i+1}]$, where $t_0 = r, t_d = 1-r$, and for some $c>0$: $|t_{i+1} - t_i| \le c d^{-1}$. Since $\inf_{0 \leq  t \leq 1}\sigma_t^2 \geq C > 0$, we have
\begin{align*}
\varsigma_{r}^- = \min_{r \leq t_i \leq 1 - r}\E\bigl[w_{t_i}^2\bigl|\F\bigr] &= \min_{r \leq t_i \leq 1 - r}\biggl((1 - t_i)^2\int_0^{t_i} \sigma_s^8\, ds + t_i^2 \int_{t_i}^{1} \sigma_s^8 \,ds\biggr)\\&\geq \min_{r \leq t_i \leq 1- r}(1 - t_i)t_i C^4 = (1-r)r C^4 > 0.
\end{align*}
Similarly, since $\sup_{0 \leq t\leq 1} \sigma_t^2 \leq K$, we have $\max_{r \leq t_i \leq 1- r}\E[w_{t_i}^2|\F] \leq \varsigma_{r}^+ < \infty$. Next, the uniform bound and Burkholder's inequality yield with generic constant $K$ on Assumption \ref{assvolaglobal}
\begin{align}\label{eq:lem:global:statistic:is:cont:2}
\P\bigl(\sup_{r \leq t \leq 1-r}\min_{1 \leq i \leq d}|w_t - w_{t_i}| \geq x \bigl| \F \bigr) &\leq \P\bigl(\max_{1 \leq i \leq d-1}\sup_{t_{i} \leq t \leq t_{i+1}}|w_t - w_{t_i}| \geq x \bigl| \F \bigr) \\& \nonumber \le K\,x^{-p} \sum_{i = 1}^{d-1} \E\biggl[ \biggl(\int_{t_{i}}^{t_{i+1}}\sigma_s^8\,ds \biggr)^{p/2} \biggl| \F\biggr] \\& \nonumber\le K\,x^{-p} \E\bigl[\sup_{0 \leq t \leq 1}|\sigma_t^4|^p \bigl| \F \bigr] \sum_{i = 1}^d (t_{i+1} - t_i)^{p/2} \le K\, x^{-p} d^{-p/2 + 1}.
\end{align}
In a similar manner, we also derive that
\begin{align}\label{eq:lem:global:statistic:is:cont:3}
\P\bigl(\sup_{0 \le t \leq r}|w_t|, \sup_{1-r \le t \leq 1}|w_t| \geq x \bigl| \F \bigr) \le K\, x^{-p} r^{p/2}.
\end{align}
Using \eqref{eq:lem:global:statistic:is:cont:2} and \eqref{eq:lem:global:statistic:is:cont:3}, we obtain the lower bound
\begin{align}\label{eq:lem:global:statistic:is:cont:4}
\P\bigl(V^{\dagger} \leq x \bigl| \F \bigr) \geq \P\bigl(\max_{1 \leq i \leq d} \pm w_i \leq x - y \bigl| \F \bigr) - y^{-p} d^{-p/2 + 1} - x^{-p}r^{p/2},
\end{align}
and the upper bound
\begin{align}\label{eq:lem:global:statistic:is:cont:5}
\P\bigl(V^{\dagger} \leq x \bigl| \F \bigr) \leq \P\bigl(\max_{1 \leq i \leq d} \pm w_{t_i} \leq x \bigl| \F \bigr).
\end{align}
Combining both \eqref{eq:lem:global:statistic:is:cont:4} and \eqref{eq:lem:global:statistic:is:cont:5} in turn yields
\begin{align}\label{eq:lem:global:statistic:is:cont:6}\nonumber
&\bigl|\P\bigl(V^{\dagger} \leq x + \delta \bigl| \F \bigr) - \P\bigl(V^{\dagger} \leq x \bigl| \F \bigr) \bigr| \\&\leq \bigl|\P\bigl(\max_{1 \leq i \leq d} \pm w_{t_i} \leq x + \delta \bigl| \F \bigr) - \P\bigl(\max_{1 \leq i \leq d} \pm w_i \leq x - \delta \bigl| \F \bigr) \bigr| + \delta^{-p} d^{-p/2 + 1} + x^{-p} r^{p/2}.
\end{align}
Since $0 < \varsigma_{r}^{\pm} < \infty$ for any $r > 0$, an application of Lemma 2.1 in \cite{kato} yields that
\begin{align}\label{eq:lem:global:statistic:is:cont:7}
\bigl|\P\bigl(\max_{1 \leq i \leq d} \pm w_{t_i} \leq x + \delta \bigl| \F \bigr) - \P\bigl(\max_{1 \leq i \leq d} \pm w_{t_i} \leq x \bigl| \F \bigr) \bigr|\leq C(\varsigma_{r}^{\pm}) \delta \sqrt{1 \vee \log(d/\delta)},
\end{align}
with $C(\varsigma_{r}^{\pm}) < \infty$ depending on $\varsigma_{r}^{\pm}$. Since $\mathcal{K}$ is compact and $0 \not \in \mathcal{K}$, we have $\sup_{x \in \mathcal{K}}x^{-p} \leq C_p(\mathcal{K}) < \infty$. For $p > 2$, select $d = \delta^{-q}$ with positive $q$ satisfying $2(p+q) < pq$. Then combining \eqref{eq:lem:global:statistic:is:cont:6} with \eqref{eq:lem:global:statistic:is:cont:7} yields
\begin{align*}
&\bigl|\P\bigl(V^{\dagger} \leq x + \delta \bigl| \F \bigr) - \P\bigl(V^{\dagger} \leq x \bigl| \F \bigr) \bigr| \le  2C(\varsigma_{r}^{\pm}) \delta (q\log(\delta) + 1) + C_p(\mathcal{K}) r^{p/2} + \delta^{-(p+q) + pq/2}.
\end{align*}
Given $\epsilon > 0$, we first select $r>0$ such that $C_p(\mathcal{K}) r^{p/2} \leq \epsilon/2$. This defines $C(\varsigma_{r}^{\pm})$. We may then select $\delta > 0$ such that $2C(\varsigma_{r}^{\pm}) \delta (q\log(\delta) + 1) + \delta^{-(p+q) + pq/2} \leq \epsilon/2$, hence the claim follows.
\end{proof}

We are now ready to proceed to the proof of Theorem \ref{thm:bootstrap:global}. Let $x > 0$. Due to Lemma \ref{lem:aux:globalboot:1} and \ref{lem:global:statistic:is:cont}, there exist $\delta_n,\epsilon_n \to 0$ such that
\begin{align*}
\P\bigl({V}^{\dagger} \leq x \big|\F\bigr) &\leq \P\bigl(\hat{V}_n^{\dagger} \leq x + |{V}^{\dagger} - \hat{V}_n^{\dagger}| \big|\F\bigr) \leq \P\bigl(\hat{V}_n^{\dagger} \leq x + \delta_n \big|\F\bigr) + \KLEINO_{\P}(1) \\&\leq \P\bigl(\hat{V}_n^{\dagger} \leq x \big|\F\bigr) + \epsilon_n+ \KLEINO_{\P}(1) \leq \P\bigl({V}^{\dagger} \leq x \big|\F\bigr) + 2\epsilon_n +\KLEINO_{\P}(1).
\end{align*}
By Lemma \ref{lem:global:statistic:is:cont}, $q_{\alpha}(V^{\dagger}|\F)$ is continuous for $\alpha > 0$. Hence we conclude from the above
that for any $\alpha > 0$
\begin{align}\label{eq:quantiles:converge}
\bigr|{q}_{\alpha}({V}^{\dagger}|\F) - \hat{q}_{\alpha}(\hat{V}_n^{\dagger}|\F) \bigr| = \KLEINO_{\P}(1).
\end{align}
Together with Theorem \ref{upperglobal} and \eqref{eq:quantiles:converge} this yields
\begin{align}
\P\bigl(\bar{V}_n^{\dagger} \leq \hat{q}_{\alpha}(\hat{V}_n^{\dagger}|\F) \bigr) \to \P\bigl(V^{\dagger} \leq  {q}_{\alpha}({V}^{\dagger}|\F) \bigr) = \alpha,
\end{align}
which completes the proof.\\[.2cm]

{\bf{Proof of Theorem \ref{lowerglobal}.}}
Denote by $(Z_i)_{1\le i\le n}$ a sequence of i.i.d.\,standard normally distributed random variables. Let $(U_i)_{1\le i\le n}$ be an i.i.d.\,sequence of random variables with bounded support and a symmetric distribution and such that $\E[U_i]=0$, $\E[U_i^2]=1$. We then consider the following special model for the volatility $\sigma_s$. 
\begin{align}\label{modellow}
\sigma_s = \bigl|1 + \vartheta_n U_{i}\bigr|, \quad s \in [(i-1)/n,i/n)\;,
\end{align}
and the associated observed process
\(
\big(\Delta_i^n X = \int_{(i-1)/n}^{i/n} \sigma_s d W_s = \bigl|1 +\vartheta_n U_{i}\bigr| Z_i \sqrt{\Delta_n}\big)_{1\le i\le n}
\).
The sequence of parameters $\vartheta_n$ is introduced here to determine the roughness by fluctuations of $\sigma_s$. We shall later see how it is directly related to $\vartheta_n$ in \eqref{rhosmooth}.

\begin{theo}\label{thm_lower_global_aux_theta}
Suppose we make observations $(\Delta_i^n X)_{1 \leq i \leq n}$ in model \eqref{modellow}. For parameter sequences $\vartheta_n, \vartheta_{n,0}, \vartheta_{n,1} \in [0,1)$, consider for $0<\td<1$ the null hypothesis $\mathcal{H}_0$ and alternative $\mathcal{H}_1$
\begin{itemize}
\item[$\mathcal{H}_0$:] $\vartheta_{n}= \vartheta_{n,0}$,~ for $1 \leq i \leq n$,
\item[$\mathcal{H}_1$:] $\vartheta_n = \vartheta_{n,1 }\neq \vartheta_{n,0}$,~ for $ \lfloor \td n\rfloor +1 \leq i \leq n$.
\end{itemize}
Then for $\tilde b_n = \KLEINO(n^{-1/2})$ and $|\vartheta_{n,0}^2 - \vartheta_{n,1}^2| \leq \tilde b_n$ we have
\begin{align*}
 \inf_{\psi} \gamma_{\psi}\bigl(\aalpha,\tilde b_n \bigr)  \to 1 \quad \text{as $n \to \infty$}. 
\end{align*}
with the notion of the global testing error in Equation \eqref{defn_minimax_optimal_test}.
\end{theo}

\begin{proof}[Proof of Theorem \ref{thm_lower_global_aux_theta}]
First, we may simplify the experiment slightly as long as we increase the provided information. Let us assume that it is known a priori that the time of change is $\td = 1/2$, basically to simplify the notation. Also, consider the special shifted case where we set $\vartheta_{n,0} = 0$, and hence only the change $\vartheta_{n,1}$ is of interest. The general case readily follows by a simple translation argument. Since $\vartheta_{n,0} = 0$, the sample $(\Delta_i^n X)_{1 \leq i \leq \lfloor n/2\rfloor}$ does not carry any information (ancillary statistic) about the parameter $\vartheta_{n,1}$, and it suffices to consider observations $(\Delta_i^n X)_{\lfloor n/2\rfloor +1 \leq i \leq n}$ by sufficiency. We shift indices such that we consider in the sequel the two statistical experiments
\begin{itemize}
\item[$\mathcal{E}_0$]: Observe $\{Z_j\}_{1 \leq j \leq \lfloor n/2\rfloor }$,
\item[$\mathcal{E}_1$]: Observe $\{Y_j\}_{1 \leq j \leq \lfloor n/2\rfloor}$ with $Y_j = |1 + \vartheta_{n,1} U_{j}| Z_j$.
\end{itemize}
We denote the associated (overall) probability measures with $\Q_0$ resp. $\Q_1$. Following the steps as in the proof of Theorem \ref{thm_lowerbound} we derive an analogous criterion to Equation \eqref{eq_likelohood_ratio_to_one}. Using Pinsker's inequality, we see that it suffices to establish
\begin{align}\label{eq_thm_lower_global_aux_theta_1}\nonumber
1 - \frac{1}{2}\bigl\|\Q_1 - \Q_0\bigr\|_{TV} &\geq 1 - \frac{1}{\sqrt{2}}\sqrt{\mathbf{D}\bigl(\Q_1\|\Q_0\bigr)} \\&= 1 - \KLEINO\bigl(1\bigr).
\end{align}
With a Lebesgue density $f_{U}$ of $U_1$, we can express the density $f_Y$ of $Y_1$ as
\[f_Y(y)=\int f_{Y|U_1=u_1}(y)f_U(u_1)\,du_1=\int \frac{1}{\sqrt{2\pi}\,|1+\vartheta_{n,1} u_1|}\exp{\Big(\frac{-y^2}{2(1+\vartheta_{n,1} u_1)^2}\Big)}f_U(u_1)\,du_1\,.\]
Considering $\vartheta_{n,1}\rightarrow 0$, we can drop the absolute value above. In the following we have to prove that above specified choices of $\tilde b_n$ and $\vartheta_{n,1}$ imply convergence of the Kullback-Leibler divergence to zero.
\begin{align*}&\E_{\Q_1}\Big[\log{\frac{d\Q_1}{d\Q_0}}\Big]=\\
&=\int\hspace*{-.1cm}\ldots\hspace*{-.1cm}\int\log\Bigg(\int \hspace*{-.1cm}\ldots \hspace*{-.1cm} \int \prod_{j=1}^{\lfloor n/2\rfloor}\exp{\Big(-\log(1+\vartheta_{n,1} u_j)-\frac{y_j^2}{2}\big((1+\vartheta_{n,1} u_j)^{-2}-1\big)\Big)}\\
&\hspace*{5cm} f_U({u_1})du_1\ldots f_U({u_{\lfloor n/2\rfloor}})du_{\lfloor n/2\rfloor}\Bigg)f_{Y}(y_1)dy_1\ldots f_{Y}(y_{\lfloor n/2\rfloor})dy_{\lfloor n/2\rfloor}\\
&=\int\hspace*{-.1cm}\ldots\hspace*{-.1cm}\int\log\Bigg(\int \hspace*{-.1cm}\ldots \hspace*{-.1cm} \int \prod_{j=1}^{\lfloor n/2\rfloor}\exp\Big(-\vartheta_{n,1} u_j+\frac{(\vartheta_{n,1} u_j)^2}{2}+\frac{y_j^2}{2}\big(2\vartheta_{n,1} u_j-3\vartheta_{n,1}^2u_j^2\big)\\
& \hspace*{1cm} +\mathcal{O}(\vartheta_{n,1}^3u_j^3+\vartheta_{n,1}^4u_j^4)\Big)\, f_U({u_1})du_1\ldots f_U({u_{\lfloor n/2\rfloor}})du_{\lfloor n/2\rfloor}\Bigg)f_{Y}(y_1)dy_1\ldots f_{Y}(y_{\lfloor n/2\rfloor})dy_{\lfloor n/2\rfloor}\\
&=\int \hspace*{-.1cm}\ldots\hspace*{-.1cm} \int\log\Bigg(\int \hspace*{-.1cm} \ldots \hspace*{-.1cm}\int \prod_{j=1}^{\lfloor n/2\rfloor}\exp{\Big((y_j^2-1)\vartheta_{n,1} u_j+\frac{(\vartheta_{n,1} u_j)^2}{2}\big(1-3y_j^2\big)+\mathcal{O}(\vartheta_{n,1}^3u_j^3+\vartheta_{n,1}^4u_j^4)\Big)}\\
&\hspace*{4.5cm}  f_U({u_1})du_1\ldots f_U({u_{\lfloor n/2\rfloor}})du_{\lfloor n/2\rfloor}\Bigg)f_{Y}(y_1)dy_1\ldots f_{Y}(y_{\lfloor n/2\rfloor})dy_{\lfloor n/2\rfloor}\\
&=\int\hspace*{-.1cm}\ldots\hspace*{-.1cm}\int\log\Bigg(\int \hspace*{-.1cm} \ldots \hspace*{-.1cm}\int \prod_{j=1}^{\lfloor n/2\rfloor}\Big(1+\vartheta_{n,1}u_j(y_j^2-1)+\frac{(\vartheta_{n,1} u_j)^2}{2}\big(2+y_j^4-5y_j^2\big)\\
&\hspace*{1cm}  +\mathcal{O}(\vartheta_{n,1}^3u_j^3+\vartheta_{n,1}^4u_j^4)\Big)f_U({u_1})du_1\ldots f_U({u_{\lfloor n/2\rfloor}})du_{\lfloor n/2\rfloor}\Bigg)f_{Y}(y_1)dy_1\ldots f_{Y}(y_{\lfloor n/2\rfloor})dy_{\lfloor n/2\rfloor}\\
&=\int \hspace*{-.1cm}\ldots\hspace*{-.1cm} \int\sum_{j=1}^{\lfloor n/2\rfloor}\Big(\frac{(\vartheta_{n,1})^2}{2}\big(2+y_j^4-5y_j^2\big)+\mathcal{O}(\vartheta_{n,1}^4)\Big) f_{Y_1}(y_1)dy_1\ldots f_{Y_n}(y_{\lfloor n/2\rfloor})dy_{\lfloor n/2\rfloor}\\
&=\mathcal{O}(n\,\vartheta_{n,1}^4)\,.
\end{align*}
The third order term vanishes by symmetry. Hence, if $\vartheta_{n,1}^4 = \KLEINO(n^{-1})$, then \eqref{eq_thm_lower_global_aux_theta_1} holds, which completes the proof.
\end{proof}
\noindent
Theorem \ref{lowerglobal} is deduced as a corollary of Theorem \ref{thm_lower_global_aux_theta}. 
Providing the experimenter additional information can only decrease the lower boundary on minimax distinguishability. We assume that it is known a priori that the time of change $\td = 1/2$, which puts us into the framework of Theorem \ref{thm_lower_global_aux_theta}. Set $\vartheta_{n,0} =  n^{-\aalpha}$ and $\vartheta_{n,1} = b_n^{1/2} n^{-\aalpha'}$ with $\aalpha \ge \aalpha'$. Since $\aalpha' < \aalpha$ reflects a larger fluctuation in our setup, we are only interested in the case where $\vartheta_{n,1} \geq \vartheta_{n,0}$ in our testing problem. 
If $\vartheta_{n,1} \geq \vartheta_{n,0}$, then 
$|\vartheta_{n,0}^2 - \vartheta_{n,1}^2|= \vartheta_{n,1}^2 - \vartheta_{n,0}^2  = \KLEINO\big(n^{-1/2}\big)$ leads to
\begin{align}
b_nn^{-2\aalpha'} - n^{-2\aalpha}=\KLEINO\big(n^{-1/2}\big)\,,
\end{align}
which in turns yields the condition
\begin{align}
b_n =\KLEINO\big(n^{-1/2+2\aalpha'}+n^{-2(\aalpha-\aalpha')}\big)\,.
\end{align}
The claim then follows from Theorem \ref{thm_lower_global_aux_theta} for $b_n=n^{-2(\aalpha-\aalpha')} \vee n^{-1/2 +2\aalpha'}$.\\[.2cm]

{\bf{Proof of Proposition \ref{prop:optimal:test:global}.}} We only show the claim for $\psi_{\alpha}^{\dagger}$, as optimality for a test based on Proposition \ref{corrupperglobal} follows in the same manner. We first show that 
\begin{align}\label{eq:prop:optimal:test:global:1}
{V}_n^{\dagger}\xrightarrow{\P} \infty 
\end{align}
under the alternative of Testing problem \ref{testingproblem} for any $b_n'$ satisfying \eqref{eq:prop:optimal:test:global:b_n}. Proposition \ref{propupperglobal}, Lemma \ref{propR} and the triangle inequality give
\begin{align}
{V}_n^{\dagger} \geq n^{-1/2}\Bigl|(1-\td)\sum_{i = 1}^{n \td}\E\bigl[(\Delta_i^n\varrho)^2\bigr] - \td \sum_{i = n\td + 1}^{n}\E\bigl[(\Delta_i^n\varrho)^2\bigr]\Bigr| - \mathcal{O}_{\P}(1).
\end{align} 
For large enough $n$, this is further bounded from below by
\begin{align}
{V}_n^{\dagger} \ge (1- \td)\td \bigl(  b_n' n^{1/2-2\aalpha'} - n^{1/2 -2\aalpha}\bigr) - \mathcal{O}_{\P}(1).
\end{align}
For $\td \in (0,1)$, it is now easy to see that \eqref{eq:prop:optimal:test:global:b_n} implies
\begin{align}
(1- \td) \td\bigl(  b_n' n^{1/2-2\aalpha'} - n^{1/2 -2\aalpha}\bigr) \to \infty, \quad \text{as $n \to \infty$,}
\end{align}
hence \eqref{eq:prop:optimal:test:global:1} follows. Next, we establish optimality of the test $\psi^{\dagger}_{\alpha}$. By Theorem \ref{thm:bootstrap:global} it is an asymptotic level $\alpha$-test. It is left to show consistency when \eqref{eq:prop:optimal:test:global:b_n} is valid.  From \eqref{eq:quantiles:converge} and $V^{\dagger} = \mathcal{O}_{\P}(1)$, we conclude that 
\begin{align}
\hat{q}_{1-\alpha}(\hat{V}_n^{\dagger}|\F) = \mathcal{O}_{\P}(1) = \KLEINO_{\P}\bigl((1- \td) \td\bigl(  b_n' n^{1/2-2\aalpha'} - n^{1/2 -2\aalpha}\bigr)\bigr)
\end{align}
for any fixed, finite $0 < \alpha < 1$. Thereby, $\P(V_n^{\dagger}>\hat q_{1-\alpha}(\hat V_n^{\dagger}|\mathcal{F}))\rightarrow 1$, which completes the proof. 
\bibliographystyle{chicago}
\bibliography{literatur}

\end{document}